\documentclass[english,oneside,12pt]{amsart}
\usepackage[utf8]{inputenc}
\usepackage[english]{babel}
\usepackage{amstext,amsthm,amsmath,amssymb}
\usepackage{smfthm}

\usepackage{mathpazo} 
\usepackage{MnSymbol}
\usepackage{wasysym}
\usepackage{marvosym}

\usepackage[babel=true]{csquotes}

\usepackage{geometry}
\usepackage[usenames,dvipsnames,svgnames,table]{xcolor}
\usepackage[colorlinks=true,urlcolor=DarkBlue,linkcolor=DarkRed,citecolor=DarkGreen,pagebackref=true]{hyperref}

\usepackage{tikz}
\usetikzlibrary{shapes,arrows}
\usetikzlibrary{fit,positioning}
\usetikzlibrary{patterns}
\usepackage{xkeyval}
\usepackage{moreverb}
\usepackage{epic}

\usepackage{pgfcore}
\usepgfmodule{shapes,plot,decorations}

\usepackage{perpage} 
\MakePerPage{footnote} 

\newcommand{\N}{\mathbb{N}}
\newcommand{\Z}{\mathbb{Z}}

\newcommand{\R}{\mathbb{R}}
\newcommand{\C}{\mathcal{C}} 

\newcommand{\A}{\mathbb{A}}

\renewcommand{\ss}{\mathrm{SL}_{d+1}(\mathbb{R})}

\newcommand{\SO}{\mathrm{SO}_{d,1}^{\circ}(\mathbb{R})}

\newcommand*{\so}[1]{\mathrm{SO}_{#1,1}^{\circ}(\mathbb{R})}
\newcommand*{\s}[1]{\mathrm{SL}_{#1}(\mathbb{R})}
\newcommand*{\spm}[1]{\mathrm{SL}^{\pm}_{#1}(\mathbb{R})}

\newcommand{\LG}{\Lambda_{\Gamma}}
\newcommand{\G}{\Gamma}
\newcommand{\g}{\gamma}
\newcommand{\LGP}{\Lambda_{P}}

\renewcommand{\C}{\mathcal{C}}
\newcommand{\U}{\mathcal{U}}
\newcommand{\F}{\mathcal{F}}

\newcommand{\V}{\mathcal{V}}
\renewcommand{\AA}{\mathcal{A}}

\newcommand{\E}{\mathcal{E}}
\newcommand{\GG}{\mathcal{G}}
\newcommand{\D}{\mathcal{D}}

\renewcommand{\O}{\Omega}

\newcommand{\dO}{\partial \Omega}

\renewcommand{\d}{d_{\Omega}}

\renewcommand{\S}{\mathbb{S}}
\renewcommand{\P}{\mathcal{P}}
\renewcommand{\A}{\mathbb{A}}

\newcommand{\PP}{\mathbb{P}}

\newcommand{\Quo}{\Omega/\!\raisebox{-.90ex}{\ensuremath{\Gamma}}}
\newcommand*{\Quotient}[2]{\ensuremath{#1/\!\raisebox{-.90ex}{\ensuremath{#2}}}}

\newcommand{\Aut}{\textrm{Aut}}

\newcommand{\sm}{\smallsetminus}

\theoremstyle{plain}
\newtheorem{theorem}{Theorem}[section]

\newtheorem{propo}[theorem]{Proposition}
\newtheorem{cor}[theorem]{Corollary}
\newtheorem{lemma}[theorem]{Lemma}

\theoremstyle{definition}

\newtheorem{de}[theorem]{Definition}

\theoremstyle{remark}
\newtheorem{rem}[theorem]{Remark}
\newtheorem{nota}[theorem]{Notation}

\newtheorem{theooo}{Theorem}[section]

\subjclass{22E40, 20F55, 20F67, 53C60}
\title{Coxeter group in Hilbert geometry}

\author{
Ludovic Marquis
}
\email{ludovic.marquis@univ-rennes1.fr}
\address{IRMAR, University of Rennes, France}

\begin{document}

\newcommand\point{\textbullet}
\newcommand*\points[1]{%
 \ifcase\value{#1}\or
   $\cdot$ \or $\udotdot$
    \or $\therefore$ \or
  $\diamonddots$ 
   \or  $\fivedots$
   \or $\davidsstar$
   \or 7
 \fi}
\renewcommand\theenumi{%
 \points{enumi}}
\renewcommand\labelenumi{%
 \points{enumi}}

\begin{abstract}
A theorem of Tits - Vinberg allows to build an action of a Coxeter group $\Gamma$ on a properly convex open set $\Omega$ of the real projective space, thanks to the data $P$ of a polytope and reflection across its facets. We give sufficient conditions for such action to be of finite covolume, convex-cocompact or geometrically finite. We describe a hypothesis that makes those conditions necessary.

Under this hypothesis, we describe the Zariski closure of $\Gamma$, find the maximal $\Gamma$-invariant convex set, when there is a unique $\Gamma$-invariant convex set, when the convex set $\Omega$ is strictly convex, when we can find a $\Gamma$-invariant convex set $\Omega'$ which is strictly convex.
\end{abstract}


\maketitle
\tableofcontents

\section*{Introduction}
\subsubsection*{General Framework}

\par{
The study of groups acting on Hilbert geometry or convex projective structures on manifold starts with the pioneering work of Kuiper \cite{MR0063115} in the 50's. After came the works of Benzécri \cite{MR0124005}, Vinberg \cite{MR0158414,MR0201575,MR0302779}, Kac-Vinberg \cite{MR0208470}, Koszul \cite{MR0239529} and Vey \cite{MR0283720} in the 60's. Then the field took a deep breath, and came back in the 90's with Goldman \cite{MR1053346}, followed by Suhyoung Choi, Labourie, Loftin, Inkang Kim and a long series of articles by Benoist in the 2000's. The recent works of Jaejeong Lee, Misha Kapovich, Cooper, Long, Tillmann, Thistlethwaite, Ballas, Gye-Seon Lee, Suhyoung Choi, Nie, Crampon and the author show a growing interest for this field.
}
\\
\par{
We want to study the action of discrete groups $\G$ of $\spm{d+1}$ on a properly\footnote{A bounded convex subset of an affine chart.} convex open set $\O$ of the projective sphere $\S^d=\S(\R^{d+1})=\{ \textrm{Half-line of } \R^{d+1} \}$. Note that on every properly convex open set $\O$ of $\S^d$ there is a distance $d_{\O}$ and a measure $\mu_{\O}$ invariant by the group $\Aut(\O)=\{ \g \in \spm{d+1} \,|\, \g(\O) = \O \}$ of automorphisms of $\O$.
}
\\
\par{
At this moment, \emph{divisible convex sets}, the convex sets $\O$ for which there exists a discrete subgroup $\G$ of $\Aut(\O)$ such that $\Quo$ is compact, have received almost all the attention. The \emph{quasi-divisible convex sets}, the one for which there exist a discrete subgroup $\G$ of $\Aut(\O)$ such that $\Quo$ is of finite volume, are starting to be studied, see \cite{Cooper:2011fk,Ballas:2012uq,Ballas:2014fk,Crampon:2012fk,Marquis:2009kq,Marquis:2010fk}.
}
\\
\par{
There is at least four ways to say that the action of $\G$ on $\O$ is ``cofinite''. The first two ways are the following: the action of $\G$ on $\O$ is \emph{cocompact} (resp. of \emph{cofinite volume}) when the quotient orbifold $\Quo$ is compact (resp. of finite volume for the measure induced by $\mu_{\O}$).
}
\\
\par{
If we assume moreover that the action of $\G$ on $\R^{d+1}$ is strongly irreducible\footnote{The action of any finite index subgroup of $\G$ on $\R^{d+1}$ is irreducible.}, Benoist shows in \cite{MR1767272} that there exists a smallest closed $\G$-invariant subset $\LG$ of the real projective space $\PP(\R^{d+1}) =\PP^d(\R) = \PP^d$. We still denote $\LG$ the one of the two preimages of $\LG$ in $\S^d$ which is included in $\dO$. We denote by $\overline{C}(\LG)$ the convex hull\footnote{The smallest closed convex subset of $\O$ containing $\LG$ in its closure in $\S^d$.} of $\LG$ in $\O$. We remark that $\overline{C}(\LG)$ is a closed subset of $\O$ which has a non empty interior since the action of $\G$ on $\R^{d+1}$ is strongly irreducible.
}
\\
\par{
We will say that the action of $\G$ on $\O$ is \emph{convex-cocompact} (resp. \emph{geometrically finite}) when the quotient $\Quotient{\overline{C}(\LG)}{\G}$ is compact (resp. of finite volume for the measure induced by $\mu_{\O}$).
}
\\
\par{
The definition of cocompact, finite volume or convex-cocompact action make no doubt, but the definition of geometrical finiteness deserves a detailed comment that will be done in Section \ref{best_why}.
}
\\
\par{
The theory of Coxeter groups has two benefits for us. First, it gives a simple and explicit recipe to build a lot of groups with different behaviours from the point of view of geometric group theory. Second, the Theorem of Tits-Vinberg gives the hope to build a lot of interesting actions of Coxeter groups on Hilbert geometry. So, we will focus on actions of Coxeter groups $W$ on convex subsets of the projective sphere $\S^d$.
}
\\
\par{
We point out, for the reader not familiar with Hilbert geometry, that Hilbert geometries can also be very different. For example, if $\O$ is the round ball of an affine chart $\R^d$ of $\S^d$ then $(\O,d_{\O})$ is isometric to the real hyperbolic space of dimension $d$ and if $\O$ is a triangle then $(\O,d_{\O})$ is bi-Lipschitz equivalent the euclidean plane. In particular, our discussion includes the context of hyperbolic geometry.
}
\medskip
\subsubsection*{Precise Framework}
\par{
In order to make a Coxeter group acts on the projective sphere, one can take a projective polytope $P$ of $\S^d$, and choose a projective reflection $\sigma_s$ across each facet $s$ of $P$\footnote{Note that in projective geometry there are many reflections across a given hyperplane.}. We want to consider the subgroup $\G=\G_P$ of $\spm{d+1}$ generated by the reflections $(\sigma_s)_{s \in S}$, where $S$ is the set of facets of $P$. In order to get a discrete subgroup of $\spm{d+1}$, we need some hypothesis on the set of reflections $(\sigma_s)_{s \in S}$. Roughly speaking, the hypothesis will be that if $s$ and $t$ are two facets of $P$ such that $s\cap t$ is of codimension 2 then the product $\sigma_s \sigma_t$ is conjugate to a rotation of angle $\frac{\pi}{m}$, where $m$ is an integer. We also authorize the case $m=\infty$, and for this case a special condition is needed.
}
\\
\par{
The precise definition is Definition \ref{def_cox_poly}. Such a polytope will be called a \emph{Coxeter polytope}. Given a Coxeter polytope $P$, we can consider the set $\C=\C_P=\bigcup_{\g \in \G} \g(P)$. The Theorem of Tits-Vinberg (Theorem \ref{theo_vinberg}) tells us that $\G$ is discrete and $\C$ is a convex subset of $\S^d$. This theorem provides a huge amount of examples with drastically different behaviours.
}
\\
\par{
The goal of this article is to tackle the following questions: Let $P$ be a Coxeter polytope of $\S^d$, let $\G$ be the discrete subgroup generated by the reflections $(\sigma_s)_{s \in S}$ and $\C=\bigcup_{\g \in \G} \g(P)$. In order to get nice irreducible examples, we assume that the action of $\G$ on $\R^{d+1}$ is strongly irreducible, so $\C$ has to be properly convex. Let $\O$ be the interior of $\C$. When is the action of $\G$ on $\O$ cocompact ?\footnote{Already answer by Vinberg, see Theorem 2 of \cite{MR0302779} or corollary \ref{cocompact}.}of finite covolume ? convex cocompact ? geometrically finite ?
}
\\
\par{
Our goal is also to answer questions about the Zariski closure of $\G$, about the convex set $\O$ and about the other possible convex set preserved by $\G$. Precisely, we mean :
}

\begin{enumerate}
\item What are the possible Zariski closures for $\G$ ?
\item Is the convex set $\O$ the largest properly convex open set preserved ?
\item[$\medtriangleup$] When does the action of $\G$ on $\S^d$ preserve a unique properly convex open subset$\,$?
\item[$\meddiamond$] When is the convex set $\O$ the smallest properly convex open set preserved ?
\item[$\medstar$] When is the convex set $\O$ strictly convex ? with $\C^1$ boundary ? both ?
\item[$\davidsstar$] When does the action of $\G$ on $\S^d$ preserve a strictly convex open set ? a properly convex open set with $\C^1$ boundary ? a strictly convex open set with $\C^1$ boundary$\,$?
\end{enumerate}

\vspace{5pt}
\par{
If we don't make any hypothesis on the Coxeter polytope $P$, the behaviour of the action $\G \curvearrowright \O$ can be very complicated. So, we will make a non-trivial hypothesis along this text. I think this hypothesis is relevant and offers an access to a wide family of examples. For example, this hypothesis is satisfied by every Coxeter polygon and every Coxeter polyhedron whose dihedral angles are non-zero.
}
\\
\par{
Now, we briefly explain the hypothesis that we will make most of the time along this text. A nice way to get information about a polytope is to look around a vertex. The link of a Coxeter  polytope $P$ at a vertex $p$ is a Coxeter polytope $P_p$ of one dimension less than $P$ and which is ``$P$ seen from $p$''. In the context of hyperbolic geometry, it is the intersection of $P$ with a small sphere centered at $p$.
}
\\
\par{
Vinberg introduces the following terminology in \cite{MR0302779}: A Coxeter  polytope $P$ is \emph{perfect} when the action of $\G$ on $\O$ is cocompact. We will mainly assume that $P$ is \emph{$2$-perfect}, which means that the link of every vertex of $P$ is perfect or equivalently that $P \cap \dO$ is contained in the set of vertices of $P$. See Proposition \ref{def_2-perfect} for precisions. 
}
\\
\par{
Vinberg shows in \cite{MR0302779} that perfect Coxeter polytopes come from three different families\footnote{See definition \ref{def_3types}.}:
}
\begin{itemize}
\item $P$ is \emph{elliptic}, i.e. $\G$ is finite.
\item $P$ is \emph{parabolic}, i.e. $\O$ is an affine chart.
\item Otherwise, $\O$ is properly convex. In that case, we say that $P$ is \emph{loxodromic}.
\end{itemize}
In particular, if $P$ is 2-perfect, then the link at any vertex is either elliptic, parabolic or loxodromic. We can now state our results.

\begin{theooo}[(Theorems \ref{Theo_geo_fini}, \ref{Theo_vol_fini} and \ref{Theo_conc_comp})]\label{theo_principal}
Let $P$ be a 2-perfect Coxeter  polytope. Let $\G=\G_P$ be the subgroup of $\ss$ generated by the reflections around the facets of $P$. Let $\O=\O_P$ be the interior of the $\G$-orbit of $P$. Suppose that the action of $\G$ on $\R^{d+1}$ is strongly irreducible. Then:
\begin{itemize}
\item The action $\G \curvearrowright \O$ is geometrically finite.
\item Moreover, the action $\G \curvearrowright \O$ is of finite covolume if and only if the link $P_p$ of every vertex $p$ of $P$ is elliptic or parabolic.
\item Finally, the action $\G \curvearrowright \O$ is convex cocompact if and only if the link $P_p$ of every vertex $p$ of $P$ is elliptic or loxodromic.
\end{itemize}
\end{theooo}

We keep the same notation and hypothesis for the following theorems.

\begin{theooo}[(Theorem \ref{adh_zar_2-perf})]\label{theo_zari_intro}
The Zariski closure of $\G$ is either conjugate to $\SO$ or is equal to $\ss$.
\end{theooo}

\begin{theooo}[(Theorem \ref{Theo_maxi})]\label{maxi_intro}
Every properly convex open set preserved by $\G$ is included in $\O$.
\end{theooo}

\begin{theooo}[(Theorem \ref{Theo_mini})]\label{mini_intro}
The convex set $\O$ is the smallest properly convex open set preserved by $\G$ if and only if the action $\G \curvearrowright \O$ is of finite covolume.
\end{theooo}

By using one of the results of \cite{MR2094116,Cooper:2011fk}, we can also show:

\begin{theooo}[(Theorem \ref{omega_strict})]\label{strict_intro}
The following are equivalent:
\begin{enumerate}
\item The properly convex open set $\O$ is strictly convex. 
\item The boundary $\dO$ of $\O$ is of class $\C^1$.
\item The action $\G \curvearrowright \O$ is of finite covolume and the group $\G$ is relatively hyperbolic relatively to the links $P_p$ for which $P_p$ is parabolic.
\end{enumerate}
In that case, the metric space $(\O,d_{\O})$ is Gromov-hyperbolic.
\end{theooo}

Thanks to the moduli space computed in \cite{MR2660566}, we will easily get the following theorem as a corollary of Theorem \ref{strict_intro}. 

\begin{theooo}\label{existenceDim3_intro}
In dimension $3$, there exists an indecomposable\footnote{A convex set that is not the join of two convex sets of smaller dimension.} quasi-divisible properly convex open set which is not divisible nor strictly convex.
\end{theooo}

We recall that one cannot find such an example in dimension $2$, thanks to \cite{MR0124005,Marquis:2009kq}. A construction in any dimension is an open question in the divisible or the quasi-divisible context.

\begin{theooo}[(Theorem \ref{exist_strict})]\label{exis_strict_intro}
If moreover all the loxodromic vertices are simple\footnote{A vertex is \emph{simple} when its link is a simplex.}, the following are equivalent:
\begin{enumerate}
\item There exists a strictly convex open set $\O'$ preserved by $\G$.
\item There exists a properly convex open set $\O'$ with $\C^1$-boundary preserved by $\G$.
\item The group $\G$ is relatively hyperbolic relatively to the links $P_p$ for which $P_p$ is parabolic.
\end{enumerate}
\end{theooo}

\par{
Along the way, we will study a nice procedure : truncation which allows to build a new polytope from a starting one by cutting a simple vertex (See Subsection \ref{label_truncation}). This procedure is present in a survey of Vinberg \cite{MR783604} in the context of hyperbolic geometry, it has also been used by the author in \cite{MR2660566}, this time in the context of projective geometry. The approach in this text will be less computational and more geometrical than in \cite{MR2660566}. We think this procedure is interesting in its own right. Moreover, the introduction of this procedure gives nicer statements of the previously quoted theorems.
}

\subsubsection*{Others works around the subject}

\par{
The starting point and main inspiration for this article, is the article \cite{MR0302779} of Vinberg, which presents the notion of Coxeter  polytope\footnote{Note that Vinberg prefers to work with $\G$ than with $P$. Vinberg called such a $\G$ a \emph{linear Coxeter group}.} and studies the first property. Cocompact actions are studied in Vinberg's text but actions of cofinite volume, convex-cocompact or geometrically finite action are not. There are also lecture notes by Benoist \cite{fivelectures} which present a proof of the theorem of Tits-Vinberg. The examples of the article \cite{MR2218481,MR2295544} of Benoist are built thanks to the theorem of Tits-Vinberg.
}
\\
\par{
One can also study the moduli space of a Coxeter polytope but we will not do it in this text. Suhyoung Choi with Gye-Seon Lee, Craig Hodgson and the author have works on this problem \cite{MR2247648,Choi:2010ys,Choi:2012kx,MR2660566}. We will devote several forthcoming articles with Suhyoung Choi, Gye-Seon Lee and/or Ryan Greene to the problem of moduli space.
}
\\
\par{
The study of geometrically finite actions was started in \cite{Crampon:2012fk} of M. Crampon and the author. We stress that in the last article the authors made the hypothesis that the convex set $\O$ on which the group $\G$ acts is strictly convex with $\C^1$-boundary. The study of actions of cofinite volume is the main purpose of the articles \cite{Cooper:2011fk} of Cooper, Long and Tillmann, \cite{Marquis:2009kq} and \cite{Marquis:2010fk} of the author. We stress that the hypothesis of strict convexity of $\O$ is central in \cite{Crampon:2012fk,Marquis:2010fk}. This hypothesis is absent from \cite{Marquis:2009kq} and is not always present in \cite{Cooper:2011fk}. There is also a paper of Suhyoung Choi about geometrically finite actions \cite{Choi:2010fk}.
}
\\
\par{
We point out that in this text, we did not make any assumption about the regularity of the boundary of $\O$. One of the goals is actually to build examples where $\O$ is not strictly convex and the action is cocompact or of finite covolume. 
}

\subsubsection*{Plan of the article}

\par{
The first part of the article contains preliminaries about convexity, Hilbert geometry, Coxeter groups and Coxeter polytopes. The second part is a recalling of the theorem of Tits-Vinberg and of important results of Vinberg coming from the article \cite{MR0302779}. The third part presents the definition of link of a polytope and makes precise the hypothesis : ``$P$ is $2$-perfect''.
}
\\
\par{
The fourth part presents the lemmas for the study of the geometry around a vertex. The fifth part is a classification of degenerate 2-perfect polytopes. The sixth part is devoted to the proof of Theorem \ref{theo_principal}. The seventh part studies the Zariski closure of $\G$, and it contains the proof of Theorem \ref{theo_zari_intro}. The eighth part contains the proof of Theorems \ref{maxi_intro}, \ref{mini_intro}, \ref{strict_intro} and \ref{exis_strict_intro}.
}

\subsubsection*{Acknowledgements}

\par{
The author thanks Yves Benoist for a couple of dense discussion about this text. We thanks \'Ernest Vinberg which is a major source of inspiration for this article. Finally, we warmly thank Gye-Seon Lee who found a lot of errors in a previous version.
}
\par{
The author thanks the ANR facets of discrete groups and ANR Finsler geometry for their supports.
}
\section{Preliminaries}
\subsection{Convexity in the projective sphere}$\,$

\par{
Let $V$ be a real vector space. A convex cone $\C$ is \emph{sharp} when $\C$ does not contain any affine line. Consider the \emph{projective sphere} $\S(V)= \{ \textrm{Half-lines of } V\} = \Quotient{V \smallsetminus \{ 0\}}{\sim}$ where $\sim$ is the equivalence relation induced by the action of $\R^*_+$ by homothety on $V$. Of course, $\S(V)$ is the 2-fold cover of the real projective space $\PP(V)$. The notion of convexity is nicer in $\S(V)$ than in $\PP(V)$. We will denote $\S : V  \smallsetminus \{ 0\} \rightarrow \S(V)$ the natural projection.
}
\par{
A subset $C$ of $\S(V)$ is \emph{convex} (resp. \emph{properly convex}) when the set $\S^{-1}(C)$ is a convex cone (resp. sharp convex cone) of $V$. Given a hyperplane $H$ of $\S(V)$, the two connected components of $\S(V) \smallsetminus H$ are called \emph{affine charts}. An open set $\O \neq \S(V)$ of $\S(V)$ is convex (resp. properly convex) if and only if there exists an affine chart $\A$ such that $\O \subset \A$ (resp. $\overline{\O} \subset \A$) and $\O$ is convex in the usual sense in $\A$.
}

\subsection{Hilbert geometry}$\,$

\par{
On every properly convex open set $\O$ of $\S^d$ there is a distance $d_{\O}$ defined thanks to the cross-ratio, in the following way: take any two points $x \neq y \in \O$ and draw the line between them. This line intersects the boundary $\dO$ of $\O$ in two points $p$ and $q$. We assume that $x$ is between $p$ and $y$. Then the following formula defines a distance (see Figure \ref{disttt}):
}
$$d_{\O}(x,y) =  \displaystyle \frac{1}{2}\ln \Big( [p:x:y:q] \Big)$$
\vspace*{.25em}

This distance gives to $\O$ the same topology than the one inherited from $\S(V)$. The metric space $(\O,d_{\O})$ is complete, the closed ball are compact, the group $\Aut(\O)$ acts by isometries on $\O$, and therefore acts properly.

\begin{figure}[h!]
\centering
\begin{tikzpicture}
\filldraw[draw=black,fill=gray!20]
 plot[smooth,samples=200,domain=0:pi] ({4*cos(\x r)*sin(\x r)},{-4*sin(\x r)});
cycle;
\draw (-1,-2.5) node[anchor=north west] {$x$};
\fill [color=black] (-1,-2.5) circle (2.5pt);
\draw (1,-2) node[anchor=north west] {$y$};
\fill [color=black] (1,-2) circle (2.5pt);
\draw [smooth,samples=200,domain=-4:4] plot ({\x},{0.25*\x-2.25});
\draw (-1.97,-2.75) node[anchor=north west] {$p$};
\fill [color=black] (-1.98,-2.75) circle (2.5pt);
\draw (1.7,-1.8) node[anchor=north west] {$q$};
\fill [color=black] (1.63,-1.85) circle (2.5pt);
\draw [smooth,samples=200,domain=-3:2] plot ({\x},{1.2*\x-1.3});
\draw[->,distance=10pt,thick] (-1,-2.5) -- (-0.2,-1.54);
\draw (-1,-1.6) node[anchor=north west] {$v$};
\fill [color=black] (-1.75,-3.42) circle (2.5pt);
\draw (-2,-3.6) node[anchor=north west] {$p^-$};
\fill [color=black] (.6,-.6) circle (2.5pt);
\draw (.8,-.3) node[anchor=north west] {$p^+$};
\draw (0,-3.7) node {$\O$};
\end{tikzpicture}
\caption{Hilbert distance}
\label{disttt}
\end{figure}
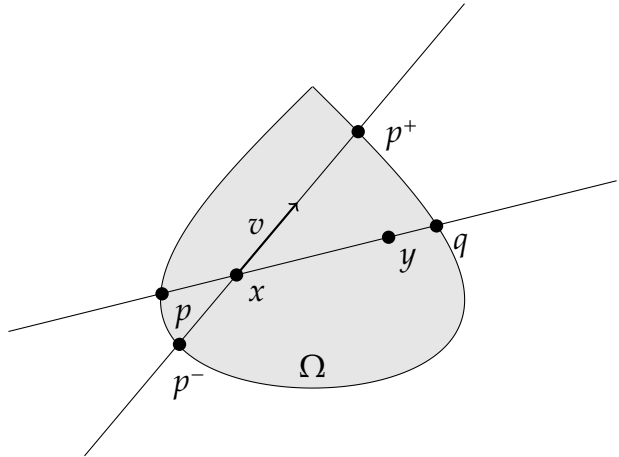

\par{
This distance is called the \emph{Hilbert distance} and has the good taste to came from a Finsler metric on $\O$ defined by a very simple formula. Let $x$ be a point of $\O$ and $v$ a vector of the tangent space $T_x \O$ of $\O$ at $x$. The quantity  $\left. \frac{d}{dt}\right| _{t=0} d_{\O}(x,x+tv)$ defines a Finsler metric $F_{\O}(x,v)$ on $\O$. Moreover, if we choose an affine chart $\A$ containing $\O$ and a euclidean norm $|\cdot|$ on $\A$, we get:
}
$$F_{\O}(x,v) = \left. \frac{d}{dt}\right| _{t=0} d_{\O}(x,x+tv) = \frac{|v|}{2}\Bigg(\frac{1}{|xp^-|} + \frac{1}{| xp^+|} \Bigg)$$
\par{
Where $p^-$ and $p^+$ are the intersection points of the half-line starting at $p$ with direction $-v$ and $v$ with $\dO$ and $|ab|$ is the distance between points $a,b$ of $\A$ for the euclidean norm $|\cdot|$ (see Figure \ref{disttt}). The regularity of this Finsler metric is the regularity of the boundary $\dO$ of $\O$, and the Finsler structure gives rise to an absolutely continuous measure $\mu_{\O}$ with respect to Lebesgue measure. We will not need any explicit formula for this measure; we will only use the following proposition which is straightforward and explained in \cite{MR2270228}:
}
 
\begin{propo}\label{compa}
Let $\O_1 \subset \O_2$ be two properly convex open sets; then for any Borel set $\AA$ of $\O_1$, we have $\mu_{\O_2}(\AA) \leqslant \mu_{\O_1}(\AA)$.
\end{propo}

\subsection{Coxeter Group}$\,$

\par{
Coxeter group are going to be the main object of this paper, so we take the time to recall some basic facts.
}

\begin{de}
A \emph{Coxeter system} is the data of a finite set $S$ and a symmetric matrix $M=(M_{st})_{s,t \in S}$ such that the diagonal coefficients $M_{ss}=1$ and the other coefficients $M_{st} \in \{2,3,..., n, ...,\infty \}$. The cardinality of $S$ is called the \emph{rank} of the Coxeter system $(S,M)$. With a Coxeter system, one can build a \emph{Coxeter group} $W_S$; it is a group defined by generators and relations. The generators are the elements of $S$ and we impose the relations $(st)^{M_{st}}=1$ for all $s,t \in S$ such that $M_{st} \neq \infty$.
\end{de}

\par{
There are two basic objects associated to a Coxeter system or a Coxeter group: its Coxeter diagram and its Gram matrix. We recall the definition of these two objects and the basic consequences.
}
\par{
One can associate to $W$ a \textit{labelled} graph, also denoted by $W$, called the \emph{Coxeter diagram} of $W$. The vertices of $W$ are the elements of $S$. Two vertices $s,t \in S$ are linked by an edge if and only if $M_{st} \neq 2$. The label of an edge linking $s$ to $t$ in $W$ is the number $M_{st} > 2$. A Coxeter group is \emph{irreducible} when its Coxeter graph is connected. Of course, any Coxeter group is the direct product of the Coxeter groups associated to the connected components of its Coxeter graph.
}
\par{
One also associate to $W$ a symmetric matrix of size the cardinality of $S$, namely its \emph{Gram matrix} $Cos(W)$, defined by the following formula: $(Cos(W))_{st} = -2\cos \Big( \frac{\pi}{M_{st}} \Big)$ for $s,t \in S$.
}
\par{
An irreducible Coxeter group $W$ is a \emph{spherical Coxeter group} (resp. \emph{affine Coxeter group}) if its Gram Matrix is positive definite (resp. positive but not definite). Vinberg and Margulis showed that an irreducible Coxeter group which is not spherical nor affine is large\footnote{admits a finite index subgroup which admits an onto morphism on a non-abelian free group.} in \cite{MR1748082}. Therefore an irreducible Coxeter group is either spherical, affine or large.
}
\par{
More generally a Coxeter group is \emph{spherical} (resp. \emph{affine} resp. \emph{euclidean}) when all its connected components are irreducible spherical Coxeter group  (resp. affine resp. affine or spherical).
}
\par{
The irreducible spherical and affine Coxeter groups have been classified (by Coxeter in \cite{MR1503182} for the spherical case). We reproduce the list of those Coxeter diagrams in Figures \ref{spheri_diag} and \ref{affi_diag}. We use the usual convention that an edge that should be labelled $3$ has in fact no label. We stress that among them the only ones which are not trees or have an edge labelled $\infty$ are the affine Coxeter diagram named $\tilde{A}_n$ for $n \geqslant 1$. As already remarked by Vinberg, those Coxeter groups play a special role in this context.
}

\begin{figure}[!ht]
\centering
\begin{minipage}[b]{7.5cm}
\centering
\begin{tikzpicture}[thick,scale=0.6, every node/.style={transform shape}]
\node[draw,circle] (A1) at (0,0) {};
\node[draw,circle,right=.8cm of A1] (A2) {};
\node[draw,circle,right=.8cm of A2] (A3) {};
\node[draw,circle,right=1cm of A3] (A4) {};
\node[draw,circle,right=.8cm of A4] (A5) {};
\node[left=.8cm of A1] {$A_n$};

\draw (A1) -- (A2)  node[above,midway] {};
\draw (A2) -- (A3)  node[above,midway] {};
\draw[loosely dotted,thick] (A3) -- (A4) node[] {};
\draw (A4) -- (A5) node[above,midway] {};
 
 
\node[draw,circle,below=1.2cm of A1] (B1) {};
\node[draw,circle,right=.8cm of B1] (B2) {};
\node[draw,circle,right=.8cm of B2] (B3) {}; 
\node[draw,circle,right=1cm of B3] (B4) {};
\node[draw,circle,right=.8cm of B4] (B5) {};
\node[left=.8cm of B1] {$B_n$};

\draw (B1) -- (B2)  node[above,midway] {$4$};
\draw (B2) -- (B3)  node[above,midway] {};
\draw[loosely dotted,thick] (B3) -- (B4) node[] {};
\draw (B4) -- (B5) node[above,midway] {};
  
 
\node[draw,circle,below=1.5cm of B1] (D1) {};
\node[draw,circle,right=.8cm of D1] (D2) {};
\node[draw,circle,right=1cm of D2] (D3) {}; 
\node[draw,circle,right=.8cm of D3] (D4) {};
\node[draw,circle, above right=.8cm of D4] (D5) {};
\node[draw,circle,below right=.8cm of D4] (D6) {};
\node[left=.8cm of D1] {$D_n$};

\draw (D1) -- (D2)  node[above,midway] {};
\draw[loosely dotted] (D2) -- (D3);
\draw (D3) -- (D4) node[above,midway] {};
\draw (D4) -- (D5) node[above,midway] {};
\draw (D4) -- (D6) node[below,midway] {};
  
 
\node[draw,circle,below=1.2cm of D1] (I1) {};
\node[draw,circle,right=.8cm of I1] (I2) {};
\node[left=.8cm of I1] {$I_2(p)$};

\draw (I1) -- (I2)  node[above,midway] {$p$};
 
 
\node[draw,circle,below=1.2cm of I1] (H1) {};
\node[draw,circle,right=.8cm of H1] (H2) {};
\node[draw,circle,right=.8cm of H2] (H3) {}; 
\node[left=.8cm of H1] {$H_3$};

\draw (H1) -- (H2)  node[above,midway] {$5$};
\draw (H2) -- (H3)  node[above,midway] {};
 
 
\node[draw,circle,below=1.2cm of H1] (HH1) {};
\node[draw,circle,right=.8cm of HH1] (HH2) {};
\node[draw,circle,right=.8cm of HH2] (HH3) {}; 
\node[draw,circle,right=.8cm of HH3] (HH4) {};
\node[left=.8cm of HH1] {$H_4$};

\draw (HH1) -- (HH2)  node[above,midway] {$5$};
\draw (HH2) -- (HH3)  node[above,midway] {};
\draw (HH3) -- (HH4)  node[above,midway] {};
  
 
\node[draw,circle,below=1.2cm of HH1] (F1) {};
\node[draw,circle,right=.8cm of F1] (F2) {};
\node[draw,circle,right=.8cm of F2] (F3) {}; 
\node[draw,circle,right=.8cm of F3] (F4) {};
\node[left=.8cm of F1] {$F_4$};

\draw (F1) -- (F2)  node[above,midway] {};
\draw (F2) -- (F3)  node[above,midway] {$4$};
\draw (F3) -- (F4)  node[above,midway] {};
 
 
\node[draw,circle,below=1.2cm of F1] (E1) {};
\node[draw,circle,right=.8cm of E1] (E2) {};
\node[draw,circle,right=.8cm of E2] (E3) {}; 
\node[draw,circle,right=.8cm of E3] (E4) {};
\node[draw,circle,right=.8cm of E4] (E5) {};
\node[draw,circle,below=.8cm of E3] (EA) {};
\node[left=.8cm of E1] {$E_6$};

\draw (E1) -- (E2)  node[above,midway] {};
\draw (E2) -- (E3)  node[above,midway] {};
\draw (E3) -- (E4)  node[above,midway] {};
\draw (E4) -- (E5)  node[above,midway] {};
\draw (E3) -- (EA)  node[left,midway] {};
 
 
\node[draw,circle,below=1.8cm of E1] (EE1) {};
\node[draw,circle,right=.8cm of EE1] (EE2) {};
\node[draw,circle,right=.8cm of EE2] (EE3) {}; 
\node[draw,circle,right=.8cm of EE3] (EE4) {};
\node[draw,circle,right=.8cm of EE4] (EE5) {};
\node[draw,circle,right=.8cm of EE5] (EE6) {};
\node[draw,circle,below=.8cm of EE3] (EEA) {};
\node[left=.8cm of EE1] {$E_7$};

\draw (EE1) -- (EE2)  node[above,midway] {};
\draw (EE2) -- (EE3)  node[above,midway] {};
\draw (EE3) -- (EE4)  node[above,midway] {};
\draw (EE4) -- (EE5)  node[above,midway] {};
\draw (EE5) -- (EE6)  node[above,midway] {};
\draw (EE3) -- (EEA)  node[left,midway] {};
 
 
\node[draw,circle,below=1.8cm of EE1] (EEE1) {};
\node[draw,circle,right=.8cm of EEE1] (EEE2) {};
\node[draw,circle,right=.8cm of EEE2] (EEE3) {}; 
\node[draw,circle,right=.8cm of EEE3] (EEE4) {};
\node[draw,circle,right=.8cm of EEE4] (EEE5) {};
\node[draw,circle,right=.8cm of EEE5] (EEE6) {};
\node[draw,circle,right=.8cm of EEE6] (EEE7) {};
\node[draw,circle,below=.8cm of EEE3] (EEEA) {};
\node[left=.8cm of EEE1] {$E_8$};

\draw (EEE1) -- (EEE2)  node[above,midway] {};
\draw (EEE2) -- (EEE3)  node[above,midway] {};
\draw (EEE3) -- (EEE4)  node[above,midway] {};
\draw (EEE4) -- (EEE5)  node[above,midway] {};
\draw (EEE5) -- (EEE6)  node[above,midway] {};
\draw (EEE6) -- (EEE7)  node[above,midway] {};
\draw (EEE3) -- (EEEA)  node[left,midway] {};
 
\end{tikzpicture}
\caption{Irreducible spherical diagram}
\label{spheri_diag}
\end{minipage}
\begin{minipage}[t]{7.5cm}
\centering
\begin{tikzpicture}[thick,scale=0.6, every node/.style={transform shape}]
\node[draw,circle] (A1) at (0,0) {};
\node[draw,circle,above right=.8cm of A1] (A2) {};
\node[draw,circle,right=.8cm of A2] (A3) {};
\node[draw,circle,right=.8cm of A3] (A4) {};
\node[draw,circle,right=.8cm of A4] (A5) {};
\node[draw,circle,below right=.8cm of A5] (A6) {};
\node[draw,circle,below left=.8cm of A6] (A7) {};
\node[draw,circle,left=.8cm of A7] (A8) {};
\node[draw,circle,left=.8cm of A8] (A9) {};
\node[draw,circle,left=.8cm of A9] (A10) {};

\node[left=.8cm of A1] {$\tilde{A}_n$};

\draw (A1) -- (A2)  node[above,midway] {};
\draw (A2) -- (A3)  node[above,midway] {};
\draw (A3) -- (A4) node[] {};
\draw (A4) -- (A5) node[above,midway] {};
\draw (A5) -- (A6) node[] {};
\draw (A6) -- (A7) node[] {};
\draw (A7) -- (A8) node[] {};
\draw[loosely dotted,thick] (A8) -- (A9) node[] {};
\draw (A9) -- (A10) node[] {};
\draw (A10) -- (A1) node[] {};

 
\node[draw,circle,below=1.7cm of A1] (B1) {};
\node[draw,circle,right=.8cm of B1] (B2) {};
\node[draw,circle,right=.8cm of B2] (B3) {}; 
\node[draw,circle,right=1cm of B3] (B4) {};
\node[draw,circle,right=.8cm of B4] (B5) {};
\node[draw,circle,above right=.8cm of B5] (B6) {};
\node[draw,circle,below right=.8cm of B5] (B7) {};
\node[left=.8cm of B1] {$\tilde{B}_n$};

\draw (B1) -- (B2)  node[above,midway] {$4$};
\draw (B2) -- (B3)  node[above,midway] {};
\draw[loosely dotted,thick] (B3) -- (B4) node[] {};
\draw (B4) -- (B5) node[above,midway] {};
\draw (B5) -- (B6) node[above,midway] {};
\draw (B5) -- (B7) node[above,midway] {};

 
\node[draw,circle,below=1.5cm of B1] (C1) {};
\node[draw,circle,right=.8cm of C1] (C2) {};
\node[draw,circle,right=.8cm of C2] (C3) {}; 
\node[draw,circle,right=1cm of C3] (C4) {};
\node[draw,circle,right=.8cm of C4] (C5) {};
\node[left=.8cm of C1] (CCC) {$\tilde{C}_n$};

\draw (C1) -- (C2)  node[above,midway] {$4$};
\draw (C2) -- (C3)  node[above,midway] {};
\draw[loosely dotted,thick] (C3) -- (C4) node[] {};
\draw (C4) -- (C5) node[above,midway] {$4$};

 
\node[draw,circle,below=1cm of C1] (D1) {};
\node[draw,circle,below right=0.8cm of D1] (D3) {}; 
\node[draw,circle,below left=0.8cm of D3] (D2) {};
\node[draw,circle,right=.8cm of D3] (DA) {};
\node[draw,circle,right=1cm of DA] (DB) {};
\node[draw,circle,right=.8cm of DB] (D4) {};
\node[draw,circle, above right=.8cm of D4] (D5) {};
\node[draw,circle,below right=.8cm of D4] (D6) {};
\node[below=1.5cm of CCC] {$\tilde{D}_n$};

\draw (D1) -- (D3)  node[above,midway] {};
\draw (D2) -- (D3)  node[above,midway] {};
\draw (D3) -- (DA) node[above,midway] {};
\draw[loosely dotted] (DA) -- (DB);
\draw (D4) -- (DB) node[above,midway] {};
\draw (D4) -- (D5) node[above,midway] {};
\draw (D4) -- (D6) node[below,midway] {};
  
 
\node[draw,circle,below=2.5cm of D1] (I1) {};
\node[draw,circle,right=.8cm of I1] (I2) {};
\node[left=.8cm of I1] {$\tilde{A}_1$};

\draw (I1) -- (I2)  node[above,midway] {$\infty$};
 
 
\node[draw,circle,below=1.2cm of I1] (H1) {};
\node[draw,circle,right=.8cm of H1] (H2) {};
\node[draw,circle,right=.8cm of H2] (H3) {}; 
\node[left=.8cm of H1] {$\tilde{B}_2$};

\draw (H1) -- (H2)  node[above,midway] {$4$};
\draw (H2) -- (H3)  node[above,midway] {$4$};
 
 
\node[draw,circle,below=1.2cm of H1] (HH1) {};
\node[draw,circle,right=.8cm of HH1] (HH2) {};
\node[draw,circle,right=.8cm of HH2] (HH3) {}; 
\node[left=.8cm of HH1] {$\tilde{G}_2$};

\draw (HH1) -- (HH2)  node[above,midway] {$6$};
\draw (HH2) -- (HH3)  node[above,midway] {};
  
 
\node[draw,circle,below=1.2cm of HH1] (F1) {};
\node[draw,circle,right=.8cm of F1] (F2) {};
\node[draw,circle,right=.8cm of F2] (F3) {}; 
\node[draw,circle,right=.8cm of F3] (F4) {};
\node[draw,circle,right=.8cm of F4] (F5) {};
\node[left=.8cm of F1] {$\tilde{F}_4$};

\draw (F1) -- (F2)  node[above,midway] {};
\draw (F2) -- (F3)  node[above,midway] {$4$};
\draw (F3) -- (F4)  node[above,midway] {};
\draw (F4) -- (F5)  node[above,midway] {};

 
\node[draw,circle,below=1.2cm of F1] (E1) {};
\node[draw,circle,right=.8cm of E1] (E2) {};
\node[draw,circle,right=.8cm of E2] (E3) {}; 
\node[draw,circle,right=.8cm of E3] (E4) {};
\node[draw,circle,right=.8cm of E4] (E5) {};
\node[draw,circle,below=.8cm of E3] (EA) {};
\node[draw,circle,below=.8cm of EA] (EB) {};
\node[left=.8cm of E1] {$\tilde{E}_6$};

\draw (E1) -- (E2)  node[above,midway] {};
\draw (E2) -- (E3)  node[above,midway] {};
\draw (E3) -- (E4)  node[above,midway] {};
\draw (E4) -- (E5)  node[above,midway] {};
\draw (E3) -- (EA)  node[left,midway] {};
\draw (EB) -- (EA)  node[left,midway] {};
 
 
\node[draw,circle,below=3cm of E1] (EE1) {};
\node[draw,circle,right=.8cm of EE1] (EEB) {};
\node[draw,circle,right=.8cm of EEB] (EE2) {};
\node[draw,circle,right=.8cm of EE2] (EE3) {}; 
\node[draw,circle,right=.8cm of EE3] (EE4) {};
\node[draw,circle,right=.8cm of EE4] (EE5) {};
\node[draw,circle,right=.8cm of EE5] (EE6) {};
\node[draw,circle,below=.8cm of EE3] (EEA) {};
\node[left=.8cm of EE1] {$\tilde{E}_7$};

\draw (EE1) -- (EEB)  node[above,midway] {};
\draw (EE2) -- (EEB)  node[above,midway] {};
\draw (EE2) -- (EE3)  node[above,midway] {};
\draw (EE3) -- (EE4)  node[above,midway] {};
\draw (EE4) -- (EE5)  node[above,midway] {};
\draw (EE5) -- (EE6)  node[above,midway] {};
\draw (EE3) -- (EEA)  node[left,midway] {};
 
 
\node[draw,circle,below=1.8cm of EE1] (EEE1) {};
\node[draw,circle,right=.8cm of EEE1] (EEE2) {};
\node[draw,circle,right=.8cm of EEE2] (EEE3) {}; 
\node[draw,circle,right=.8cm of EEE3] (EEE4) {};
\node[draw,circle,right=.8cm of EEE4] (EEE5) {};
\node[draw,circle,right=.8cm of EEE5] (EEE6) {};
\node[draw,circle,right=.8cm of EEE6] (EEE7) {};
\node[draw,circle,right=.8cm of EEE7] (EEE8) {};
\node[draw,circle,below=.8cm of EEE3] (EEEA) {};
\node[left=.8cm of EEE1] {$\tilde{E}_8$};

\draw (EEE1) -- (EEE2)  node[above,midway] {};
\draw (EEE2) -- (EEE3)  node[above,midway] {};
\draw (EEE3) -- (EEE4)  node[above,midway] {};
\draw (EEE4) -- (EEE5)  node[above,midway] {};
\draw (EEE5) -- (EEE6)  node[above,midway] {};
\draw (EEE6) -- (EEE7)  node[above,midway] {};
\draw (EEE8) -- (EEE7)  node[above,midway] {};
\draw (EEE3) -- (EEEA)  node[left,midway] {}; 
\end{tikzpicture}
\caption{Irreducible affine diagram}
\label{affi_diag}
\end{minipage}
\end{figure}

\subsection{The face of a properly convex closed (or open) set}$\,$

\par{
Let $C$ be a properly convex closed subset of $\S^d$. We introduce the following equivalence relation on $C$; $x \sim y$ when the segment $[x,y]$ can be extended beyond $x$ and $y$. The equivalence classes of $\sim$ are called \emph{the open faces of $C$}, the closure of an open face is a \emph{face of $C$}. The \emph{support} of a face or an open face is the smallest projective space containing it. The \emph{dimension of a face} is the dimension of its support. For properly convex open subsets, we just apply this definition to their closure, so a face of $\O$ is a subset of $\overline{\O}$.
}
\par{
The interior of a face $F$ in its support (i.e its relative interior) is equal to the unique open face $f$ such that $\overline{f}= F$. Finally, one should remark that if $f$ is an open face of $C$ then $f$ is a properly convex open set in its support. The only face of dimension $d$ is $C$. A face of dimension $d-1$ is called a \emph{facet}, a face of dimension $0$ a \emph{vertex}, a face of dimension $1$ an \emph{edge} and a face of dimension $d-2$ a \emph{ridge}.
}

\subsection{Mirror polytope}\label{def_mirror}\label{proj}$\,$

\par{
A \emph{projective polytope} is a properly convex closed set $P$ of $\S(V)$ with non-empty interior such that there exists a finite number of linear form $\alpha_1,...,\alpha_r $
on $V$ such that $P = \S(\{ x \in V \smallsetminus \{ 0 \} \,|\, \alpha_i(x) \leqslant 0,\, i=1...r\})$.
}
\par{
A \emph{projective reflection} is an element of $\mathrm{SL}^{\pm}(V)$ of order 2 which is the identity on a hyperplane. Each projective reflection $\sigma$ can be written as: $\sigma=Id-\alpha\otimes v$ where $\alpha$ is a linear form and $v$ a vector such that $\alpha(v)=2$. This notation means that $\sigma(x)=x-\alpha(x)v$. 
}
\par{
A \emph{projective rotation} is an element of $\mathrm{SL}(V)$  which is the identity on a codimension 2 subspace $H$ and conjugate to the $2 \times 2$ matrix
$\left(\begin{matrix}
\cos(\theta) & -\sin(\theta) \\ 
\sin(\theta) & \cos(\theta)
\end{matrix}\right)$ on a plan $\Pi$ such that $H \oplus \Pi = V$. The two following lemmas are easy but essential.
}

\label{inequa_D}

\begin{lemma}[(Vinberg, Proposition 6 of \cite{MR0302779})]\label{classi_dim1}
Let $\sigma_s= Id - \alpha_s \otimes v_s$ and $\sigma_t= Id - \alpha_t \otimes v_t$ be two reflections of $\R^2$. Let $\G$ be the group generated by $\sigma_s$ and $\sigma_t$. Let $\C$ be the cone $\{ x \in \R^2 \,\mid\, \alpha_s(x) \leqslant 0 \textrm{ and } \alpha_t(x) \leqslant 0 \}$. If the sets $(\g(\C))_{\g \in \G}$ have disjoint interiors, then:

\begin{center}
\begin{tabular}{cc}
(C)&
$\left\{
\begin{tabular}{cc}
1) & $\alpha_s(v_t) \leqslant 0$ \textrm{ and } $\alpha_t(v_s) \leqslant 0$\\
\textrm{and}\\
2) & $\alpha_s(v_t)= 0 \Leftrightarrow \alpha_t(v_s) =0$.
\end{tabular}
\right.$
\end{tabular}
\end{center}
\end{lemma}

\begin{lemma}[(Vinberg, Propositions 6 and 7 of \cite{MR0302779})]
With the same notation, iIf the condition (C) is satisfied then the group $\G$ preserves a symmetric bilinear form $b$ on $\R^2$.

\begin{enumerate}
\item If $\alpha_s(v_t) \alpha_t(v_s) < 4$ then $b$ is positive definite and the element $\sigma_s \sigma_t$ is a rotation of angle $2\theta_{st}$ where $\alpha_s(v_t) \alpha_t(v_s) = 4 \cos^2(\theta_{st})$. In particular, the group $\G$ is discrete if and only if the number $m_{st} = \frac{\pi}{\theta_{st}}$ is an integer.

\item If $\alpha_s(v_t) \alpha_t(v_s) > 4$ then $b$ is of signature $(1,1)$, the element $\sigma_s \sigma_t$ is loxodromic\footnote{Here, this means that $\sigma_1 \sigma_2$ is diagonalizable over $\R$.}, the action on $\PP^1$ preserves a unique properly convex open $\O$ set, and the action on $\O$ is cocompact.

\item Otherwise $\alpha_s(v_t) \alpha_t(v_s) = 4$, $b$ is positive and degenerate, the element $\sigma_s \sigma_t$ is unipotent\footnote{Here, this means that $(\sigma_s \sigma_t-Id)^2=0$.}, the action on $\PP^1$ preserves a unique affine chart $\A^1$, and the action on $\A^1$ is cocompact.
\end{enumerate}
 These actions on $\R^2$ are described by Figure \ref{presen_3}.
\end{lemma}

\begin{figure}[h!]
\centering
\scalebox{.8}{%
\begin{tikzpicture}[line cap=round,line join=round,>=triangle 45,x=1.0cm,y=1.0cm]
\clip(-2.64,-2.66) rectangle (2.59,2.59);
\fill[fill=black,fill opacity=0.1] (0,0) -- (22.49,22.49) -- (32.55,0) -- cycle;
\draw [line width=1.6pt,domain=-2.64:2.59] plot(\x,{(-0-0*\x)/6});
\draw [line width=1.6pt] (0,-2.66) -- (0,2.59);
\draw [line width=1.6pt,domain=-2.64:2.59] plot(\x,{(-0--0.71*\x)/0.71});
\draw [line width=1.6pt,domain=-2.64:2.59] plot(\x,{(-0--0.71*\x)/-0.71});
\draw (0,0)-- (22.49,22.49);
\draw (22.49,22.49)-- (32.55,0);
\draw (32.55,0)-- (0,0);
\draw (1.21,0.72) node[anchor=north west] {$\mathcal{C}$};
\end{tikzpicture}
}
\scalebox{.8}{%
\begin{tikzpicture}[line cap=round,line join=round,>=triangle 45,x=1.0cm,y=0.5cm]
\clip(-4.09,-0.5) rectangle (4.07,8.86);
\fill[fill=black,fill opacity=0.1] (0,0) -- (0,30) -- (5,30) -- cycle;
\draw [line width=2pt,domain=0.0:4.0741396026318455] plot(\x,{(-0--30*\x)/10});
\draw [line width=2pt,domain=-4.092043260206375:0.0] plot(\x,{(-0--30*\x)/-10});
\draw (0,0) -- (0,8.86);
\draw [domain=0.0:4.0741396026318455] plot(\x,{(-0--30*\x)/5});
\draw [line width=1.6pt,domain=0.0:4.0741396026318455] plot(\x,{(-0--30*\x)/7.5});
\draw [line width=1.6pt,domain=-4.092043260206375:0.0] plot(\x,{(-0--30*\x)/-5});
\draw [line width=1.6pt,domain=-4.092043260206375:0.0] plot(\x,{(-0--30*\x)/-7.5});
\draw (0,30)-- (5,30);
\draw [line width=1.6pt] (0,0)-- (0,30);
\draw (0,30)-- (5,30);
\draw [line width=1.6pt] (5,30)-- (0,0);
\draw (0.37,7.91) node[anchor=north west] {$\mathcal{C}$};
\begin{scriptsize}
\fill [color=black] (2.18,8.24) circle (0.5pt);
\fill [color=black] (2.5,8) circle (0.5pt);
\fill [color=black] (-2.5,8) circle (0.5pt);
\fill [color=black] (-2.23,8.16) circle (0.5pt);
\fill [color=black] (2.34,8.12) circle (0.5pt);
\fill [color=black] (-2.36,8.08) circle (0.5pt);
\end{scriptsize}
\end{tikzpicture}
}
\scalebox{.8}{%
\begin{tikzpicture}[line cap=round,line join=round,>=triangle 45,x=1.1cm,y=0.6cm]
\clip(2.56,-4.02) rectangle (10.87,2.87);
\fill[fill=black,fill opacity=0.1] (6.04,3.42) -- (6.79,3.37) -- (6.46,-3) -- cycle;
\draw (6.15,2.03) node[anchor=north west] {$\mathcal{C}$};
\draw [line width=1.6pt] (0.69,-3)-- (13.06,-3);
\draw [line width=1.6pt] (3.73,3.52)-- (6.46,-3);
\draw [line width=1.6pt] (6.46,-3)-- (4.52,3.47);
\draw [line width=1.6pt] (5.29,3.42)-- (6.46,-3);
\draw [line width=2pt] (6.04,3.42)-- (6.46,-3);
\draw [line width=2pt] (6.46,-3)-- (6.79,3.37);
\draw [line width=1.6pt] (7.54,3.38)-- (6.46,-3);
\draw [line width=1.6pt] (6.46,-3)-- (8.28,3.36);
\draw [line width=1.6pt] (9.01,3.32)-- (6.46,-3);
\draw [line width=1.6pt] (6.46,-3)-- (9.73,3.27);
\draw (6.79,3.37)-- (6.46,-3);
\draw (6.46,-3)-- (6.04,3.42);
\begin{scriptsize}
\fill [color=black] (9,1) circle (1.0pt);
\fill [color=black] (10,0) circle (1.0pt);
\fill [color=black] (3.5,0) circle (1.0pt);
\fill [color=black] (4.5,1) circle (1.0pt);
\fill [color=black] (4,0.5) circle (1.0pt);
\fill [color=black] (9.5,0.5) circle (1.0pt);
\end{scriptsize}
\end{tikzpicture}
}
\caption{}
\label{presen_3}
\end{figure}

The two previous lemmas motive the following definition:

\begin{de}
A \emph{mirror polytope} is a convex projective polytope $P$ with the data of a projective reflection $\sigma_s$ across each facet $s$ of $P$, such that for any two facets $s$ and $t$ of $P$ such that $s \cap t$ is a ridge of $P$, the pair $\{ \sigma_s, \sigma_t\}$ satisfies the conditions $(C)$. We say that \emph{the dihedral angle between the facets $s$ and $t$ is $\theta_{st}$} when we have $\alpha_s(v_t) \alpha_t(v_s) = 4 \cos^2(\theta_{st})$. Otherwise, we say that the angle is $0$. Two mirror polytopes are \emph{isomorphic} if one can find an isomorphism of vector spaces which sends the first polytope to the second, and sends the reflections of the first to the reflections of the second. When $P$ and $Q$ are isomorphic, we will write $P \simeq Q$.
\end{de}

\begin{nota}
\textit{
The following notation will be used along this text. Let $P$ be a mirror polytope; the symbol $S$ will denote the set of facets of $P$. We can always write $P= \S(\{ x \in V \setminus \{ 0 \} \,|\, \alpha_s(x) \leqslant 0,\, s \in S \})$. For each facet $s \in S$,  we denote by $\sigma_s$ the reflection of $P$ which fixes each point of $s$. We can write it $\sigma_s = Id - \alpha_s\otimes v_s$ with $v_s \in V$ and $\alpha_s(v_s)=2$. Be careful that the couple $(\alpha_s,v_s)$ is unique up to a multiplicative positive constant,\footnote{By the action $\lambda \cdot (\alpha_s,v_s) = (\lambda \alpha_s,\lambda^{-1} v_s)$.} but nothing will depend on this choice. The point $[v_s] \in \S(V)$, which is unique, is called the \emph{polar} of the facet $s$ (or of $\sigma_s$) and the hyperplane $\{ x \in \S^d \,|\, \alpha_s(x)=0 \}$ is called the \emph{support} of the facet $s$ or of $\sigma_s$. We will denote by the symbol $\G_P$ or simply $\G$ the group generated by the reflections $\sigma_s$ for $s\in S$.
}
\end{nota}

\begin{cor}
Let $P$ be a mirror polytope. If the sets $\g(\mathring{P})$ are disjoint for $\g \in \G_P$, then the family $(\alpha_s(v_t))_{s,t\in S}$ verifies the condition $(C)$ and the following condition:

$$
\begin{array}{cl}
(D)

&
\left\{
\begin{array}{cl}
1)  &  \alpha_s(v_t) \alpha_t(v_s) = 4 \cos^2(\theta_{st}),\\
 & \textrm{ and the number } m_{st} = \frac{\pi}{\theta_{st}}\\
 & \textrm{ is an integer greater or equal to } 2,\\
 \textrm{or}\\
2) & \alpha_s(v_t) \alpha_t(v_s) \geqslant 4.
\end{array}
\right.
\end{array}
$$
\end{cor}

\begin{figure}
\centering
\begin{tikzpicture}[line cap=round,line join=round,>=triangle 45,x=1.0cm,y=1.0cm]
\clip(1.69,-2.97) rectangle (14.61,5.13);
\fill[line width=1.6pt,fill=black,fill opacity=0.2] (7.35,3.15) -- (8.86,3.22) -- (8,4.78) -- cycle;
\fill[line width=1.6pt,fill=black,fill opacity=0.2] (7.35,3.15) -- (4,3) -- (6.63,1.35) -- cycle;
\fill[line width=1.6pt,fill=black,fill opacity=0.2] (6.63,1.35) -- (5.2,-2.22) -- (8.41,0.23) -- cycle;
\fill[line width=1.6pt,fill=black,fill opacity=0.2] (8.41,0.23) -- (11.64,-1.8) -- (9.89,1.37) -- cycle;
\fill[line width=1.6pt,fill=black,fill opacity=0.2] (9.89,1.37) -- (12.52,3.38) -- (8.86,3.22) -- cycle;
\draw [line width=1.6pt,dash pattern=on 1pt off 1pt on 2pt off 4pt] (7.35,3.15)-- (8.86,3.22);
\draw [line width=1.6pt,dash pattern=on 1pt off 1pt on 2pt off 4pt] (8.41,0.23)-- (6.63,1.35);
\draw [line width=1.6pt,dash pattern=on 1pt off 1pt on 2pt off 4pt] (8.86,3.22)-- (9.89,1.37);
\draw [line width=1.6pt,dash pattern=on 1pt off 1pt on 2pt off 4pt] (6.63,1.35)-- (7.35,3.15);
\draw [line width=1.6pt,dash pattern=on 1pt off 1pt on 2pt off 4pt] (8.41,0.23)-- (9.89,1.37);
\draw (8.12,2.29) node[anchor=north west] {$P$};
\draw (5.5,3.7) node[anchor=north west] {$v_1$};
\draw (10.23,3.04) node[anchor=north west] {$v_3$};
\draw (6.9,4.22) node[anchor=north west] {$v_2$};
\draw (9.82,0.4) node[anchor=north west] {$v_4$};
\draw (6.24,-0.6) node[anchor=north west] {$v_5$};
\begin{scriptsize}
\fill [color=black] (7.64,3.88) circle (1.5pt);
\fill [color=black] (10.13,2.64) circle (1.5pt);
\fill [color=black] (9.71,0.41) circle (1.5pt);
\fill [color=black] (6.59,-0.53) circle (1.5pt);
\fill [color=black] (5.86,3.08) circle (1.5pt);
\end{scriptsize}
\end{tikzpicture}
\caption{Illustration of equation (C)}
\label{presen_polyg}
\end{figure}
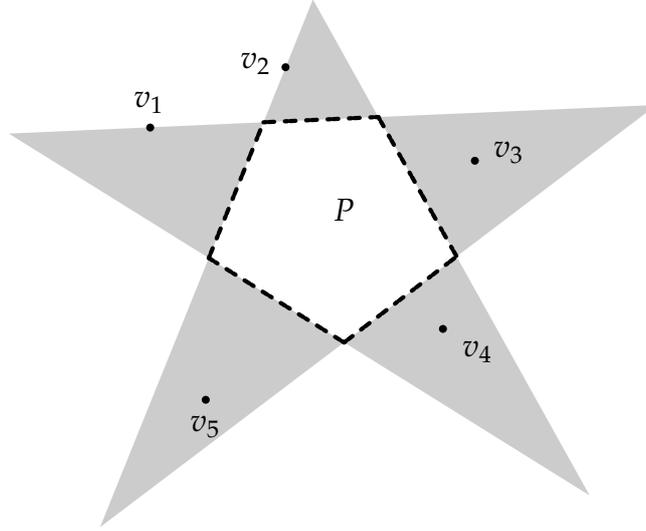

\begin{de}\label{def_cox_poly}
A mirror polytope $P$ is a \emph{Coxeter  polytope} when all its dihedral angles are sub-multiples\footnote{Precisely, $\theta = \frac{\pi}{m}$ with $m$ an integer greater than or equal to $2$ OR $m=\infty$.} of $\pi$.
\end{de}

If $P$ is a Coxeter  polytope, the \emph{Coxeter system associated to $P$} is the Coxeter system $(S,M)$, where $S$ is the set of facets of $P$ and where for all $s,t \in S$, we have $M_{st}=m_{st}$ if the facets $s,t \in S$ are such that $s \cap t$ is a ridge of $P$ and $\theta_{st} = \frac{\pi}{m_{st}}$,  otherwise $M_{st} =\infty$. We will denote by the letter $W_P$ or simply $W$ the \emph{Coxeter group associated to the system $(S,M)$}.

\begin{rem}
Figure \ref{presen_polyg} shows a pentagon $P$. Any mirror  structure on this polygon verifies that the polar $[v_i]$ of the facet are in the grey triangle given by the facet $i$. This is a consequence of the inequalities $(C)$. We will see that these inequalities have usefull implications.
\end{rem}

\subsection{The limit set of positively proximal subgroup of $\ss$}

In this section, we just state a theorem of existence of limit sets. We will give a more detailed discussion in paragraph \ref{detail_prox}.

\subsubsection{Strongly irreducible case}

\begin{theorem}[(Benoist, Lemma 2.9 and 3.3 of \cite{MR1767272})]\label{exis_limi_set}
Let $\G$ be a strongly irreducible subgroup of $\ss$ preserving a properly convex open set. There exists a smallest closed $\G$-invariant subset $\LG$ for the action of $\G$ on $\PP^d$. This closed subset is called the \emph{limit set} of $\G$.
\end{theorem}

\begin{cor}\label{small_and_big}
Let $\G$ be a strongly irreducible subgroup of $\ss$ preserving a properly convex open set. There exists a smallest and a largest $\G$-invariant convex open subset for the action of $\G$ on $\PP^d$.  
\end{cor}

\subsubsection{Irreducible case}

\begin{lemma}[(Benoist, Lemma 2.9 and 3.3 of \cite{MR1767272})]\label{benoist_decomp}
Let $\G$ be an irreducible subgroup of $\ss$ preserving a properly convex open set $\O$. Let $\G_0$ be the Zariski connected component of $\G$. There exists a decomposition $\R^{d+1}=\bigoplus_{i=1,...,r} E_i$ in strongly irreducible $\G_0$-sub-modules such that the action of $\G_0$ on each factor preserves a properly convex open cone. The \emph{limit set} of $\G$ is the union of the limit set of $\G_0$ in $\PP(E_i)$.
\end{lemma}

\section{The Theorem of Tits-Vinberg and the Theorems of Vinberg}

In this section, we recall the Theorem of Tits-Vinberg and the Theorems of Vinberg.

\subsection{Tiling theorem}$\,$
\par{
To avoid any confusion, we recall a general definition of a tiling.
}
\begin{de}
A family $(E_i)_{i \in I}$ of closed set \emph{tiles} a topological set $X$ when we have the following three conditions: For all $i \in I$, the interior of $E_i$ is dense in $E_i$, the union of the $E_i$ is $X$ and for all $i \neq j$ in $I$, the intersection of the interiors of $E_i$ and $E_j$ is empty.
\end{de}

\par{
If $(S,M)$ is a Coxeter system then for every subset $S'$ of $S$, one can consider the Coxeter group $W_{S'}$ associated to the Coxeter system $(S',M')$, where $M'$ is the restriction of $M$ to $S'$. Theorem \ref{theo_vinberg} shows that the natural morphism $W_{S'} \rightarrow W_S$ is injective. Therefore, $W_{S'}$ may be identified with the subgroup of $W_S$ generated by the subset $S'$.
}
\par{
 If $P$ is a Coxeter  polytope and $f$ is a face (or an open face) of $P$, and $\overline{f} \neq P$, then we will write $S_f= \{ s \in S \mid f \subset s \}$ and $W_f=W_{S_f}$.
}
\\
\par{
Let $(S,M)$ be a Coxeter system. A \emph{standard parabolic subgroup} of the Coxeter group $W_S$ is a subgroup generated by some elements of $S$. A \emph{parabolic subgroup} of $W_S$ is conjugate of a standard parabolic subgroup.
}

\begin{theorem}[(Tits, chapter V \cite{MR0240238} for the Tits's simplex or Vinberg \cite{MR0302779})]\label{theo_vinberg}
Let $P$  be a Coxeter  polytope of $\S(V)$, $W_P$ be the associated Coxeter group and $\G_P$ the group generated by the projective reflections $(\sigma_s)_{s \in S}$. Then,
\begin{enumerate}
\item The morphism $\sigma:W_P \rightarrow \Gamma_P$ defined by $\sigma(s) =
\sigma_s$ is an isomorphism.

\item The polytopes $\big( \gamma(P) \big)_{\gamma \in \Gamma_P}$ tile aconvex set$\C_P$ of $\S(V)$.

\item[$\medtriangleup$] The group $\Gamma_P$ acts properly on $\O_P= \mathring{\C_P}$, the interior of $\C_P$.

\item[$\meddiamond$] The group $\Gamma_P$ is a discrete subgroup of $\mathrm{SL}^{\pm}(V)$.

\item[$\medstar$] An open face $f$ of $P$ lies in $\O_P$ if and only if the Coxeter group $W_f$ is finite.

\item[$\davidsstar$] For every parabolic subgroup $U$ of $W_P$, the union $\bigcup_{\g \in U} \g(P)$ is convex.
\end{enumerate}
\end{theorem}

\begin{cor}\label{cocompact}
The convex set$\C_{P}$ is open if and only if the action of $\G_{P}$ on $\O_{P}$ is cocompact if and only if for every vertex $p$ of $P$ the Coxeter group $W_p$ is finite. Following Vinberg, we will say that in this case, $P$ is \emph{perfect}\footnote{Definition 8 of \cite{MR0302779}.}.
\end{cor}

The following theorem can give the impression to be a corollary, but in fact Vinberg uses it to conclude the proof of his theorem (see Lemma 10 of \cite{MR0302779}).

\begin{theorem}[(Coxeter \cite{MR1503182})]\label{finite}
Theconvex set$\C_{P}$ is the projective sphere $\S(V)$ if and only if the group $W_P$ is finite if and only if the Coxeter group $W_P$ is spherical.
\end{theorem}

\begin{rem}
The sixth point of Theorem \ref{theo_vinberg} is not explicit in Vinberg's article but it is an easy consequence of the techniques he develops.
\end{rem}

\subsection{The Cartan Matrix of a Coxeter polytope}\label{def_mix}

\begin{de}
A matrix $A$ of $\textrm{M}_m{(\R)}$ is a \emph{Cartan matrix} when:
\begin{itemize}
\item $\forall \, i = 1...m$, $a_{ii} = 2$.
\item $\forall \, i,j= 1...m$, $a_{ij} = 0 \Leftrightarrow a_{ji}=0$.
\item All non-diagonal coefficients of $A$ are negative or null. 
\end{itemize}
\end{de}

\par{
A matrix is \emph{reducible} if after a simultaneaous permutation of the rows and the columns, one as a non trivial diagonal bloc matrix. A matrix is \emph{irreducible} if and only if it is not reducible.
}
\par{
The theorem of Perron-Frobenius shows that \emph{the spectral radius of an irreducible matrix with positive or null coefficients is a simple eigenvalue}. Hence, an irreducible Cartan matrix $A$ has a unique eigenvalue $\lambda_A$ of minimal modulus. We will say that $A$ is of \emph{positive type}, \emph{zero type} or \emph{negative type} when $\lambda_A >0$, $\lambda_A =0$ or $\lambda_A <0$.
}
\par{
Given a Coxeter polytope $P$, one can define the matrix $A$ where $A_{ij} = \alpha_i(v_j)$. By definition of a Coxeter polytope, $A$ is a Cartan matrix; we will call it the \emph{Cartan matrix associated to the Coxeter polytope $P$} and denoted it $A_P$.
}
\par{
Of course, the Coxeter group $W_P$ is irreducible if and only if the Cartan matrix $A_P$ is irreducible. In that case, we say that $P$ is of \emph{positive type} (resp. \emph{zero type}, resp. \emph{negative type}) according to the type of $A_P$.
}
\par{
If the Coxeter group $W_P$ is not irreducible, then the Cartan matrix $A_P$ is the sum of its irreducible components, we say that $P$ is of \emph{positive type}, (resp. \emph{zero type}, resp. \emph{negative type}) if all the irreducible components are of \emph{positive type}, (resp. \emph{zero type}, resp. \emph{negative type}). It is easy to find a Coxeter  polytope $P$, such that the components of $A_P$ do not have not the same type.
}

\subsection{Tits's simplex}\label{sub:tits_simplex}$\,$
\par{
To each Coxeter group $W$, we can associate a Coxeter polytope. The polytope will be a simplex of dimension the rank of $W$ minus 1. The construction\footnote{In order to get a Coxeter polytope, one has to take the dual of the standard representation introduced by Tits.} is the following:
}
\par{
Suppose that $W$ arises from the Coxeter system $M=(M_{st})_{s,t \in S}$. Consider the vector space $V=(\R^S)^*$, and denote by $(e_s)_{s\in S}$ the canonical basis of $\R^S$. We consider the simplicial cone $\C=\{ \varphi \in (\R^S)^* \,|\, \varphi(e_s) \leqslant 0,\, \forall s \in S \}$; the simplex we want is $P=\S(\C)$. The reflection associated to the element $s\in S$ is the reflection across the facet $\S(\{ \varphi \,|\, \varphi(e_s)=0 \}) \cap P$, and given by the formula $\sigma_s(\varphi) = \varphi-2\varphi(e_s) B_W(e_s,\cdot)$, where $B_W$ is the symmetric bilinear form given by $B_W(e_s,e_t)= -\cos \Big(\frac{\pi}{M_{st}}\Big)$.
}
\par{
The resulting Coxeter polytope will be called the \emph{Tits simplex associated to $W$} and denoted by $\Delta_W$. The polar of the facet $s$ is the point $[2B_W(e_s,\cdot)]$ of $\S((\R^S)^*)$. We stress that the group $\G_{\Delta_W}$ preserves the symmetric bilinear form $B_W$.
}

\subsection{Proper convexity of $\O_P$}$\,$

\begin{theorem}[(Vinberg, Lemma 15 and Proposition 25 \cite{MR0302779})]\label{carac_prop}
Let $P$ be a Coxeter  polytope of $\S^d$. The convex set$\O_P$ is properly convex if and only if the Cartan matrix $A_P$ of $P$ is of negative type.
\end{theorem}

\begin{rem}
The terminology in \cite{MR0302779} and the terminology we use can be in opposition. A cone $C$ is strictly convex for \cite{MR0302779} when it is properly convex for us. Vinberg prefer to speak about a reduced linear Coxeter group while we prefer to talk about a Coxeter  polytope.
\end{rem}

\subsection{Irreducible Coxeter polytope}$\,$

\par{
The following proposition gives the shape of $\O_P$ via the type of $A_P$.
}

\begin{propo}[(Vinberg, \cite{MR0302779})]\label{5thomy}
Let $P$ be an irreducible Coxeter  polytope of $\S^d$. Let $W$ be the Coxeter group associated to $P$. We are in exactly one the following five cases:
\begin{enumerate}
\item The Coxeter group $W$ is spherical; in that case:
\begin{itemize}
\item $A_P$ is of positive type and of rank $d+1$,
\item $\O_P = \S^d$,
\item in fact, $P \simeq \Delta_W$.
\end{itemize}
\item The Coxeter group $W$ is affine but not of type $\tilde{A}_n$; in that case:
\begin{itemize}
\item $A_P$ is of zero type and of rank $d$,
\item $\O_P$ is an affine chart,
\item in fact, $P \simeq \Delta_W$,
\item the action of $\G_P$ on $\O_P$ is cocompact and preserves a euclidean metric.
\end{itemize}
\item The Coxeter group $W$ is affine of type $\tilde{A}_n$, and $A_P$ is of zero type; in that case:
\begin{itemize}
\item $A_P$ is of rank $d$,
\item $\O_P$ is an affine chart,
\item in fact, $P \simeq \Delta_W$,
\item the action of $\G_P$ on $\O_P$ is cocompact and preserves a euclidean metric.
\end{itemize}
\item The Coxeter group $W$ is affine of type $\tilde{A}_n$, and $A_P$ is of negative type; in that case:
\begin{itemize}
\item $A_P$ is of rank $d+1$,
\item $\O_P$ is a simplex; in particular $\O_P$ is a properly convex open set,
\item the action of $\G_P$ on $\O_P$ is cocompact.
\end{itemize}
\item The Coxeter group $W$ is large; in that case:
\begin{itemize}
\item $A_P$ is of rank $r \leqslant d+1$,
\item $\O_P$ is a properly convex open set (which is not a simplex).
\end{itemize}
\end{enumerate}
\end{propo}

\begin{proof}[Explanation of proof]
The first point is given by the Proposition 22 of \cite{MR0302779}, the second and third points are given by Proposition 23 of \cite{MR0302779}. Theorem \ref{carac_prop} shows that in the fourth and fifth point the convex set $\O_P$ is properly convex. Hence, the only thing left to explain is that in the fourth point the convex set $\O_P$ has to be a simplex. This is Lemma 8 in \cite{MR1748082} of Margulis and Vinberg.
\end{proof}

\subsection{Product of Coxeter  polytopes}

\subsubsection{Spherical projective completion}

\par{
If $V$ is a vector space, then $V$ is an affine chart of $\S(V\oplus \R)$. The space $\S(V)$ is a projective hyperplane of $\S(V\oplus \R)$. Finally, $\S(V\oplus \R) \smallsetminus \S(V)$ has two connected components, each isomorphic to the affine space $V$. Hence, $\S(V)$ is the \emph{hyperplane at infinity of} $V$ and  $\S(V\oplus \R)$ is the \emph{spherical projective completion} of $V$.
}

\subsubsection{The Coxeter cone above a Coxeter polytope}$\,$

\par{
Let $P$ be a Coxeter  polytope of $\S^d$, then $\S^{-1}(P)$ is a convex cone of $\R^{d+1}$. The affine space $\R^{d+1}$ is an affine chart of its spherical projective completion $\S^{d+1}=\S(\R^{d+1} \oplus \R)$. We denote by $H_{\infty}$ the projective subspace $\S^d$ in $\S^{d+1}$, i.e. the hyperplane at infinity of $\R^{d+1}$.
}
\par{
The closure $\overline{\S^{-1}(P)}$ of $\S^{-1}(P)$ in $\S^{d+1}$ is a polytope, each facet of $\S^{-1}(P)$ has a reflection coming from $P$, except the facet $H_{\infty} \cap \overline{\S^{-1}(P)}$ to which we associate the reflection across $H_{\infty}$ with polar the origin of the affine chart defined by $H_{\infty}$ which does not contain $\S^{-1}(P)$.
}
\par{
This Coxeter polytope associated to $P$ will be called \emph{the Coxeter cone above $P$} and denoted by the symbol $P \otimes \cdot$, it is a Coxeter  polytope. One should remark that the polytope $P \otimes \cdot$ has one facet $f_{\infty}$ more than $P$, all the ridges included in the facet $f_{\infty}$ have dihedral angle $\frac{\pi}{2}$, so $W_{P \otimes \cdot} = W_P \times \Quotient{\Z}{2\Z}$ where the factor $\Quotient{\Z}{2\Z}$ is given by the reflection $\sigma_{\infty}$ across $H_{\infty}$. Finally, one should remark that $(P \otimes \cdot )\cap H_{\infty}$ is the Coxeter  polytope $P$, and that the convex set $\O_{P \otimes \cdot}$ is the convex hull of $\O_P \subset H_{\infty}$, $0$ and $\sigma_{\infty}(0)$. Figure \ref{mirror_cone} may help to understand the situation.
}

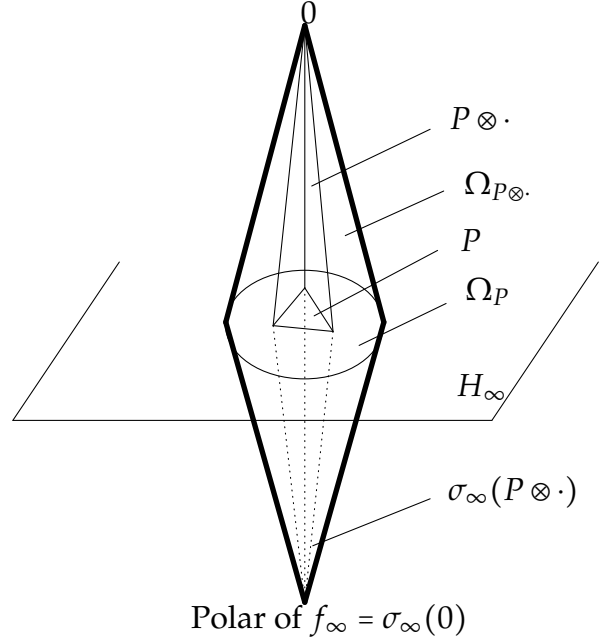
\begin{figure}
\centering
\begin{tikzpicture}[line cap=round,line join=round,>=triangle 45,x=0.35cm,y=0.35cm]
\clip(-2.08,-12.66) rectangle (23.76,12.94);
\draw (4,2)-- (0,-4);
\draw (0,-4)-- (18,-4);
\draw (18,-4)-- (22,2);
\draw [line width=2pt] (10.97,10.96)-- (13.94,-0.28);
\draw [line width=2pt] (10.97,10.96)-- (8,-0.28);
\draw [line width=2pt] (8,-0.28)-- (10.97,-10.89);
\draw [line width=2pt] (10.97,-10.89)-- (13.94,-0.28);
\draw (10.97,1.02)-- (12.03,-0.63);
\draw (12.03,-0.63)-- (9.78,-0.41);
\draw (9.78,-0.41)-- (10.97,1.02);
\draw (10.97,1.02)-- (10.97,10.96);
\draw (9.78,-0.41)-- (10.97,10.96);
\draw (12.03,-0.63)-- (10.97,10.96);
\draw [dotted] (9.78,-0.41)-- (10.97,-10.89);
\draw [dotted] (12.03,-0.63)-- (10.97,-10.89);
\draw [dotted] (10.97,1.02)-- (10.97,-10.89);
\draw [rotate around={0:(10.96,-0.37)}] (10.96,-0.37) ellipse (1.04cm and 0.72cm);
\draw (10.4,12.3) node[anchor=north west] {$0$};
\draw (6.28,-10.55) node[anchor=north west] {$\textrm{Polar of } f_{\infty} = \sigma_{\infty}(0)$};
\draw (16.24,-1.83) node[anchor=north west] {$H_{\infty}$};
\draw (11.32,-0.04)-- (15.96,2.44);
\draw (16.39,3.67) node[anchor=north west] {$P$};
\draw (11.23,4.88)-- (15.68,7.03);
\draw (12.49,3)-- (16.19,4.69);
\draw (13.01,-1.16)-- (16.28,0.43);
\draw (16.59,1.85) node[anchor=north west] {$\Omega_P$};
\draw (15.99,8.31) node[anchor=north west] {$P\otimes \cdot$};
\draw (16.54,5.84) node[anchor=north west] {$\Omega_{P\otimes \cdot}$};
\draw (11.27,-8.75)-- (15.68,-6.97);
\draw (15.94,-5.67) node[anchor=north west] {$\sigma_{\infty}(P\otimes \cdot)$};
\end{tikzpicture}
\caption{The Coxeter cone above a Coxeter polytope}
\label{mirror_cone}
\end{figure}

\par{
In particular, the convex set $\O_{P \otimes \cdot}$ is never properly convex and if $P$ is not elliptic, the action of $W_{P \otimes \cdot}$ on $\O_{P \otimes \cdot}$ is never cocompact.
}
\subsubsection{The product of two convex sets}$\,$

\par{
A sharp convex cone $\C$ of a vector space $V$ is \emph{decomposable} if we can find a decomposition $V=V_1\oplus V_2$ of $V$ such that this decomposition induces a decomposition of $\C$ (i.e. $\C_i=V_i \cap \C$ and $\C=\C_1\times \C_2$). A sharp convex cone is \emph{indecomposable} if it is not decomposable.
}
\par{
We induce this definition to properly convex open set. A properly convex open set $\O$ is \emph{indecomposable} if the cone $\S^{-1}(\O)$ above $\O$ is indecomposable.
}
\\
\par{
This definition suggests a definition of \emph{a product} of two properly convex open sets which is not the Cartesian product. Given two properly convex open sets $\O_1$ and $\O_2$ of the spherical projective spaces $\S(V_1)$ and $\S(V_2)$, we define a new properly convex open set $\O_1 \otimes \O_2$ of the spherical projective space $\S(V_1 \times V_2)$ by the following formula: if $\C_i$ is the cone $\S^{-1}(\O_i)$ then $\O_1 \otimes \O_2 = \S(\C_1 \times \C_2)$.
}
\par{
It is important to note that if $\O_i$ is of dimension $d_i$ then $\O_1 \otimes \O_2$ is of dimension $d_1+d_2+1$. Here is a more pragmatic way to see this product. Take two properly convex subsets $\omega_i$ of a spherical projective space $\S(V)$ with support in direct sum, the $\omega_i$ are not open but we assume that they are open in their supports; assume also that there exists an affine chart containing both $\omega_i$. Then the convex hull in such an affine chart of $\omega_1 \cup \omega_2$ is $\omega_1 \otimes \omega_2$. Some call $\omega_1 \otimes \omega_2$ the \emph{join} of $\omega_1$ and $ \omega_2$.
}
\\
\par{
Just to be clear, we give the definition of a cone in the projective context. A properly convex open set $\O$ is a \emph{cone} when there exist two open faces $\omega_1$ and $\omega_2$ of $\O$ such that $\omega_1$ is a singleton, $\omega_2$ is of dimension $d-1$ and $\O=\omega_1 \otimes \omega_2$. The face $\omega_1$ is called the \emph{summit} of the cone and $\omega_2$ is called the \emph{basis} of the cone.
}
\subsubsection{The product of two Coxeter polytopes}$\,$
\par{
Let $P$ and $Q$ be two Coxeter  polytopes of $\S^d$ and $\S^e$. Then $\S^{-1}(P)$ and $\S^{-1}(Q)$ are convex cones of $\R^{d+1}$ and $\R^{e+1}$. We can take the Cartesian product of these two cones to get a convex cone $\C_{P,Q}$ of $\R^{d+e+2}$ and then project this cone to $\S^{d+e+1}$ to get a polytope $P\otimes Q$ of dimension $d+e+1$.
}
\par{
The facet of $P\otimes Q$ are in correspondence with the facets of $P$ union the facets of $Q$. By extending trivially each reflection from $\R^{d+1}$ (or $\R^{e+1}$) to $\R^{d+e+2}$, we get a Coxeter  polytope whose Coxeter group is $W_P \times W_Q$ and we get $\O_{P \otimes Q}=\O_P \otimes \O_Q$.
}

\subsubsection{Return to the cone}

One can remark that the sphere $\S^0$ of dimension 0 is just two points and is tiled by the Coxeter group $\Quotient{\Z}{2\Z}$ via the Coxeter  polytope of dimension $0$ i.e. a point, hence the Coxeter cone $P \otimes \cdot$ above $P$ is the product of $P$ with the Coxeter  polytope of dimension $0$. This explains our notation.

\subsubsection{Decomposability}

\begin{de}
A Coxeter  polytope $P$ is \emph{decomposable} if one can find two Coxeter  polytopes such that $P = Q \otimes R$, otherwise $P$ is \emph{indecomposable}.
\end{de}

\begin{rem}
If a Coxeter  polytope $P$ is decomposable then the Coxeter group $W_P$ is reducible. The converse is false, think of the right angled square, this Coxeter polygon is indecomposable but the associated Coxeter group is reducible.
\end{rem}

\subsubsection{Theorem of decomposability of Vinberg}

\begin{theorem}\label{decomp_produ}(Corollary 4 of \cite{MR0302779})
Let $P$ be a Coxeter polytope of $\S^d$. We denote by $W$ the Coxeter group associated to $P$. Suppose that $W$ is reducible. If $\textrm{rank}(A_P) = d+1$ or $P$ is a simplex then $P$ is decomposable.
\end{theorem}

\subsection{Elliptic, parabolic, loxodromic Coxeter polytopes}

\begin{de}\label{def_3types}
A Coxeter  polytope $P$ of $\S^d$ is
\begin{enumerate}
\item \emph{elliptic} when $A_P$ is of positive type,
\item \emph{parabolic} when $A_P$ is of zero type and of rank $d$,
\item \emph{loxodromic} when $A_P$ is of negative type and of rank $d+1$.
\end{enumerate}
\end{de}

\begin{rem}
Let $P$ be a Coxeter  polytope. If $P$ is elliptic then the rank of $A_P$ is necessarily $d+1$. If $A_P$ is of zero type then the rank of $A_P$ cannot be $d+1$, but it can be strictly less than $d$. If  $A_P$ is of negative type then the rank of $A_P$ can be strictly less than $d+1$.
\end{rem}

\begin{rem}
We recall that for Vinberg, a Coxeter  polytope $P$ is \emph{hyperbolic} if $P$ is loxodromic and $\G_P$ preserves an ellipsoid (i.e. $\G_P$ is a subgroup of a conjugate of $\so{d}$).
\end{rem}

\subsection{About irreducibility}
\subsubsection{Characterisation of the irreducibility of $\G_P$}

\begin{theorem}[(Vinberg, Prop 18 and Corollary of prop 19 of \cite{MR0302779})]\label{irred}
Let $P$ be a Coxeter  polytope of $\S^d$. Then the following assertions are equivalent:
\begin{enumerate}
\item The representation $\rho:W_P \rightarrow \spm{d+1}$ is irreducible.
\item The Coxeter group $W_P$ is irreducible and the family $(v_s)_{s \in S}$ generates $\R^{d+1}$.
\item The Coxeter group $W_P$ is irreducible and the Cartan Matrix $A_P$ of $P$ is of rank $d+1$.
\end{enumerate}
In particular, if $W_P$ is infinite then $\rho$ is irreducible if and only if the Coxeter  polytope $P$ is irreducible and loxodromic.
\end{theorem}

\begin{rem}
Let $P$ be a Coxeter  polytope. From Theorem \ref{exis_limi_set}, we learn that we can define a limit set for the group $\G_P$ as soon as the group $\G_P$ is irreducible. Hence, Theorem \ref{irred} shows that the limit set of $\G_P$ is defined as soon as $P$ is irreducible and loxodromic. We will denote the limit set of $\G_P$ by the symbol $\LGP$ or $\LG$. The limit set is a crucial object for us. Its definition is easier to handle when the group $\G_P$ is strongly irreducible. The next theorem shows that if $P$ is irreducible and loxodromic then $\G_P$ is strongly irreducible except if $W_P$ is affine.
\end{rem}

\subsubsection{From irreducible to strongly irreducible}

\begin{theorem}[(Folklore)]\label{tri}
Let $P$ be a Coxeter  polytope of $\S^d$. Suppose that the representation $\rho:W_P \rightarrow \spm{d+1}$ is irreducible. Then we have the following exclusive trichotomy:
\begin{enumerate}
\item The Coxeter group $W_P$ is spherical and $\O_P=\S^d$.
\item The Coxeter group $W_P$ is affine of type $\tilde{A}_n$ and $\O_P$ is a simplex.
\item The Coxeter group $W_P$ is large, $\O_P$ is a properly convex open set and the linear group $\G_P$ is strongly irreducible.
\end{enumerate}
\end{theorem}

\par{
We give a short explanation for this theorem since we did not find any proof of it in the literature, although the result is surely known.
}

\begin{propo}\label{infini_opc}
Let $\G$ be an infinite group of $\ss$ acting properly on a convex set $\O$ of $\S^d$. If $\G$ is an irreducible subgroup of $\ss$, then $\O$ is a properly convex open set.
\end{propo}

\begin{proof}
The vector space generated by $\overline{\O} \cap -\overline{\O}$ is preserved by $\G$, therefore either $\O=-\O$ or $\overline{\O} \cap -\overline{\O} = \varnothing$. In the first case, $\O=\S^d$ and $\G$ has to be finite since the action is proper. The second condition means that $\O$ is properly convex.
\end{proof}

\begin{proof}[Proof of Theorem \ref{tri}]
\par{
From Theorem \ref{irred}, we know that $W=W_P$ has to be an irreducible Coxeter group, therefore we have three cases: $W$ can be spherical, affine or large. If $W$ is spherical then Theorem \ref{theo_vinberg} shows that $\O= \O_P = \S^d$. If $W$ is not spherical then $W$ is infinite and Proposition \ref{infini_opc} shows that $\O$ is properly convex. If $W$ is affine then Proposition \ref{5thomy} shows that $W$ is of type $\tilde{A}_n$ and $\O$ is a simplex. Of course, in that case, the linear group $\G=\G_P$ is not strongly irreducible since the vertices of $\O$ have to be fixed by a finite-index subgroup of $\G$.
}
\par{
If $W$ is large, it remains to show that $\G$ is a strongly irreducible subgroup of $\ss$. Suppose the representation is not strongly irreducible, then consider $\G_0$ the Zariski connected component of $\G$. The subgroup $\G_0$ of $\G$ is of finite index. The vector space $\R^{d+1}$ is the sum of the strongly irreducible $\G_0$-submodules $\R^{d+1}= \bigoplus_{i \in I} E_i$
. In particular, $\G_0$ splits as a non-trivial direct product and this is absurd by Theorem \ref{indecompo} below.
}
\end{proof}


\begin{theorem}[(Paris \cite{MR2333366} or Prop 8 of de Cornulier and de la Harpe in \cite{MR2395791})]\label{indecompo} No finite index subgroup of a large irreducible Coxeter group splits as a non-trivial direct product.
\end{theorem}

\section{The setting}

The study of the geometry around a vertex will be crucial in the sequel, so we introduce some definitions.

\subsection{Link of a Coxeter polytope}\label{def_link}

Let $P$ be a Coxeter polytope of $\S^d$ and $p$ be vertex of $P$. The \emph{link $P_p$ of $P$} at $p$ is the set of half-lines starting at $p$ intersecting $P$. It is a Coxeter  polytope included in the projective space $\S \Big(\Quotient{\R^{d+1}}{p^2}\Big)=\S^{d-1}_p$ where $p^2$ is the line generated by the half-line $p$.

To avoid confusion, we get $P_p$ by the following procedure:

\begin{enumerate}
\item Forget all the facets of $P$ not containing $p$, so that you get a convex cone whose summit is $p$;
\item Forget at the same time all the reflections around facets of $P$ not containing $p$, so that you get a Coxeter convex cone whose summit is $p$;
\item Consider the set $P_p$ of half-line starting at $p$ intersecting $P$, look at it in the projective sphere $\S^{d-1}_p$; it is a convex subset, better it is a polytope.
\item The reflections around the facets containing $p$ fix $p$, so they pass to the quotient $\Quotient{\R^{d+1}}{p^2}$ and acts as reflection around the facets of $P_p$.
\end{enumerate}
\par{
Since $P_p$ is a Coxeter  polytope of $\S^{d-1}_p$ we can apply to it the Vinberg-Tits Theorem \ref{theo_vinberg} to get a convex subset $\O_p:=\O_{P_p}$ of $\S^{d-1}_p$. We shall concentrate on a special class of Coxeter  polytope for which this procedure gives a lot of information.
}


\subsection{Quasi-perfect, 2-perfect Coxeter polytopes}

\begin{propo}\label{def_2-perfect}
Let $P$ be a Coxeter  polytope. Then the following are equivalent:
\begin{enumerate}
\item The intersection $P \cap \dO_P$ is finite.
\item The intersection $P \cap \dO_P$ is included in the set of vertices of $P$.
\item For every edge $e$ of $P$, the Coxeter group $W_e$ is finite.
\item For every vertex $p$ of $P$, the Coxeter  polytope $P_p$ is perfect.
\end{enumerate}
In that case, we say that the Coxeter polytope $P$ is \emph{2-perfect}.
\end{propo}

\begin{proof}
Since the convex polytope $P$ is included in the closed convex set $\overline{\O_P}$, the relative interior of an edge of $P$ intersects the boundary $\dO_P$ if and only if it is included in the boundary, so $1) \Leftrightarrow 2)$. The implication $2) \Leftrightarrow 3)$ is a direct consequence of the point $5)$ of Theorem \ref{theo_vinberg}. For $3) \Leftrightarrow 4)$, there is a natural correspondence between the edges of $P$ and the vertices of the link $(P_p)_{p\in \V}$, where $\V$ is the set of vertices of $P$, by definition of $P_p$. The equivalence is then a consequence of Corollary \ref{cocompact}.
\end{proof}

\begin{rem}
Every Coxeter  polygon is 2-perfect and a Coxeter  polytope of dimension 3 is 2-perfect if and only if none of its dihedral angle is 0. 
\end{rem}

In \cite{MR0302779}, Vinberg introduces the notion of quasi-perfect Coxeter polytope. 

\begin{de}
A Coxeter  polytope $P$ is \emph{quasi-perfect} when $P$ is 2-perfect and for every vertex $p$ of $P$, the Coxeter  polytope $P_p$ is either elliptic or parabolic.
\end{de}


\begin{rem}\label{Tits_2-perf}
Let $(S,M)$ be a Coxeter system and $W$ the corresponding Coxeter group.
The Tits simplex $\Delta_W$ is perfect if and only if for every subsystem $S'$ such that $\textrm{Card}(S \smallsetminus S')=1$ we have $W_{S'}$ finite. A large irreducible Coxeter group such that $\Delta_W$ is perfect is usually called a \emph{Lannér} Coxeter group. These groups have been classified by Lannér in \cite{MR0042129}.

The Tits simplex $\Delta_W$ is quasi-perfect if and only if for every subsystem $S'$ such that $\textrm{Card}(S \smallsetminus S')=1$ we have $W_{S'}$ finite or irreducible affine. A large irreducible Coxeter group such that $\Delta_W$ is quasi-perfect is usually called a \emph{quasi-Lannér} Coxeter group (sometimes Koszul Coxeter group). They have been classified by Koszul and Chein in \cite{LectHypCoxGrKoszul,MR0294181}.

Finally, the Tits simplex $\Delta_W$ is 2-perfect if and only if for every subsystem $S'$ such that $\textrm{Card}(S \smallsetminus S')=2$ we have $W_{S'}$ finite. They are sometimes called \emph{Lorentzian} Coxeter groups. They have been classified by Maxwell in \cite{MR679972}, the complete list have been published by Chen and Labbé in \cite{Chen:2013fk}.
\end{rem}

\subsection{A geometric quadrichotomy for quasi-perfect Coxeter polytope}
$\,$
\par{
In \cite{MR0302779}, Vinberg arranges quasi-perfect polytope into four families:
}

\begin{theorem}[(Vinberg, Proposition 26 of \cite{MR0302779})]\label{quadri1}
Let $P$ be a quasi-perfect Coxeter polytope; then $P$ is in one of the following four exclusive cases:
\begin{enumerate}
\item elliptic,
\item parabolic,
\item loxodromic and irreductible or
\item decomposable and not elliptic; in fact $P$ is of the form: $P=Q \otimes \cdot$ where $Q$ is parabolic.
\end{enumerate}
\end{theorem}

\begin{rem}
We stress that in the last case of Theorem \ref{quadri1}, the Coxeter  polytope is not perfect. Hence, if $P$ is perfect and decomposable then $P$ is elliptic.
\end{rem}

\begin{rem}
If a Coxeter  polytope $P$ is parabolic then $P$ is indecomposable even if $W_P$ is not irreducible. Indeed, if $P$ were decomposable then $P=P_1\otimes \cdots \otimes P_r$ where $P_i$ is of dimension $d'_i$. Each $\G_{P_i}$ is virtually isomorphic to $\Z^{d_i}$ with $d'_i \geqslant d_i$ since $\G_{P_i}$ acts properly on $\O_{P_i}$. But,
the group $\G_P$ acts cocompactly on an affine chart of $\S^{d'_1+...+d'_r+r-1}$, so $\Z^{d_1+\cdots +d_r}$ acts cocompactly hence $r=1$ and $d'_i=d_i$.
\end{rem}

\subsection{The final context}$\,$

\par{
 In the case where $P$ is a 2-perfect Coxeter polytope of $\S^d$, all the link $P_p$ are perfect so the convex set $\O_p$ is either the whole space $\S^{d-1}_p$, an affine chart or a properly convex open set from Theorem \ref{quadri1}. We want to understand the geometry of the action of $\G_P$ on $\O_P$ by mean of the shape of the $(\O_p)_{p \in \V}$, where $\V$ is the set of vertices of $P$. 
}
\\
\par{
For a general Coxeter  polytope, we will say that a vertex $p$ is \emph{elliptic} (resp. \emph{parabolic} resp. \emph{loxodromic}) when the Coxeter  polytope $P_p$ is \emph{elliptic} (resp. \emph{parabolic} resp. \emph{loxodromic}). For a \textit{2-perfect Coxeter polytope}, we have a nice trichotomy: every vertex $p$ of $P$ has to be \emph{elliptic} or \emph{parabolic} or \emph{loxodromic}.
}

\begin{rem}
We stress that the word spherical, affine, large, euclidean, irreducible refer to properties of Coxeter groups and the word elliptic, parabolic, loxodromic, indecomposable to properties of Coxeter  polytope, and so of \textit{linear} Coxeter groups.
\end{rem}

\section{The lemmas}

\subsection{The shape of a convex set around a point of its boundary}$\,$

\par{
Let $\O$ be a properly convex open set and $p$ be a point of $\dO$. We say that $p$ is \emph{$\C^1$} when the hypersurface $\dO$ is differentiable at $p$ (if and only if $\O$ admits a unique supporting hyperplane at $p$). We say that $p$ is \emph{not strictly convex} when there exists a non-trivial segment $s \subset \dO$ such that $p\in s$. When $p$ is of class $\C^1$ and strictly convex, we say that $p$ is \emph{round}. A properly convex open set is \emph{round} when every point of its boundary is round.
}
\\
\par{
To study the boundary around a point, the following spaces are very useful. We denote by $\D_p(\O)$ (resp. $\D_p(\overline{\O})$) the space of half-lines starting at $p$ and meeting $\O$ (resp. $\overline{\O}$). These two spaces are convex subsets of $\S^{d-1}_p$. We also a map between these spaces $\S_p:\dO\smallsetminus \{p\} \rightarrow \D_p(\overline{\O})$ given by $\S_p(q)=[pq)$. The point $p$ is $\C^1$ if and only if  $\D_p(\O)$ is an affine chart of $\S^{d-1}_p$. The point $p$ is strictly convex if and only if $\S_p$ is injective. One should remark that $\S_p$ is always onto.
}

\begin{rem}\label{bien_utile}
Let $P$ be a Coxeter  polytope and $p$ a vertex of $P$, then we have $\D_p(\O_P)=\O_{P_p}=\O_p$.
\end{rem}

\subsection{Consequence of ellipticity}

\begin{lemma}\label{petit_sphere}
Let $P$ be a Coxeter polytope of $\S^d$. The action of $\G_P$ on $\S^d$ has no global fixed point.
\end{lemma}

\begin{proof}
The only fixed point of a reflection $\sigma$ are the points of the support of $\sigma$. But, the intersection of the support of all the facets of $P$ is empty.
\end{proof}

\begin{rem}
The last lemma is false in the context of $\PP^d$. The simplest example is the Coxeter  triangle with dihedral angle $(\frac{\pi}{2},\frac{\pi}{2},\frac{\pi}{m})$. The reason for this difference is that in $\PP^d$, a reflection fixes the point of its support and its polar. 
\end{rem}

\begin{propo}\label{propo_sphe}
Let $P$ be an irreducible loxodromic Coxeter  polytope and $p$ a vertex of $P$. The vertex $p$ is elliptic if and only if $p \in \O_P$, and in that case $p \in C(\LGP)$, the open convex hull\footnote{The smallest convex open set that contains $\LGP$ in its closure.} of $\LGP$.
\end{propo}

\begin{proof}
Theorem \ref{theo_vinberg} shows that $p \in \O_P$ if and only if $W_p$ is finite and Theorem \ref{finite} shows that $W_p$ is finite if and only if $\O_p=\S^{d-1}_p$ if and only if $P_p$ is elliptic. So, we only have to prove that $p \in C(\LGP)$. The point $p$ is the unique fixed point of the finite group $W_p$ acting on $\O_P$ since the action of $W_p$ on $\S^{d-1}_p$ has no global fixed point from Lemma \ref{petit_sphere}. 
Hence, the point $p$ belongs to $C(\LGP)$ since the center of mass\footnote{For the existence of a center of mass for every bounded subset of a properly convex subset, we refer the reader to \cite{Marquis:2013fk} lemma 4.2.} of any orbit of a finite group acting on a properly convex open set is a fixed point.
\end{proof}

\subsection{Consequence of parabolicity}

We introduce formally the trick to show that a convex projective manifold is of finite volume. This trick has been used in \cite{Marquis:2009kq,Marquis:2010fk,Cooper:2011fk,Crampon:2012fk}.

\begin{de}
Let $\O$ be a properly convex open subset of $\S^d$ and $p \in \dO$. We say that $\O$ admits \emph{two ellipsoids of security at $p$} when there exist two ellipsoids $\E^{int}$ and $\E^{ext}$ such that $\E^{int} \subset \O \subset \E^{ext}$ and $\partial \E^{int} \cap \dO = \partial \E^{ext}\cap \dO = \{ p \}$ (See  Figure \ref{elli_secu}).
\end{de}

\begin{propo}\label{prop_consq_elli}
Let $\O$ be a properly convex open set  and $p \in \dO$. Let $K$ be a compact subset of $\dO \smallsetminus \{ p \}$ and denote by $\C_{K,p}$ the convex hull of $K \cup \{ p \}$ in $\O$. Suppose that $\O$ admits two ellipsoids of security at $p$. Then $p$ is a round point of $\dO_P$  and for a sufficiently small neighbourhood $\U$ of $p$ in $\S^d$ we have $\mu_{\O}(\C_{K,p} \cap \U) < \infty$.
\end{propo}

\begin{figure}
\centering
\begin{tikzpicture}[line cap=round,line join=round,>=triangle 45,x=1.5cm,y=1.5cm]
\clip(-2,-0.5) rectangle (2.49,3.51);
\draw [domain=-2:2.49] plot(\x,{(-0-0*\x)/5});
\filldraw[draw=black,fill=gray!20][rotate around={90:(0,1.51)},line width=1.6pt,dash pattern=on 1pt off 1pt on 2pt off 4pt] (0,1.51) ellipse (2.27cm and 1.57cm);
\filldraw[draw=black,fill=gray!40][rotate around={90:(0,1.07)},line width=2pt] (0,1.07) ellipse (1.61cm and 1.28cm);
\filldraw[draw=black,fill=gray!60][rotate around={90:(0,0.64)},line width=1.6pt,dash pattern=on 1pt off 1pt on 2pt off 4pt] (0,0.64) ellipse (0.95cm and 0.9cm);
\fill[line width=0.4pt,fill=black,pattern=north east lines] (-0.18,2.12) -- (0.1,2.14) -- (0.33,2.06) -- (0,0) -- (-0.3,2.08) -- cycle;
\draw (-0.3,2.08)-- (0,0);
\draw (0.33,2.06)-- (0,0);
\draw (0.5,0.5)-- (2,0.5);
\draw (0.7,1)-- (2,1);
\draw (0.01,2.15)-- (1.01,3.31);
\draw (-1.5,1.5)-- (0,1.5);
\draw (0.81,2)-- (2,2);
\draw (1.02,3.58) node[anchor=north west] {$K$};
\draw (2,2.2) node[anchor=north west] {$\mathcal{E}^{ext}$};
\draw (2,0.7) node[anchor=north west] {$\mathcal{E}^{int}$};
\draw (2,1.16) node[anchor=north west] {$\Omega$};
\draw (-2.1,1.67) node[anchor=north west] {$\mathcal{C}_{K,p}$};
\draw (-0.1,-0.02) node[anchor=north west] {$p$};
\begin{scriptsize}
\fill [color=black] (0,0) circle (1.0pt);
\fill [color=black] (-0.3,2.08) circle (1.0pt);
\fill [color=black] (0.33,2.06) circle (1.0pt);
\end{scriptsize}
\end{tikzpicture}
\caption{Two ellipsoids of security}
\label{elli_secu}
\end{figure}
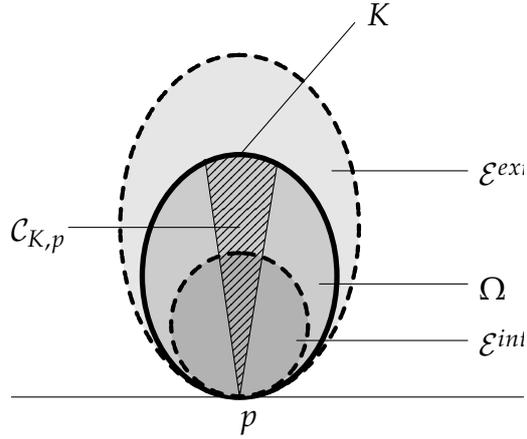

\begin{proof}
The roundness is obvious. For the finiteness of the volume, the claim is true and obvious when $\O$ is an ellipsoid since an ellipsoid endowed with its Hilbert metric is the projective model of hyperbolic geometry. Therefore the existence of $\E^{int}$ via Proposition \ref{compa} implies the proposition.
\end{proof}

The goal of this paragraph is to show the following proposition:

\begin{propo}\label{propo_para}
Let $P$ be an irreducible loxodromic Coxeter  polytope and $p$ be a vertex of $P$. If the vertex $p$ is parabolic then:
\begin{enumerate}
\item The point $p$ is a round point of $\dO_P$.
\item The convex set $\O_P$ admits two ellipsoids of security at $p$, which are preserved by $\G_p$.
\item There exists a neighbourhood $\U$ of $p$ in $\S^d$ such that $\mu_{\O_P}(P \cap \U) < \infty$.
\item $p \in \LGP$.
\end{enumerate}
\end{propo}

\par{
The following lemma is due to Vinberg:
}

\begin{lemma}[(Vinberg, Proposition 23 of \cite{MR0302779})]\label{aff_euc}
Let $P$ be a parabolic Coxeter  polytope of $\S^d$. Then $\G_P$ acts by euclidean transformation on the affine chart $\O_{P}$ (i.e. there exists a positive definite scalar product on $\O_P$ preserved by $\G_P$).
\end{lemma}

An avatar of the following lemma can be find in \cite{Cooper:2011fk,Marquis:2009kq,Marquis:2010fk,Crampon:2012fk}. In fact in  \cite{Cooper:2011fk} and \cite{Crampon:2012fk}, the reader can find a proof without the second hypothesis. We give a proof with this hypothesis for the convenience of the reader.

\begin{lemma}\label{exis_elli_secu}
Let $\G$ be a discrete subgroup of $\s{d+1}$ preserving a properly convex open set $\O$ of $\S^d$. Let $p$ be a point of $\dO$ and let $\G_p=\{ \g \in \G \,|\, \g(p) = p \}$. Suppose that:
\begin{enumerate}
\item The subgroup $\G_p$ acts cocompactly on an affine chart $\A^{d-1}$ of $\S^{d-1}_p$ and,
\item The action of $\G_p$ on $\A^{d-1}$ is by euclidean transformation.
\end{enumerate}
 then $\O$ admits two ellipsoids of security at $p$ which are preserved by $\G_p$.
\end{lemma}

It will be convenient for the proof to call the boundary of an ellipsoid an \emph{ellisphere}.


\begin{proof}
\par{
The action of $\G_p$ on $\A^{d-1}$ is by euclidean transformation, therefore the action of $\G_p$ on $\R^{d+1}$ preserves a quadratic form of signature $(d,1)$; in other words preserves an ellipsoid $\E$ such that $p\in \partial \E$. Also, there exists a convex compact fundamental domain $D$ for the action of $\G_p$ on the affine chart $\A^{d-1}$. We denote by $\C_p$ the cone of vertex $p$ and basis $D$ (see Figure \ref{parabolicness_fig}). 
}


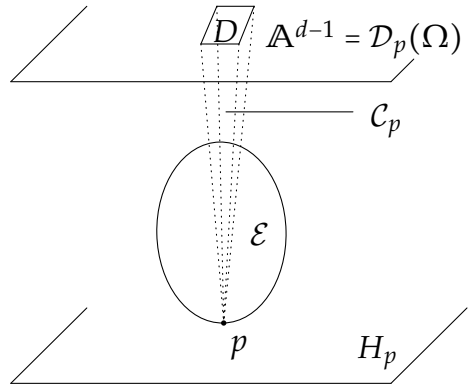
\begin{figure}[!ht]
\centering
\begin{tikzpicture}[line cap=round,line join=round,>=triangle 45,x=1.0cm,y=1.0cm]
\clip(-0.52,-0.49) rectangle (6.35,5.62);
\draw (1,5)-- (0,4);
\draw (0,4)-- (5,4);
\draw (5,4)-- (5.3,4.3);
\draw (1,1)-- (0,0);
\draw (0,0)-- (5,0);
\draw (5,0)-- (6,1);
\draw (2.71,5)-- (3.2,5);
\draw (3.2,5)-- (3,4.5);
\draw (2.5,4.5)-- (2.71,5);
\draw (2.5,4.5)-- (3,4.5);
\draw [dotted] (2.5,4.5)-- (2.8,0.8);
\draw [dotted] (3,4.5)-- (2.8,0.8);
\draw [dotted] (2.71,5)-- (2.8,0.8);
\draw [dotted] (3.2,5)-- (2.8,0.8);
\draw [rotate around={-88.72:(2.77,2)}] (2.77,2) ellipse (1.2cm and 0.85cm);
\draw (2.5,4.97) node[anchor=north west] {$D$};
\draw (3.2,4.97) node[anchor=north west] {$\A^{d-1} = \mathcal{D}_p(\Omega)$};
\draw (2.83,3.6)-- (4.5,3.6);
\draw (4.59,3.85) node[anchor=north west] {$\mathcal{C}_p$};
\draw (4.4,0.77) node[anchor=north west] {$H_p$};
\draw (2.73,0.78) node[anchor=north west] {$p$};
\draw (3,2.27) node[anchor=north west] {$\mathcal{E}$};
\begin{scriptsize}
\fill [color=black] (2.8,0.8) circle (1.0pt);
\end{scriptsize}
\end{tikzpicture}
\caption{The action of a parabolic subgroup}
\label{parabolicness_fig}
\end{figure}
\par{
Since, $\G_p$ acts cocompactly on an affine chart $\A^{d-1}$, we get that $\D_p(\O)=\A^{d-1}$, hence $\O$ has a unique supporting hyperplane $H_p$ at $p$. The pencil $\mathcal{F}$ of ellisphere generated by $\partial \E$ and $H_p$ gives a one parameter family of ellipsoids preserved by $\G_p$. Moreover the intersection of any ellisphere of the pencil $\F$ with $\C_p$ is compact since $\G_p$ acts cocompactly on $\D_p(\O)$.
}
\par{
Therefore to find an ellipsoid $\E^{int}$ (resp. $\E^{ext}$) inside (resp. outside) $\O$ preserved by $\G_p$, it is sufficient to see that any ellisphere $\partial \E'$ of the pencil $\mathcal{F}$ which is sufficiently close (resp. far) from $p$ verifies: $\E' \cap \C_p \subset \O$ (resp.  $\O \cap \C_p \subset \E'$). 
}
\par{
Hence, $\O$ admits two ellipsoids of security at $p$ which are preserved by $\G_p$. Thanks to Proposition \ref{prop_consq_elli}, the point $p$ is round.
}
\end{proof}

\begin{proof}[Proof of Proposition \ref{propo_para}]
\par{
The point $p$ is parabolic, so the Coxeter polytope $P_p$ is perfect, preserves an affine chart of $\S^{d-1}_p$ and acts compactly by euclidean transformations on it by Lemma \ref{aff_euc}. Hence, Lemma \ref{exis_elli_secu} shows the second point.
}
\par{
Proposition \ref{prop_consq_elli} shows that the second point implies the first and third one. The last point is a trivial consequence of the fact that for any infinite order element $\g$ of $\G_p$ and for all point $x \in \overline{\O_P}$, we have $\g^n(x) \to p$.
}
\end{proof}

\subsection{A lemma about negative type Coxeter polytope}

\begin{lemma}\label{exis_affi_chart_polar}
Let $P$ be Coxeter polyedron of $\S^d$. If $P$ is of negative type then there exists an affine chart of $\S^d$ containing $P$ and all its polars.
\end{lemma}

\begin{proof}
We can assume that $P$ is indecomposable. By Perron-Frobenius theorem and the definition of being of negative type there exists a strictly positive vector $\mu \in (\R_+)^S$ and a real $\lambda < 0$ such that $A_P \mu = \lambda \mu$. So, if we take $\alpha = \sum_{s\in S} \mu_s \alpha_s$, then for each $t\in S$, we get that $\alpha(v_t) = \lambda \mu_t < 0$. This implies that the affine chart $\A := \{ x \in \S \,\mid\, \alpha(x) < 0 \}$ contains all the polar of $P$ in its closure. Moreover, since $P= \{x \in P \,\mid \, \forall s \in S,\, \alpha_s(x) \leqslant 0 \}$, we get that $P \subset \A$.
\end{proof}

\subsection{Truncability}\label{label_truncation}

\subsubsection{Definition of truncability}

\begin{de}
Let $p$ be a vertex of a Coxeter polytope $P$ of $\S^d$ and $S_p$ the set of facets of $P$ containing $p$. The vertex $p$ of a Coxeter polytope is \emph{truncable} when the projective subspace $\Pi_p$ spanned by the $[v_s]$ for $s \in S_p$:
\begin{enumerate}
\item is a hyperplane, 
\item meets the interior of $P$,
\item a ridge $e$ of $P$ verifies $e \cap \Pi_p \neq \varnothing$ if and only if $p \in e$ and $e \cap \Pi_p$ have to be included in the relative interior of $e$. 
\end{enumerate}

We will denote by $\Pi_p^+$ (resp.  $\Pi_p^-$) the connected component of $\S^d \smallsetminus \Pi_p$ which does not contain $p$ (resp. containing $p$). We stress that $\Pi_p^+$ and $\Pi_p^-$ are affine charts. See Figure \ref{illus_trunca}.
\end{de}

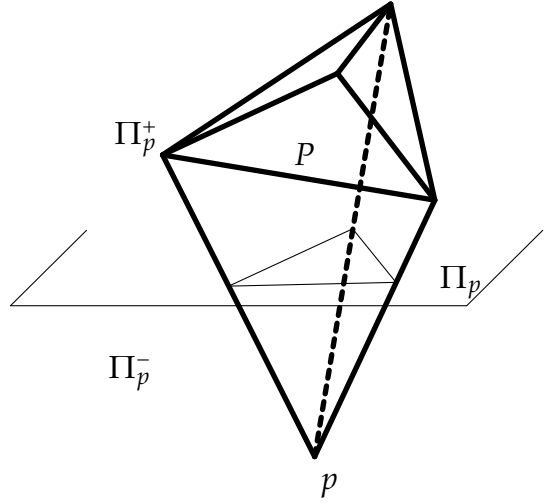
\begin{figure}
\centering
\begin{tikzpicture}[line cap=round,line join=round,>=triangle 45,x=1.0cm,y=1.0cm]
\clip(-4.45,-0.57) rectangle (3.5,6.21);
\draw [line width=2pt] (-2,4)-- (0,0);
\draw [line width=2pt] (1.58,3.4)-- (0,0);
\draw [line width=2pt,dashed] (0,0)-- (1,6);
\draw [line width=2pt] (1,6)-- (-2,4);
\draw [line width=2pt] (-2,4)-- (1.58,3.4);
\draw [line width=2pt] (1.58,3.4)-- (1,6);
\draw [line width=2pt] (1,6)-- (0.3,5.08);
\draw [line width=2pt] (0.3,5.08)-- (1.58,3.4);
\draw [line width=2pt] (0.3,5.08)-- (-2,4);
\draw (-3,3)-- (-4,2);
\draw (2,2)-- (-4,2);
\draw (2,2)-- (3,3);
\draw (-0.08,-0.08) node[anchor=north west] {$p$};
\draw (1.5,2.68) node[anchor=north west] {$\Pi_p$};
\draw (-2.78,4.6) node[anchor=north west] {$\Pi_p^+$};
\draw (-2.86,1.5) node[anchor=north west] {$\Pi_p^-$};
\draw (-0.42,4.3) node[anchor=north west] {$P$};
\draw (-1.13,2.26)-- (0.5,3.02);
\draw (0.5,3.02)-- (1.07,2.31);
\draw (1.07,2.31)-- (-1.13,2.26);
\begin{scriptsize}
\fill [color=black] (0,0) circle (1.0pt);
\end{scriptsize}
\end{tikzpicture}
\caption{Illustration of Truncability of $p$}
\label{illus_trunca}
\end{figure}

\begin{rem}[(Consequence)]
Suppose $P$ is a Coxeter polytope and $p$ is a truncable vertex. We can define a new polytope $P^{\dagger p}$. The facets of $P^{\dagger p}$ are the facets of $P$  (we call them the \emph{old facets}) plus one extra facet defined by the hyperplane $\Pi_p$ (called the \emph{new facet}). The polar of the old facets are unchanged and the polar of the new facet is $p$. Therefore, it is easy to check that the polytope $P^{\dagger p}$ has the following property:
\begin{itemize}
\item The dihedral angles of the new ridges are $\frac{\pi}{2}$.
\item The vertices of the new facet are all elliptic if and only if $P_p$ is perfect.
\item The hyperplane $\Pi_p$ is preserved by the reflection across the facets containing $p$. The intersection $P \cap \Pi_p$ is a Coxeter  polytope of $\Pi_p$ isomorphic to $P_p$.
\end{itemize}
\end{rem}

\begin{rem}
This construction was already known in the hyperbolic space. See for example the survey \cite{MR783604} of Vinberg, Proposition 4.4. This construction in the projective context and in dimension $3$ is the main ingredient of \cite{MR2660566}.
\end{rem}

\subsubsection{Simple perfect loxodromic vertex are truncable}

\begin{propo}\label{propo_con}
Let $P$ be a Coxeter  polytope of $\S^d$ and $p$ be a vertex of $P$. If the vertex $p$ is simple perfect and loxodromic then $P$ is truncable at $p$ except if $P$ is isomorphic to $P_p \otimes \cdot$.
\end{propo}

\begin{proof}
\par{
Thanks to Lemma \ref{exis_affi_chart_polar}, we can think of everything inside an affine chart that contains $P$ and its polars. The simplicity (resp. loxodrominess) of $p$ implies that the projective space $\Pi_p$ is of dimension at most (resp. at least since the rank of $A_p$ is $d$) $d-1$.
}
\par{
We first look at the half-line $[p v_i)$, for $i \in S_p$. Lemma \ref{exis_affi_chart_polar} applied to $P_p$ gives the existence of a hyperplane $H_p$ of $\S^d$ that contains $p$ and such that for each $i \in S_p$, the open segment $]p v_i[$ is included in the connected component $H_p^-$ of $\S^d \smallsetminus H_p$ that contains the interior of $P$. 
}
\par{
Next, we look at the repartition of this half-line around $\O_P$. Lemma \ref{LemNie} shows that $\O_p$ (and so $P_p$ also) is included in the convex hull of the $\S_p(v_i)$ for $i \in S_p$. Hence, the convex set $\O_P$ is included in the convex hull of this half-lines. Roughly speaking, the $v_i$ for $i \in S_p$ are ''all around'' $P$ and $\O_P$.
}
\par{ 
Finally, we show that the $v_i$, for $i \in S_p$ cannot be too ''far'' from $p$. Let $F_t^-$ denote the component of $\S^d \smallsetminus F_t$ that contains $p$, where $F_t$ is the hyperplane generated by the facet $t \notin S_p$. Geometrically, the inequalities (C) mean that $v_i \in F_t^-$ except when the angle between the face $i$ and $t$ is $\frac{\pi}{2}$; in that case, $v_i \in F_t$. 
So we know that for each facet $t \notin S_p$ and for each $i \in S_p$ the point $v_i$ is on the segment $[p v_i) \cap \overline{F_t^-}$.
}
\par{
So $\Pi_p \cap P \neq \varnothing$ and $p \notin \Pi_p$. For the sake of clarity, we need to distinguish two cases: a) $P$ is not a cone of summit $p$ and b) is not. If we are in the case b), the inequalities $(C)$ (using the several facets of $P$ not in $S_p$) show that any facet $f$ of $P$ which intersects $\Pi_p$ has to contain $p$, and the intersection $f \cap \Pi_p$ is included in the relative interior of $f$. Now, if we are in case a) then the inequalities $(C)$ show that either $\Pi_p$ is the support of the face of $P$ not in $S_p$, in which case $P$ is isomorphic to $P_p\otimes \cdot$, or any facet $f$ of $P$ which intersects $\Pi_p$ has to contain $p$, and the intersection $f \cap \Pi_p$ is included in the relative interior of $f$.
}
\end{proof}

\begin{lemma}[(Nie, Lemma 3 of \cite{Nie:2011fk})]\label{LemNie}
Let $P$ be a perfect loxodromic simplex. The convex set $\O_P$ is included in the convex hull of its polar.
\end{lemma}

%
%

\begin{rem}
A more careful analysis of the situation would show that if $P$ is an indecomposable Coxeter  polytope of $\S^d$ and $p$ is a simple perfect vertex of $P$, then, $p$ is elliptic if and only if $\O_P \cap \Pi_p = \varnothing$, $p$ is parabolic if and only if $\O_P \cap \Pi_p = \{ p \}$ and $p$ is loxodromic if and only if $p$ is truncable. 
\end{rem}

\subsubsection{Iteration of truncation}

\begin{lemma}\label{Not_meet}
Let $P$ be a loxodromic Coxeter  polytope and $p,q$ two vertices of $P$. Suppose that $p$ and $q$ are perfect simple loxodromic vertices, denote by $f_p$ (resp. $f_q$) the facets obtained by truncation of $P$ at $p$ (resp. $q$), then the facets $f_p$ and $f_q$ do not meet.
\end{lemma}

\begin{proof}
Let $\pi_p$ (resp. $\pi_q$) be the intersection of the projective space spanned by $f_p$ (resp. $f_q$) and $\overline{\O_P}$. Since $f_p \subset \pi_p$ and $f_q \subset \pi_q$, this lemma is a consequence of the fact that $\pi_p \cap \pi_q$ is included in $\dO_P$. Since $p$ is perfect, the $f_p$ is included in the relative interior of $\pi_p$ (by Corollary \ref{cocompact}).

Let us now prove this fact (Figure \ref{ecima_mult} can be useful). Choose an affine chart $\A$ containing $\overline{\O_P}$, denote by $\C_p$ (resp. $\C_q$) the cone of summit $p$ (resp. $q$) and basis $\pi_p$ (resp. $\pi_q$) and by $\hat{\C}_p$ (resp. $\hat{\C}_q$) the cone of summit $p$ generated by $\pi_p$ (resp. $\pi_q$) in the affine chart $\A$. We remark that $\overline{\O_P}$ contains the cones $\C_p$ and $\C_q$ and is contained in $\hat{\C}_p \cap \hat{\C}_q$. Since $\O_P$ is convex, this is possible only when $\pi_p \cap \pi_q$ is included in $\dO_P$.
\end{proof}

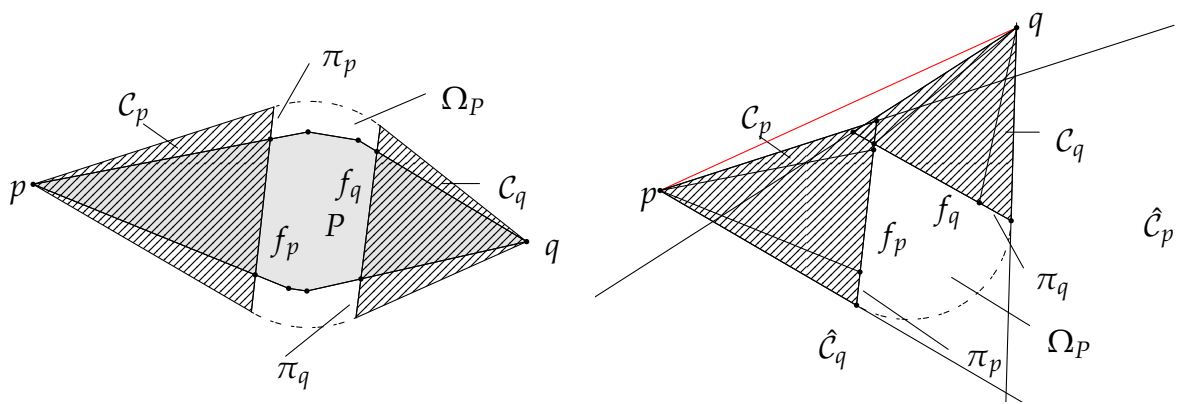
\begin{figure}
\centering
\begin{tikzpicture}[line cap=round,line join=round,>=triangle 45,x=0.5cm,y=0.5cm]
\clip(-8.58,-4.96) rectangle (7.1,4.84);
\fill[fill=black,pattern=north east lines] (-7.26,0.8) -- (-0.93,2.85) -- (-1.52,-2.58) -- cycle;
\fill[fill=black,pattern=north east lines] (1.87,2.35) -- (5.72,-0.72) -- (1.23,-2.74) -- cycle;
\fill[fill=black,fill opacity=0.1] (-7.26,0.8) -- (-1.02,2) -- (-0.03,2.19) -- (1.29,1.97) -- (5.72,-0.72) -- (-0.06,-2.03) -- (-0.54,-1.96) -- cycle;
\draw (-0.93,2.85)-- (-1.52,-2.58);
\draw (1.87,2.35)-- (1.23,-2.74);
\draw (-7.26,0.8)-- (-0.93,2.85);
\draw (-0.93,2.85)-- (-1.52,-2.58);
\draw (-1.52,-2.58)-- (-7.26,0.8);
\draw (1.87,2.35)-- (5.72,-0.72);
\draw (5.72,-0.72)-- (1.23,-2.74);
\draw (1.23,-2.74)-- (1.87,2.35);
\draw (-7.26,0.8)-- (-1.02,2);
\draw (-7.26,0.8)-- (-1.42,-1.6);
\draw (1.78,1.67)-- (5.72,-0.72);
\draw (1.36,-1.71)-- (5.72,-0.72);
\draw (-1.02,2)-- (-0.03,2.19);
\draw (-0.03,2.19)-- (1.29,1.97);
\draw (1.29,1.97)-- (1.78,1.67);
\draw (1.36,-1.71)-- (-0.06,-2.03);
\draw (-0.06,-2.03)-- (-0.54,-1.96);
\draw (-0.54,-1.96)-- (-1.42,-1.6);
\draw (-7.26,0.8)-- (-1.02,2);
\draw (-1.02,2)-- (-0.03,2.19);
\draw (-0.03,2.19)-- (1.29,1.97);
\draw (1.29,1.97)-- (5.72,-0.72);
\draw (5.72,-0.72)-- (-0.06,-2.03);
\draw (-0.06,-2.03)-- (-0.54,-1.96);
\draw (-0.54,-1.96)-- (-7.26,0.8);
\draw (-5.3,3.5) node[anchor=north west] {$\mathcal{C}_p$};
\draw (4.7,1.3) node[anchor=north west] {$\mathcal{C}_q$};
\draw (0.1,0.3) node[anchor=north west] {$P$};
\draw (-8.2,1.2) node[anchor=north west] {$p$};
\draw (5.9,-0.38) node[anchor=north west] {$q$};
\draw (0.4,1.5) node[anchor=north west] {$f_q$};
\draw (-1.3,0) node[anchor=north west] {$f_p$};
\draw (-0.82,2.56)-- (0,4);
\draw (1.06,-2.4)-- (-0.62,-3.76);
\draw (0,4.64) node[anchor=north west] {$\pi_p$};
\draw (-1.18,-3.62) node[anchor=north west] {$\pi_q$};
\draw [shift={(0,0)},dash pattern=on 1pt off 1pt on 2pt off 4pt]  plot[domain=0.9:1.88,variable=\t]({1*3*cos(\t r)+0*3*sin(\t r)},{0*3*cos(\t r)+1*3*sin(\t r)});
\draw [shift={(0,0)},dash pattern=on 1pt off 1pt on 2pt off 4pt]  plot[domain=4.18:5.13,variable=\t]({1*3*cos(\t r)+0*3*sin(\t r)},{0*3*cos(\t r)+1*3*sin(\t r)});
\draw (1.2,2.4)-- (3.06,3.2);
\draw (3.18,3.68) node[anchor=north west] {$\Omega_P$};
\draw (-3.46,1.76)-- (-4.36,2.4);
\draw (3.5,0.82)-- (4.48,0.82);
\begin{scriptsize}
\fill [color=black] (-7.26,0.8) circle (1.0pt);
\fill [color=black] (5.72,-0.72) circle (1.0pt);
\fill [color=black] (-1.02,2) circle (1.0pt);
\fill [color=black] (-1.42,-1.6) circle (1.0pt);
\fill [color=black] (1.78,1.67) circle (1.0pt);
\fill [color=black] (1.36,-1.71) circle (1.0pt);
\fill [color=black] (-0.54,-1.96) circle (1.0pt);
\fill [color=black] (-0.06,-2.03) circle (1.0pt);
\fill [color=black] (-0.03,2.19) circle (1.0pt);
\fill [color=black] (1.29,1.97) circle (1.0pt);
\end{scriptsize}
\end{tikzpicture}
\definecolor{ffqqqq}{rgb}{1,0,0}
\begin{tikzpicture}[line cap=round,line join=round,>=triangle 45,x=0.45cm,y=0.45cm]
\clip(-9.16,-5.42) rectangle (7.76,6.82);
\fill[fill=black,pattern=north east lines] (-7.26,0.8) -- (-0.93,2.85) -- (-1.52,-2.58) -- cycle;
\fill[fill=black,pattern=north east lines] (-1.62,2.52) -- (3.16,5.6) -- (3,-0.09) -- cycle;
\draw (-0.93,2.85)-- (-1.52,-2.58);
\draw (-1.62,2.52)-- (3,-0.09);
\draw (-7.26,0.8)-- (-0.93,2.85);
\draw (-0.93,2.85)-- (-1.52,-2.58);
\draw (-1.52,-2.58)-- (-7.26,0.8);
\draw (-1.62,2.52)-- (3.16,5.6);
\draw (3.16,5.6)-- (3,-0.09);
\draw (3,-0.09)-- (-1.62,2.52);
\draw (-7.26,0.8)-- (-1.02,2);
\draw (-7.26,0.8)-- (-1.42,-1.6);
\draw (-1.01,2.18)-- (3.16,5.6);
\draw (2.07,0.44)-- (3.16,5.6);
\draw (-5.2,3.6) node[anchor=north west] {$\mathcal{C}_p$};
\draw (4,2.94) node[anchor=north west] {$\mathcal{C}_q$};
\draw (-8.2,1.2) node[anchor=north west] {$p$};
\draw (3.18,6.36) node[anchor=north west] {$q$};
\draw (0.42,1) node[anchor=north west] {$f_q$};
\draw (-1.2,0.2) node[anchor=north west] {$f_p$};
\draw (1.16,-3.88)-- (-1.32,-2.16);
\draw (1.26,-1.76)-- (3.54,-3.18);
\draw (1.34,-3.58) node[anchor=north west] {$\pi_p$};
\draw [shift={(0,0)},dash pattern=on 1pt off 1pt on 2pt off 4pt]  plot[domain=4.18:6.25,variable=\t]({1*3*cos(\t r)+0*3*sin(\t r)},{0*3*cos(\t r)+1*3*sin(\t r)});
\draw (2.5,0.08)-- (3.36,-1.48);
\draw (3.72,-2.98) node[anchor=north west] {$\Omega_P$};
\draw (-3.46,1.76)-- (-4.36,2.4);
\draw (2.84,2.52)-- (3.82,2.52);
\draw [domain=-7.26:7.759999999999976] plot(\x,{(-19.98-3.38*\x)/5.74});
\draw [domain=-9.16000000000003:3.16] plot(\x,{(--17.06-5.69*\x)/-0.16});
\draw [color=ffqqqq] (-7.26,0.8)-- (3.16,5.6);
\draw [domain=-7.26:7.759999999999976] plot(\x,{(--19.98--2.05*\x)/6.33});
\draw [domain=-9.16000000000003:3.16] plot(\x,{(-17.06-3.08*\x)/-4.78});
\draw (-2.94,-2.98) node[anchor=north west] {$\hat{\mathcal{C}}_q$};
\draw (6.56,0.62) node[anchor=north west] {$\hat{\mathcal{C}}_p$};
\draw (3.36,-1.32) node[anchor=north west] {$\pi_q$};
\begin{scriptsize}
\fill [color=black] (-7.26,0.8) circle (1.0pt);
\fill [color=black] (3.16,5.6) circle (1.0pt);
\fill [color=black] (-0.93,2.85) circle (1.0pt);
\fill [color=black] (-1.52,-2.58) circle (1.0pt);
\fill [color=black] (-1.62,2.52) circle (1.0pt);
\fill [color=black] (3,-0.09) circle (1.0pt);
\fill [color=black] (-1.02,2) circle (1.0pt);
\fill [color=black] (-1.42,-1.6) circle (1.0pt);
\fill [color=black] (-1.01,2.18) circle (1.0pt);
\fill [color=black] (2.07,0.44) circle (1.0pt);
\end{scriptsize}
\end{tikzpicture}
\caption{A possible situation and an impossible situation}
\label{ecima_mult}
\end{figure}

\par{
The last lemma shows that given a loxodromic Coxeter  polytope, if we denote by $\mathcal{L}^{sp}$ the set of simple perfect loxodromic vertices, then we can define a new Coxeter  polytope $P^{\dagger}$ which is obtained from $P$ by doing the truncation around every vertex $p \in \mathcal{L}^{sp}$. We will call it the \emph{truncated Coxeter polytope} of $P$ and we will use the notation $P^{\dagger}$ to represent it. We will call \emph{old} (resp. \emph{new}) the vertices, edges, facets, ridges of $P^{\dagger}$ that were (resp. were not) in $P$. Figure \ref{Tiling_loxo} illustrates the situation.
}

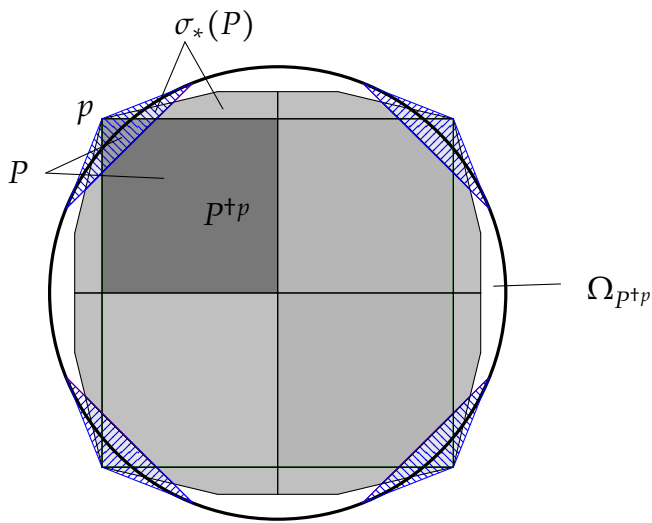
\begin{figure}
\centering
\definecolor{ffqqqq}{rgb}{0,0,1}
\definecolor{ffxfqq}{rgb}{0,1,0}
\definecolor{qqqqff}{rgb}{1,0,0}
\begin{tikzpicture}[line cap=round,line join=round,>=triangle 45,x=1.0cm,y=1.0cm]
\clip(-3.66,-3.49) rectangle (6.2,3.88);
\fill[dash pattern=on 1pt off 1pt on 2pt off 4pt,fill=black,fill opacity=0.53] (-2.31,0) -- (0,0) -- (0,2.31) -- (-1.6,2.31) -- (-2.31,1.6) -- cycle;
\fill[dash pattern=on 1pt off 1pt on 2pt off 4pt,fill=black,fill opacity=0.25] (-1.6,2.31) -- (0,2.31) -- (0,2.67) -- (-0.79,2.67) -- (-1.37,2.53) -- cycle;
\fill[dash pattern=on 1pt off 1pt on 2pt off 4pt,fill=black,fill opacity=0.29] (2.31,0) -- (0,0) -- (0,2.31) -- (1.6,2.31) -- (2.31,1.6) -- cycle;
\fill[dash pattern=on 1pt off 1pt on 2pt off 4pt,fill=black,fill opacity=0.29] (2.31,0) -- (0,0) -- (0,-2.31) -- (1.6,-2.31) -- (2.31,-1.6) -- cycle;
\fill[dash pattern=on 1pt off 1pt on 2pt off 4pt,fill=black,fill opacity=0.25] (-2.31,0) -- (0,0) -- (0,-2.31) -- (-1.6,-2.31) -- (-2.31,-1.6) -- cycle;
\fill[dash pattern=on 1pt off 1pt on 2pt off 4pt,fill=black,fill opacity=0.25] (1.6,2.31) -- (0,2.31) -- (0,2.67) -- (0.79,2.67) -- (1.37,2.53) -- cycle;
\fill[dash pattern=on 1pt off 1pt on 2pt off 4pt,fill=black,fill opacity=0.25] (1.6,-2.31) -- (0,-2.31) -- (0,-2.67) -- (0.79,-2.67) -- (1.37,-2.53) -- cycle;
\fill[dash pattern=on 1pt off 1pt on 2pt off 4pt,fill=black,fill opacity=0.25] (-1.6,-2.31) -- (0,-2.31) -- (0,-2.67) -- (-0.79,-2.67) -- (-1.37,-2.53) -- cycle;
\fill[dash pattern=on 1pt off 1pt on 2pt off 4pt,fill=black,fill opacity=0.25] (-2.31,1.6) -- (-2.31,0) -- (-2.67,0) -- (-2.67,0.79) -- (-2.53,1.37) -- cycle;
\fill[dash pattern=on 1pt off 1pt on 2pt off 4pt,fill=black,fill opacity=0.25] (-2.31,-1.6) -- (-2.31,0) -- (-2.67,0) -- (-2.67,-0.79) -- (-2.53,-1.37) -- cycle;
\fill[dash pattern=on 1pt off 1pt on 2pt off 4pt,fill=black,fill opacity=0.25] (2.31,-1.6) -- (2.31,0) -- (2.67,0) -- (2.67,-0.79) -- (2.53,-1.37) -- cycle;
\fill[dash pattern=on 1pt off 1pt on 2pt off 4pt,fill=black,fill opacity=0.25] (2.31,1.6) -- (2.31,0) -- (2.67,0) -- (2.67,0.79) -- (2.53,1.37) -- cycle;
\fill[dash pattern=on 1pt off 1pt on 2pt off 4pt,fill=black,fill opacity=0.38] (-2.31,2.31) -- (-2.31,1.6) -- (-1.6,2.31) -- cycle;
\fill[dash pattern=on 1pt off 1pt on 2pt off 4pt,fill=black,fill opacity=0.1] (-2.31,2.31) -- (-1.6,2.31) -- (-1.37,2.53) -- cycle;
\fill[dash pattern=on 1pt off 1pt on 2pt off 4pt,fill=black,fill opacity=0.1] (-2.31,2.31) -- (-2.31,1.6) -- (-2.53,1.37) -- cycle;
\fill[fill=black,fill opacity=0.1] (2.31,2.31) -- (2.31,1.6) -- (1.6,2.31) -- cycle;
\fill[fill=black,fill opacity=0.1] (2.31,2.31) -- (2.31,1.6) -- (2.53,1.37) -- cycle;
\fill[fill=black,fill opacity=0.1] (2.31,2.31) -- (1.6,2.31) -- (1.37,2.53) -- cycle;
\fill[fill=black,fill opacity=0.1] (2.31,-2.31) -- (2.31,-1.6) -- (1.6,-2.31) -- cycle;
\fill[fill=black,fill opacity=0.1] (2.31,-2.31) -- (1.6,-2.31) -- (1.37,-2.53) -- cycle;
\fill[fill=black,fill opacity=0.1] (2.31,-2.31) -- (2.31,-1.6) -- (2.53,-1.37) -- cycle;
\fill[dash pattern=on 1pt off 1pt on 2pt off 4pt,fill=black,fill opacity=0.1] (-2.31,-2.31) -- (-2.31,-1.6) -- (-1.6,-2.31) -- cycle;
\fill[dash pattern=on 1pt off 1pt on 2pt off 4pt,fill=black,fill opacity=0.1] (-2.31,-2.31) -- (-2.31,-1.6) -- (-2.53,-1.37) -- cycle;
\fill[dash pattern=on 1pt off 1pt on 2pt off 4pt,fill=black,fill opacity=0.1] (-2.31,-2.31) -- (-1.6,-2.31) -- (-1.37,-2.53) -- cycle;
\fill[pattern color=ffqqqq,fill=ffqqqq,pattern=north west lines] (-2.31,2.31) -- (-2.78,1.12) -- (-1.12,2.78) -- cycle;
\fill[pattern color=ffqqqq,fill=ffqqqq,pattern=north east lines] (-2.31,-2.31) -- (-1.12,-2.78) -- (-2.78,-1.12) -- cycle;
\fill[pattern color=ffqqqq,fill=ffqqqq,pattern=north west lines] (2.31,-2.31) -- (1.12,-2.78) -- (2.78,-1.12) -- cycle;
\fill[pattern color=ffqqqq,fill=ffqqqq,pattern=north east lines] (2.31,2.31) -- (2.78,1.12) -- (1.12,2.78) -- cycle;
\draw [line width=1.2pt] (0,0) circle (3cm);
\draw [color=qqqqff] (-2.78,1.12)-- (-1.12,2.78);
\draw [color=qqqqff] (1.12,2.78)-- (2.78,1.12);
\draw [color=qqqqff] (2.78,-1.12)-- (1.12,-2.78);
\draw [color=qqqqff] (-1.12,-2.78)-- (-2.78,-1.12);
\draw [color=ffxfqq] (-2.31,2.31)-- (-2.31,-2.31);
\draw [color=ffxfqq] (-2.31,-2.31)-- (2.31,-2.31);
\draw [color=ffxfqq] (2.31,-2.31)-- (2.31,2.31);
\draw [color=ffxfqq] (2.31,2.31)-- (-2.31,2.31);
\draw (-2.31,0)-- (0,0);
\draw (0,0)-- (0,2.31);
\draw (0,2.31)-- (-1.6,2.31);
\draw (-1.6,2.31)-- (-2.31,1.6);
\draw (-2.31,1.6)-- (-2.31,0);
\draw (-1.6,2.31)-- (0,2.31);
\draw (0,2.31)-- (0,2.67);
\draw (0,2.67)-- (-0.79,2.67);
\draw (-0.79,2.67)-- (-1.37,2.53);
\draw (-1.37,2.53)-- (-1.6,2.31);
\draw (2.31,0)-- (0,0);
\draw (0,0)-- (0,2.31);
\draw (0,2.31)-- (1.6,2.31);
\draw (1.6,2.31)-- (2.31,1.6);
\draw (2.31,1.6)-- (2.31,0);
\draw (2.31,0)-- (0,0);
\draw (0,0)-- (0,-2.31);
\draw (0,-2.31)-- (1.6,-2.31);
\draw (1.6,-2.31)-- (2.31,-1.6);
\draw (2.31,-1.6)-- (2.31,0);
\draw (-2.31,0)-- (0,0);
\draw (0,0)-- (0,-2.31);
\draw (0,-2.31)-- (-1.6,-2.31);
\draw (-1.6,-2.31)-- (-2.31,-1.6);
\draw (-2.31,-1.6)-- (-2.31,0);
\draw (1.6,2.31)-- (0,2.31);
\draw (0,2.31)-- (0,2.67);
\draw (0,2.67)-- (0.79,2.67);
\draw (0.79,2.67)-- (1.37,2.53);
\draw (1.37,2.53)-- (1.6,2.31);
\draw (1.6,-2.31)-- (0,-2.31);
\draw (0,-2.31)-- (0,-2.67);
\draw (0,-2.67)-- (0.79,-2.67);
\draw (0.79,-2.67)-- (1.37,-2.53);
\draw (1.37,-2.53)-- (1.6,-2.31);
\draw (-1.6,-2.31)-- (0,-2.31);
\draw (0,-2.31)-- (0,-2.67);
\draw (0,-2.67)-- (-0.79,-2.67);
\draw (-0.79,-2.67)-- (-1.37,-2.53);
\draw (-1.37,-2.53)-- (-1.6,-2.31);
\draw (-2.31,1.6)-- (-2.31,0);
\draw (-2.31,0)-- (-2.67,0);
\draw (-2.67,0)-- (-2.67,0.79);
\draw (-2.67,0.79)-- (-2.53,1.37);
\draw (-2.53,1.37)-- (-2.31,1.6);
\draw (-2.31,-1.6)-- (-2.31,0);
\draw (-2.31,0)-- (-2.67,0);
\draw (-2.67,0)-- (-2.67,-0.79);
\draw (-2.67,-0.79)-- (-2.53,-1.37);
\draw (-2.53,-1.37)-- (-2.31,-1.6);
\draw (2.31,-1.6)-- (2.31,0);
\draw (2.31,0)-- (2.67,0);
\draw (2.67,0)-- (2.67,-0.79);
\draw (2.67,-0.79)-- (2.53,-1.37);
\draw (2.53,-1.37)-- (2.31,-1.6);
\draw (2.31,1.6)-- (2.31,0);
\draw (2.31,0)-- (2.67,0);
\draw (2.67,0)-- (2.67,0.79);
\draw (2.67,0.79)-- (2.53,1.37);
\draw (2.53,1.37)-- (2.31,1.6);
\draw (-2.31,2.31)-- (-2.31,1.6);
\draw (-2.31,1.6)-- (-1.6,2.31);
\draw (-1.6,2.31)-- (-2.31,2.31);
\draw (-2.31,2.31)-- (-1.6,2.31);
\draw (-1.6,2.31)-- (-1.37,2.53);
\draw (-1.37,2.53)-- (-2.31,2.31);
\draw (-2.31,2.31)-- (-2.31,1.6);
\draw (-2.31,1.6)-- (-2.53,1.37);
\draw (-2.53,1.37)-- (-2.31,2.31);
\draw (2.31,2.31)-- (2.31,1.6);
\draw (2.31,1.6)-- (1.6,2.31);
\draw (1.6,2.31)-- (2.31,2.31);
\draw (2.31,2.31)-- (2.31,1.6);
\draw (2.31,1.6)-- (2.53,1.37);
\draw (2.53,1.37)-- (2.31,2.31);
\draw (2.31,2.31)-- (1.6,2.31);
\draw (1.6,2.31)-- (1.37,2.53);
\draw (1.37,2.53)-- (2.31,2.31);
\draw (2.31,-2.31)-- (2.31,-1.6);
\draw (2.31,-1.6)-- (1.6,-2.31);
\draw (1.6,-2.31)-- (2.31,-2.31);
\draw (2.31,-2.31)-- (1.6,-2.31);
\draw (1.6,-2.31)-- (1.37,-2.53);
\draw (1.37,-2.53)-- (2.31,-2.31);
\draw (2.31,-2.31)-- (2.31,-1.6);
\draw (2.31,-1.6)-- (2.53,-1.37);
\draw (2.53,-1.37)-- (2.31,-2.31);
\draw (-2.31,-2.31)-- (-2.31,-1.6);
\draw (-2.31,-1.6)-- (-1.6,-2.31);
\draw (-1.6,-2.31)-- (-2.31,-2.31);
\draw (-2.31,-2.31)-- (-2.31,-1.6);
\draw (-2.31,-1.6)-- (-2.53,-1.37);
\draw (-2.53,-1.37)-- (-2.31,-2.31);
\draw (-2.31,-2.31)-- (-1.6,-2.31);
\draw (-1.6,-2.31)-- (-1.37,-2.53);
\draw (-1.37,-2.53)-- (-2.31,-2.31);
\draw [color=ffqqqq] (-2.31,2.31)-- (-2.78,1.12);
\draw [color=ffqqqq] (-2.78,1.12)-- (-1.12,2.78);
\draw [color=ffqqqq] (-1.12,2.78)-- (-2.31,2.31);
\draw [color=ffqqqq] (-2.31,-2.31)-- (-1.12,-2.78);
\draw [color=ffqqqq] (-1.12,-2.78)-- (-2.78,-1.12);
\draw [color=ffqqqq] (-2.78,-1.12)-- (-2.31,-2.31);
\draw [color=ffqqqq] (2.31,-2.31)-- (1.12,-2.78);
\draw [color=ffqqqq] (1.12,-2.78)-- (2.78,-1.12);
\draw [color=ffqqqq] (2.78,-1.12)-- (2.31,-2.31);
\draw [color=ffqqqq] (2.31,2.31)-- (2.78,1.12);
\draw [color=ffqqqq] (2.78,1.12)-- (1.12,2.78);
\draw [color=ffqqqq] (1.12,2.78)-- (2.31,2.31);
\draw (-1.13,1.4) node[anchor=north west] {$P^{\dagger p}$};
\draw (-3.7,1.9) node[anchor=north west] {$P$};
\draw (-3.05,1.59)-- (-2.05,2.06);
\draw (-3.05,1.59)-- (-1.5,1.52);
\draw (-1.61,2.4)-- (-1.22,3.33);
\draw (-1.22,3.33)-- (-0.72,2.44);
\draw (-1.5,3.9) node[anchor=north west] {$\sigma_*(P)$};
\draw (-2.8,2.7) node[anchor=north west] {$p$};
\draw (2.84,0.09)-- (3.72,0.12);
\draw (3.92,0.37) node[anchor=north west] {$\Omega_{P^{\dagger p}}$};
\end{tikzpicture}
\caption{The starting of the tiling given here is obtained thanks to a square with three right angles and one loxodromic vertex: $p$. The convex set $\O_P$ is the union of the convex set $\O_{P^{\dagger p}}$ and the $\G_P$-orbits of the hatching zone. The limit set is the boundary of $\O_{P^{\dagger p}}$ minus the interior of the $\G_P$-orbits of the hatching zone.}
\label{Tiling_loxo}
\end{figure}

$\,$
\vspace*{-0.7cm}
\\
\par{
The following lemma gives the main properties of $P^{\dagger}$. To stay it, the notion of $(\G,\G')$-precisely invariant region is useful. If $\G$ acts on $\O$ and $\G'$ is a subgroup of $\G$ then a subset $A$ of $\O$ is \emph{$(\G,\G')$-precisely invariant} when for every $\g \in\G \smallsetminus \G'$, we have $\g(A) \cap A = \varnothing$ and for every $\g \in\G'$, we have $\g(A)=A$.
}

\begin{lemma}\label{Convex_hull}
Let $P$ be an irreducible 2-perfect loxodromic Coxeter polytope whose loxodromic vertices are simple. Consider the truncated Coxeter polytope $P^{\dagger}$ of $P$. For each loxodromic vertex $p$, we denote by $\C_p$ the cone of summit $p$ and basis the intersection of $\O_P$ with the hyperplane generated by the polars of the facets containing $p$. We have the following:
\begin{enumerate}
\item $P^{\dagger} \subset P$ and $\G_P \subset \G_{P^{\dagger}}$,
\item For every loxodromic vertex $p$ of $P$, the cone $\C_p$ is $(\G_P,\G_p)$-precisely invariant.
\item $P \cap \overline{C}(\LGP) = P^{\dagger}$,
\item $\O_{P^{\dagger}} \subset \O_P$,
\item The Coxeter polytope $P^{\dagger}$ is a quasi-perfect Coxeter polytope.
\end{enumerate}
\end{lemma}

\begin{proof}
\par{
The first statement is trivial. The second statement is a consequence of the fact that the action of $\G_P$ on the $d$-cell of the tiling of $\O_P$ is free. The main interest of the second statement is that we found a $(\G_P,\G_p)$-precisely-invariant region $C_p$ which is convex and such $\O_P \smallsetminus C_p$ is also convex. Hence the set $\O' = \bigcup_{\g\in \G_P} \g(P^{\dagger})$ is convex and $\G_P$-invariant. Hence $\LGP \subset \dO'$, and so $C(\LGP) \subset \O'$. Moreover, the limit set of $\G_p$ is included in the intersection $\pi_p$ of $\O_P$ with the hyperplane generated by the polars of the facets containing $p$. So $P \cap \overline{C}(\LGP)  =  P^{\dagger}$.
}
\par{
By definition, $\O_{P^{\dagger}}$ is the union of the orbits of $P^{\dagger}$ under $\G_{P^{\dagger}}$. Let $p$ be a loxodromic vertex and let $f$ be the new facet associated to the truncation of $p$. The set $\O_P \smallsetminus \pi_p$ has two connected components, the cone $\C_p$, and a convex set $\O^+$ which contains the interior of $P^{\dagger}$ and satisfies $\sigma_f(\O^+) \subset \C_p \subset \O_P$.  So we have $\O_{P^{\dagger}} \subset \O_P$. 
}
\par{
The last point is trivial, the truncation eliminates all the old loxodromic vertices. Moreover, since $P$ is 2-perfect the truncation process creates only elliptic vertices, so the Coxeter polytope $P^{\dagger}$ has only elliptic and parabolic vertices.
 }
\end{proof}

\subsection{Consequence of loxodrominess for non-simple vertices}

Before starting the proof, we make an important remark.

\begin{rem}[(Structure of the tiling)]\label{struct_tiling}
The tiling given by Coxeter group has a special feature, roughly speaking: "face extend to subspace". More precisely let $P$ be a Coxeter polytope. The union $\bigcup_{\g \in \G_P} \g(\partial P)$ is contained in a union of hyperplanes, in other words, every facet of the tiling extends to a hyperplane of the tiling. Even better, the $k$-skeleton of the tiling is a union of $k$-subspaces of $\overline{\O_P}$ (i.e. intersections of $k$-planes with $\overline{\O_P}$).
\end{rem}

\begin{propo}\label{propo_con2}
Let $P$ be an irreducible loxodromic Coxeter polytope of $\S^d$ and $p$ a vertex of $P$. If the vertex $p$ is perfect loxodromic then $p \not \in \LGP$.
\end{propo}

\begin{proof}
\par{
Suppose that $p \in \LGP$; then there exits a sequence of distinct elements $\g_n \in \G_P \smallsetminus \G_p$ such that $q_n:=\g_n(p) \to p$. We choose an affine chart $\A$ containing $\overline{\O_P}= \overline{\O}$. Let $K_p$ be the cone of summit $p$ generated by $P$ in $\A$ intersected with $\O$. Define $K_{q_n}:=\g_n( K_p)$ and $Q_n=\g_n(P)$. Since $\Sigma = \bigcup_{\g \in \G_P} \g(\partial P \cap \O)$ is contained $\partial K_p$, we must have $p \in K_{q_n}$ for $n$ big enough. We claim that $p \in \partial K_{q_n}$ for $n$ big enough. Indeed, if not, then, $\D_p(\O)= \underset{n}{\lim} \,\, \D_{q_n}(\O)$ is an affine chart contradicting the fact that $p$ is loxodromic. By symmetry, we get that $q_n \in \partial K_p$ for $n$ big enough. Hence, $q_n$ is on the hyperplane generated by a facet of $P$ for $n$ big enough, this is in contradiction with the fact that $q_n$ converges to $p$.
}
\end{proof}


\begin{rem}
Choi proves a similar statement for the action of a discrete group $\G$ on a properly convex open set $\O$ with the hypothesis that $\G_p$ is Gromov-hyperbolic and also a technical condition on the eigenvalue of $\G_p$ (See Theorem 6.4 of \cite{1304.1605v4}). Here we do not assume $\G_p$ Gromov-hyperbolic but we assume we are in a ''Coxeter situation''.
\end{rem}

\par{
The following definition is ad-hoc but it will be useful. A \emph{nicely embedded cone} $\C$ in a properly convex open set $\O$ is a properly convex open cone $\C$ such that $\C \subset \O$ and $\partial \C \cap \O$ is the relative interior $\mathcal{B}$ of the basis of $\C$. Hence, we have $\partial \C \smallsetminus \mathcal{B} \subset \dO$.
}

\begin{rem}
Let $P$ be a Coxeter polytope and $p$ a vertex of $P$. When $p$ is truncable there is a \emph{canonical} properly embedded cone which is $(\G_P,\G_p)$-precisely invariant : $\C_p = \Pi_p^- \cap \O_P$.
\end{rem}

\begin{cor}\label{exis_nice_emb}
Let $P$ be an irreducible loxodromic Coxeter polytope of $\S^d$ and $p$ be a vertex of $P$. If the vertex $p$ is perfect loxodromic then there exists a properly embedded cone which is $(\G_P,\G_p)$-precisely invariant.
\end{cor}

\begin{proof}
Since $p \not \in \LGP$, we get that $p \notin \overline{C}(\LGP)$, so one can choose an hyperplane $H$ such that $H \cap \O_P \neq \varnothing$ and one connected component $\S^d \smallsetminus H$ contains $\overline{C}(\LGP)$ while the other $H^-$ contains $p$. The cone $\C_p=H^- \cap\O_P$ does the job.
\end{proof}

\begin{propo}\label{propo_con3}
Let $P$ be an irreducible loxodromic Coxeter  polytope of $\S^d$ and $\mathcal{L}$ be set of perfect loxodromic vertices of $P$. For each $p \in \mathcal{L}$, we choose a nicely embedded cone $\C_p$ which is $(\G_P,\G_p)$-precisely invariant and let $\O'=\O_P \smallsetminus \bigcup_{p \in \mathcal{L}}\G_P(\overline{\C_p})$; then:
\begin{enumerate}
\item The open set $\O'$ is a $\G_P$-invariant properly convex.
\item $C(\LGP) \subset \O'$.
\item For every $p \in \mathcal{L}$, the point $p \in \dO_P$ is neither strictly convex nor with $\C^1$ boundary.
\item The point $p$ is an extremal point of $\dO_P$.
\item For every neighbourhood $\U$ of $p$ in $\S^d$ we have $\mu_{\O_P}(\U \cap P) = \infty$.
\end{enumerate}
\end{propo}

\begin{proof}
\par{
The existence of such a $\C_p$ is a consequence of Propositions \ref{propo_con} and \ref{propo_con2}. The first, third and fourth points are direct consequences of the $(\G_P,\G_p)$-precise invariance of the nicely embedded cone $\C_p$. The second point is a consequence of the fact that $\LGP$ is the smallest closed subset of $\PP^d$ that is $\G_P$-invariant. For the last point, since $\D_p(\O_P)$ is properly convex, we can find a cone $\omega_p$ of summit $p$ that contains $\O_P$ and such that $\dO_P \cap \partial \omega_p = \{ p \}$. Proposition \ref{compa} shows that $\mu_{\omega_p}(P) \leqslant \mu_{\O_P}(P)$ and Lemma \ref{pic_infini} below shows that $\mu_{\omega_p}(P)= \infty$.
}
\end{proof}

\begin{lemma}\label{pic_infini}
Let $\O$ be a properly convex open set. Suppose that $\O$ is a cone. Let $p$ be the summit of $\O$ and $P$ a convex subset of $\O$ such that $\overline{P} \cap \partial \O=\{ p\}$; then $\mu_{\O}(P)=\infty$.
\end{lemma}

\begin{proof}[Proof of Lemma \ref{pic_infini}]
Consider the affine chart $\A$ whose hyperplane at infinity is the hyperplane generated by the basis of $\O$. The automorphism group of $\O$ contains the homothety $h$ of the affine chart $\A$ of ratio $\frac{1}{2}$ fixing $p$ and $h(P) \subset P$. Of course, as $h$ is an automorphism of $\O$, we have $\mu_{\O}(h(P)) = \mu_{\O}(P)$, it follows that $\mu_{\O}(P)=\infty$.
\end{proof}

\section{Degenerate 2-perfect Coxeter polytopes}

The reader has probably noticed that the quadritomy of Theorem \ref{quadri1} is very useful. Hence, we believe that the study of the similar question for $2$-perfect Coxeter polytopes can be useful even if we will not use it.

\begin{propo}\label{classi_2perfect}
Let $P$ be a 2-perfect Coxeter polytope of $\S^d$. Then one of the following assertions is true:
\begin{enumerate}
\item $P$ is elliptic.
\item $P$ is parabolic.
\item $P$ is loxodromic and irreducible.
\item $P$ is decomposable; in fact $P=Q \otimes \cdot$ where $Q$ is parabolic or loxodromic perfect.
\end{enumerate}
\end{propo}

\begin{rem}
So, a loxodromic 2-perfect Coxeter polyhedron has to be irreducible.
\end{rem}

\begin{proof}
Consider the Cartan matrix $A_P$ of $P$; we will distinguish four cases:
\begin{enumerate}
\item $\textrm{Rank}(A_P)=d+1$ and $W_P$ is irreducible.
\item $\textrm{Rank}(A_P)=d+1$ and $W_P$ is not irreducible.
\item $P$ has a loxodromic vertex.
\item $\textrm{Rank}(A_P) < d+1$ and no loxodromic vertices.
\end{enumerate}
\par{
Suppose we are in the first case; then $A_P$ is either of positive type or of negative type, hence $P$ is either elliptic or irreducible loxodromic. Suppose we are in the second case; since $\textrm{Rank}(A_P)=d+1$, $P$ is decomposable by Theorem \ref{decomp_produ} and Lemma \ref{cas_produit} takes care of this case.
}
\par{
Suppose we are in the third case. Let $p$ be a loxodromic vertex of $P$. Consider the projective space $\pi_p$ spanned by the polar $[v_s]$ for $s \in S_p$. If $\Pi_p = \S^d$ or $p$ is not simple then we must have $\textrm{Rank}(A_P)=d+1$ 
and we are back to the previous case. If not, then $\Pi_p$ is a hyperplane and $p$ is simple. A) If $P$ is indecomposable then Proposition \ref{propo_con} shows that $P$ is truncable at $p$, hence there exists a polar of $P$ not in $\Pi_p$ and so $\textrm{Rank}(A_P)=d+1$, and we are back to the previous case again. B) If $P$ is decomposable then Lemma \ref{cas_produit} takes care of this case.
}
\par{
Suppose we are in the fourth case. Since $P$ is 2-perfect, it has only elliptic or parabolic vertices then Lemma \ref{demo26} of Vinberg concludes.
}
\end{proof}

This lemma is a direct adaptation of Vinberg's analogous lemma for the proof of Theorem \ref{quadri1}.

\begin{lemma}\label{cas_produit}
Let $P$ be a 2-perfect Coxeter polytope of $\S^d$. Suppose that $P$ is the product $P_1 \otimes P_2$ of two Coxeter  polytopes $P_1$ and $P_2$, then either:
\begin{enumerate}
\item Both are elliptic.
\item One is parabolic and the other one is the point Coxeter polytope.
\item One is loxodromic and the other one is the point Coxeter polytope.
\end{enumerate}
\end{lemma}

\begin{proof}
Suppose $P_1$ is not elliptic. A vertex $p$ of $P_2$ defines a vertex $\tilde{p}$ of $P_1\otimes P_2$ and the link $P_{\tilde{p}}$ of $P= P_1\otimes P_2$ at $\tilde{p}$ verifies $P_{\tilde{p}} = P_1 \otimes P_{2p}$. The Coxeter  polytope $P_{\tilde{p}} = P_1 \otimes P_{2p}$ is perfect hence elliptic, parabolic or loxodromic (Theorem \ref{quadri1}). The first case is impossible since $P_1$ is not elliptic.

So $P_{\tilde{p}}$ is perfect but not elliptic. Then by Theorem \ref{quadri1}, $P_{\tilde{p}}$ is indecomposable, hence $P_{2p}=\varnothing$, which means that $P_2$ is a point and $P=P_1 \otimes \cdot$.
\end{proof}

\begin{lemma}[(Vinberg, proof of Proposition 26)]\label{demo26}$\,$\\
Let $P$ be a Coxeter  polytope such that $\textrm{rank}(A_P)<d+1$.
\begin{enumerate}
\item If $P$ has an elliptic vertex then $P$ is either parabolic or decomposable. 
\item If $P$ has a parabolic vertex then $P$ is parabolic or $P=Q \otimes \cdot$ where $Q$ is a parabolic.
\end{enumerate}
\end{lemma}


\color{black}

\section{Geometry of the action}

In this part, we prove Theorem \ref{theo_principal}.\vspace{-1em}
\subsection{Cocompact action}

We rephrase Corollary \ref{cocompact} in our language to get used to it.

\begin{theorem}\label{Theo_comp}(Vinberg)
Let $P$ be a Coxeter  polytope. The action of $\G_P$ on $\O_P$ is cocompact if and only if all the vertices of $P$ are elliptic (i.e. $P$ is perfect).
\end{theorem}

\subsection{Geometrically finite action}

\begin{theorem}\label{Theo_geo_fini}
Let $P$ be a loxodromic 2-perfect Coxeter polytope. Then we always have: $$\mu_{\O_P}(C(\LGP) \cap P) < \infty.$$ In other word, the action of $\G_P$ on $\O_P$ is always geometrically finite.
\end{theorem}

\begin{proof}
Let $\mathcal{L}$ be the set of loxodromic vertices of $P$. Proposition \ref{propo_con3} shows that for each vertex $p \in \mathcal{L}$, one can find a $(\G_P,\G_p)$-precisely invariant nicely embedded cone $\C_p$. One can suppose these cones disjoint by taking them smaller. By removing the $\G_P$-orbits of all the $\C_p$, for $p \in \mathcal{L}$, one obtains a $\G_P$-invariant properly convex open set $\O'$ such that $P \cap \dO'$ is exactly the set of parabolic points of $P$. Now, Proposition \ref{propo_para} shows that there exists a neighbourhood $\U_p$ of $p$ in $\S^d$ such that $\mu_{\O_P}(P \cap \U_p) < \infty$. Since $P$ has only a finite number of vertices, we get that $\mu_{\O_P}(P\cap \O') < \infty$. Since $C(\LGP) \subset \O'$, we have $\mu_{\O_P}(C(\LGP) \cap P) < \infty$. Hence, the action of $\G_P$ on $\O_P$ is geometrically finite.
\end{proof}



\subsection{Finite volume case}

\begin{theorem}\label{Theo_vol_fini}
Let $P$ be a loxodromic 2-perfect Coxeter polytope. The action of $\G_P$ on $\O_P$ is of finite covolume if and only if all the vertices of $P$ are elliptic or parabolic (i.e. $P$ is quasi-perfect).
\end{theorem}

\begin{proof}
\par{
Suppose the action of $\G_P$ on $\O_P$ is of finite covolume. Assume one of the vertices $p$ of $P$ is loxodromic. Then the fifth point of Proposition \ref{propo_con3} shows that $\mu_{\O_P}(P)=\infty$. This is absurd, so every vertex of $P$ is either elliptic or parabolic.
}
\par{
Suppose all the vertices of $P$ are elliptic or parabolic. We know from Theorem \ref{Theo_geo_fini} that the action is geometrically finite, but since there is no loxodromic vertices, we have $\O'=\O_P$ in the proof of \ref{Theo_geo_fini} and we get that $\mu_{\O_P}(P) < \infty$.
}
\end{proof}

\subsection{Convex-cocompact action}

\begin{theorem}\label{Theo_conc_comp}
Let $P$ be a loxodromic 2-perfect Coxeter polytope. The action of $\G_P$ on $\O_P$ is convex-cocompact if and only if all the vertices of $P$ are elliptic or loxodromic.
\end{theorem}

The following corollary is immediate, thanks to Proposition \ref{propo_con}.

\begin{cor}
Let $P$ be a loxodromic 2-perfect Coxeter polytope whose loxodromic vertices are simple. Then, the action of $\G_P$ on $\O_P$ is convex-cocompact if and only if  the truncated Coxeter polytope $P^{\dagger}$ of $P$ is perfect.
\end{cor}

\begin{proof}[Proof of Theorem \ref{Theo_conc_comp}]
\par{
Suppose the action of $\G_P$ on $\O_P$ is convex-cocompact. Let $p$ be a vertex of $P$. We claim that $p \notin \LGP$. Indeed, first $p \in \dO_P$ if and only if $p$ is not elliptic; second if $p\in \dO_P$, consider the ray of $\O_P$ from any point $x_0\in P$ to $p$. The projection $r$ of this ray leaves every compact of $\Quotient{\O_P}{\G_P}$ since $P$ is a convex fundamental domain. In particular, the ray $r$ leaves the compact set $\Quotient{\overline{C}(\LGP)}{\G_P}$, thereby $p$ is not in $\LGP$. So, $p$ is not parabolic by Proposition \ref{propo_para} point 4).
}
\par{
Suppose all the vertices of $P$ are elliptic or loxodromic. We know from Theorem \ref{Theo_geo_fini} that the action is geometrically finite, but since there is no parabolic vertices, we get that $P \cap \O'$ is bounded in $(\O,d_{\O})$ in the proof of \ref{Theo_geo_fini} and so the action of $\G_P$ on $\O_P$ is convex-cocompact.
}
\end{proof}


\subsection{Geometrical definition of geometrical finiteness vs the topological one}\label{best_why}

\par{
In this paragraph, we motivated our definition of geometrical finiteness by comparing it to the definitions in pinched negative curvature and explaining why the definition we choose implies the other classical definitions.
}
\\
\par{
It is classical that if $X$ is a simply connected pinched negatively curved Riemannian manifold (i.e. a Hadamard manifold), then for every irreducible discrete group $\G$ of isometries of $X$, the thick part of the convex core is compact if and only if the volume of the convex core is finite and the group $\G$ is of finite type (thanks to \cite{MR1317633}).
}
\\
\par{
When $X$ is a properly convex open subset of $\PP^d$ which is strictly convex with $\C^1$-boundary then this equivalence remains true (\cite{Crampon:2012fk}). We stress that Margulis's lemma is valide is any Hilbert Geometry (\cite{Cooper:2011fk} or \cite{Crampon:2011fk}).
}
\\
\par{
But there is also a topological version of geometrical finiteness. The action of $\G$ on $X$ is \emph{geometrically finite} if all the points of the limit set of $\G$ are \emph{conical limit points} or \emph{bounded parabolic fixed points}. See \cite{MR1317633} for the definition. We only stress that these definitions are purely topological.
}
\\
\par{
When $X$ is a Hadamard manifold and $\G$ an irreducible group of isometry of $X$ then the topological definition of geometrically finite action is equivalent to the geometrical definition by \cite{MR1317633}. But, when $X$ is a properly convex open subset of $\PP^d$ which is strictly convex with $\C^1$-boundary, this is no longer true. We only have that the geometrical definition implies the topological one, see \cite{Crampon:2012fk} for the implication and a counterexample in dimension $4$.
}
\\
\par{
Maybe even worst, if $X$ is a properly convex open subset of $\PP^d$, which is not supposed strictly convex nor with $\C^1$-boundary then: \emph{if the volume of the convex core is finite and the group $\G$ is of finite type then the thick part of the convex core is compact}. But, I don't know if the converse is true. This implication is just a consequence of the fact that the volume of balls of radius $r>0$ in Hilbert geometry are bounded from below by a universal constant depending only on the dimension $d$ (thanks to Benzécri's theorem, see \cite{Crampon:2012fk} for the details).
}

\section{Zariski closure of $\G_P$}

\subsection{Notations}

Let us introduce some notation for the sake of clarity.
We will denote by $\textrm{Trans}_d$ the subgroup of $\ss$ which is the group of translations in the standard affine chart. In term of matrices, it is the group:

$$
\textrm{Trans}_{d}=
\left\{
\left.
\begin{pmatrix}
1 &      & 0 & u_1 \\
  & \ddots & & \vdots \\
0 & &  1 & u_d \\
0 & \cdots & 0 & 1
\end{pmatrix}
\right|
(u_1,...u_d) \in \R^d
\right\}
$$

and by $\textrm{Diag}_d$ the subgroup of $\ss$ of diagonal matrices with positive entries:

$$
\textrm{Diag}_d=
\left\{
\left.
\begin{pmatrix}
\lambda_1 &      & 0 \\
  & \ddots & \\
0 & &  \lambda_{d+1}
\end{pmatrix}
\right|
\lambda_1,...,\lambda_{d+1} \in \R_+^* \textrm{ such that } \lambda_1 \cdots \lambda_{d+1} =1
\right\}.
$$

These two groups are isomorphic as abstract groups.

\begin{nota}
If $P$ is a Coxeter  polytope, we will denote by $G_P$ the connected component of the Zariski closure of the discrete subgroup $\G_P$ of $\ss$.
\end{nota}

\subsection{The perfect case}

\par{
In the perfect case, all the arguments are in the literature, we just put them together.
}

\subsubsection{The easy case}

\begin{theorem}\label{zar_dens_perf_easy}
Let $P$ be a perfect Coxeter polytope. Let $G_P$ be the connected component of the Zariski closure of $\G_P$ in $\ss$.
\begin{enumerate}
\item If $P$ is elliptic, then $G_P=\{ 1 \}$.
\item If $P$ is parabolic, then $G_P$ is conjugate to the group $\textrm{Trans}_{d}$.
\item If $P$ is loxodromic and $W_P$ is affine, then $\O_P$ is a simplex, the Coxeter group $W_P$ is affine irreducible of type $\tilde{A}_n$ and the group $G_P$ is conjugate to $\textrm{Diag}_d$.
\end{enumerate}
\end{theorem}

\begin{proof}
In the first case, the group $\G_P$ is finite so $G_P=\{ 1 \}$. In the second case, $\G_P$ is a lattice of a conjugate of $\textrm{Trans}_d \rtimes \textrm{SO}_d$ and the image of $\G_P$ in $\textrm{SO}_d$ is finite, therefore $G_P$ is conjugate to  $\textrm{Trans}_d$. In the third case, by point $4)$ of Proposition \ref{5thomy}, $\G_P$ is a lattice of $\Aut(\O_P)= \textrm{Diag}_d \rtimes \mathfrak{S}_{d+1}$, the conclusion follows, where $\mathfrak{S}_{d+1}$ is the symmetric group on $\{ 1, ... , d+1 \}$ acting canonically on $\R^{d+1}$.
\end{proof}

\subsubsection{The interesting case}

\begin{rem}
From Theorem \ref{quadri1}, we learn that if $P$ is a perfect polytope then $P$ is either elliptic, parabolic, loxodromic with $W_P$ affine irreducible or loxodromic with $W_P$ large irreducible.
\end{rem}

\begin{theorem}[(Benoist  + Folklore)]\label{zar_dens_perf_hard}
Let $P$ be a perfect loxodromic Coxeter polytope with $W_P$ large irreducible. Then we have the following alternative:
\begin{itemize}
\item $\O_P$ is an ellipsoid and $G_P$ is conjugate to $\so{d}$ or
\item $\O_P$ is not an ellipsoid and $G_P=\ss$.
\end{itemize}
\end{theorem}

\begin{proof}
First, from Theorem \ref{tri}, we know that $\G_P$ is strongly irreducible and so $\O_P$ is indecomposable. We need to distinguish three cases: $\O_P$ is an ellipsoid, $\O_P$ is symmetric\footnote{
A properly convex open set is \emph{symmetric} if for every point $x \in \O$ there exists an isometry $\g$ of $(\O,d_{\O})$ which fixes $x$ and whose differential at $x$ is $-Id$.
} but not an ellipsoid and $\O_P$ is not symmetric.

If $\O_P$ is an ellipsoid, then $\Aut(\O_P)$ is conjugate to $\so{d}$ and $\G_P$ is a cocompact lattice of $\Aut(\O_P)$. The conclusion follows from Borel's density Theorem \ref{borel}. 

Assume $\O_P$ is symmetric but not an ellipsoid. Then $\Aut(\O_P)$ has property $(T)$, from Theorem \ref{koecher} and Theorem \ref{rank2} below. So, $\G_P$ has property (T)\footnote{We don't give the definition of property (T), since we don't need the definition for our purpose. The reader is referred to the book \cite{MR2415834} for a definition and all the proof of the theorem we will use in the sequel.} too since it is a lattice of $\Aut(\O_P)$ by Theorem \ref{heriT}. But an infinite Coxeter group does not have property (T) by Theorem \ref{propT}. So this case is absurd.

If $\O_P$ is not symmetric, then $G_P= \ss$ by Theorem \ref{zariski_dense}.
\end{proof}

\begin{theorem}[(Borel's density theorem, \cite{MR0123639})]\label{borel}
A lattice of a semi-simple lie group without compact factor is Zariski-dense.
\end{theorem}

\begin{theorem}[(Koecher, Vinberg,  \cite{MR0158414}, \cite{MR1446489} or \cite{MR1718170})]\label{koecher}
Let $\O$ be an indecomposable symmetric properly convex open subset of $\PP^d$. Then $\O$ is the symmetric space associated to $\so{d}$ or $\textrm{SL}_m(\mathbb{K})$ where $\mathbb{K}=\R,\, \mathbb{C}, \, \mathbb{H}$ and $m \geqslant 3$ or to the exceptional group $E_{6(-26)}$. In particular, the automorphism group of an indecomposable symmetric properly convex open set which is not an ellipsoid is a quasi-simple\footnote{A Lie group is \emph{quasi-simple} when its Lie algebra is simple or equivalently when all its normal subgroups are discrete.} Lie group of real rank\footnote{The real rank of a semi-simple Lie group is the common dimension of all the maximal splitted connected abelian subgroups, e.g maximal splitted tori.} at least two.
\end{theorem}

\begin{theorem}[(Kazhdan, Delaroche and Kirillov, Vasertein, Wang, Theorem 1.6.1 \cite{MR2415834})]\label{rank2} A quasi-simple Lie group of real rank at least two has property (T).
\end{theorem}

\begin{theorem}[(Kazhdan, Theorem 1.7.1 of \cite{MR2415834})]\label{heriT}$\,$\\
A lattice $\G$ of a locally compact group $G$ has property (T) if and only if $G$ has property (T).
\end{theorem}

\begin{theorem}[(Bozejko, Januszkiewicz and Spatzier in \cite{MR950825})]\label{propT}
Let $W$ be a Coxeter group. If $W$ has property (T) then, $W$ is finite.
\end{theorem}

\begin{theorem}[(Benoist \cite{MR2010735})]\label{zariski_dense}
Let $\G$ be a discrete group of $\ss$ acting cocompactly on a properly convex open set $\O$. If the group $\G$ is strongly irreducible and $\O$ is not symmetric, then $\G$ is Zariski-dense in $\ss$.
\end{theorem}

\subsection{Non-degenerate 2-perfect case}

\begin{theorem}\label{adh_zar_2-perf}
Let $P$ be a loxodromic 2-perfect Coxeter polytope of $\S^d$ which is not perfect. Then either $G_P$ is conjugate to $\so{d}$, or $G_P = \ss$.
\end{theorem}

A nice corollary is the following:

\begin{cor}\label{coro_adh_zar}
Let $P$ be a loxodromic  2-perfect Coxeter polytope of $\S^d$ which is not perfect and whose loxodromic vertices are truncable. Let $P^{\dagger}$ be the truncated Coxeter polytope associated to $P$ then either $\O_{P^{\dagger}}$ is an ellipsoid, or $G_P = \ss$. In particular, in both cases, $G_P=G_{P^{\dagger}}$.
\end{cor}

We will show the Theorem \ref{adh_zar_2-perf} and the Corollary \ref{coro_adh_zar} at paragraph  \ref{proof_zar_clo}.

\subsection{Proximality, limit sets and Zariski closure}\label{detail_prox}

\par{
The following paragraph presents basic facts about proximal action on the projective space. We have included the facts we will need and some arguments to make the reading easier to the reader not familiar to the theory. We have tried to give references when we though an argument would divert the reader's attention or be too long. This paragraph has nothing original, we borrow a lot from \cite{MR1348303,MR1389734,MR1767272}.
}

\subsubsection{Proximal elements and proximal subgroups}

\par{
An element $\g$ of $\ss$ is \emph{proximal} when the eigenvalue of maximal modulus is a simple eigenvalue. In that case, the eigenvalue of maximal modulus has to be real, it has to be the spectral radius $\rho$ or its opposite $-\rho$. The corresponding eigenspace is a line, so a point $x^+_{\g}$ of $\PP^d$. This point is called the \emph{attractive fixed point} of $\g$. Indeed, it is easy to see that outside a projective hyperplane $H$, for every point $x\in \PP^d \smallsetminus H$, we have $\g^n(x) \to x^+_{\g}$ when $n \to +\infty$.
}
\par{
A subgroup $G$ of $\PP^d$ is \emph{proximal} when it contains a proximal element.
}

\subsubsection{Proximal action}

The action of a group $G$ on $\PP^d$ is \emph{proximal} when for every two points $x,y \in \PP^d$ there exists a sequence $(g_k)_{k \in \N}$ in $G$ such that the sequences $(g_k(x))_{k \in \N}$ and $(g_k(y))_{k \in \N}$ converge to the same point.

The link between the notion of proximal group and proximal action is given by the following equivalence. If $G$ is a subgroup of $\ss$ then ``$G$ is irreducible and the action of $G$ on $\PP^d$ is proximal'' if and only if ``the group $G$ is strongly irreducible and proximal'' (Theorem 2.9 of \cite{MR1389734}).

\subsubsection{Limit set}

\par{
Suppose $G$ is strongly irreducible and proximal. Then one can show that the closure $\Lambda_G$ of all the attractive fixed points of all the proximal elements of $G$ is the smallest\footnote{Every closed $G$-invariant subset contains $\Lambda_G$.} closed $G$-invariant subset of $\PP^d$ (see Theorem 2.3 of \cite{MR1389734}). So in particular, the action of $G$ on $\Lambda_G$ is minimal\footnote{Every orbit is dense.}. This closed subset $\Lambda_G$ is called the \emph{limit set} of $G$.
}

\subsubsection{The case of an algebraic group}

\par{
If we assume moreover that $G$ is a Zariski closed subgroup of $\ss$. Then $\Lambda_G$ is the unique closed orbit of the action of $G$ on $\PP^d$. In fact, $\Lambda_G$ is even Zariski closed. In particular, $\Lambda_G$ is a smooth algebraic sub-manifold of $\PP^d$. This is due to the following fact:
}
\par{
The action of a Zariski closed subgroup $G$ of $\ss$ on $\PP^d$ is algebraic, so in particular every orbit is locally closed for the Zariski topology, i.e. every orbit is Zariski-open in its Zariski-closure. First, the limit set is closed for the Zariski topology. Indeed, take a point $x \in \Lambda_G$, the orbit $G\cdot x$ is open in its Zariski closure $\overline{G\cdot x ^{Zar}}$, hence $\overline{G\cdot x ^{Zar}} \smallsetminus G\cdot x$ is Zariski closed, hence closed also in the usual sense. But $\Lambda_G$ is the smallest closed invariant set, hence $G\cdot x = \Lambda_G$ and is Zariski closed. The fact that the orbit of a point outside $\Lambda_G$ is not closed is a consequence of the definition of $\Lambda_G$. Finally, the limit set $\Lambda_G$ is a smooth algebraic manifold since there exists a transitive action on it.
}
\subsubsection{The point of view of semi-simple group's representation theory}

\par{
Even better, since $G$ is a Zariski closed subgroup of $\ss$, it is a Lie group. Let $G_0$ be the connected component of $G$. Since the action of $G$ on $\R^{d+1}$ is strongly irreducible, the action of $G_0$ on $\R^{d+1}$ is also strongly irreducible (an algebraic variety can only have a finite number of connected components, so the index of $G_0$ in $G$ is finite). Much better, the Lie group $G_0$ is semi-simple since the group $G$ is proximal\footnote{The group $G_0$ is a reductive group (i.e. its unipotent radical is trivial) because $G_0$ is irreducible. So, to show that $G_0$ is semi-simple, one just has to show that the center of $G_0$ is discrete. Now, any element of the center has to preserve the eigenspaces of all elements of $G_0$, in  particular the proximal one, hence the center is composed only of a homothety of determinant one. qed.}.
}
\\
\par{
Hence, the representation $\rho_0:G_0 \to \ss$ is an irreducible representation of the semi-simple group $G_0$. Let $KAN=G_0$ be an Iwasawa decomposition of $G_0$ where $K$ is a maximal compact subgroup of $G$, $A$ a maximal abelian connected and diagonalizable over $\R$ subgroup and $N$ a maximal unipotent subgroup.
}
\\
\par{
A representation $\rho$ of a connected semi-simple group with finite center $G_0$ is \emph{proximal} when the subspace $\textrm{Fix}(N) = \{ x \in \R^{d+1} \, |\, \forall n \in N,\, n(x)=x \}$ is a line. In \cite{MR1348303} Abels, Margulis and Soifer show that: \textit{an irreducible representation $\rho:G_0 \to \ss$ is proximal if and only if the group $\rho(G_0)$ is proximal} (Theorem 6.3). In particular, the representation $\rho_0$ is proximal.
}
\\
\par{
Since the subspace $\textrm{Fix}(N)$ is a line, it is a point $x_{N}$ of $\PP^d$. The orbit of $x_{N}$ under the group $G_0$ is equal to the orbit of $x_{N}$ under the compact group $K$ (since $N$ is normal in $AN$), hence it is closed, thereby it is the unique closed orbit of $G$ on $\PP^d$, i.e. the limit set $\Lambda_G$.
}

\subsubsection{Zariski closure}

\par{
This procedure is particularly interesting when one starts with a discrete subgroup $\G$ of $\ss$. Then one can consider $G_0$ the connected component of the Zariski closure of $\G$. The action of $\G$ on $\R^{d+1}$ is strongly irreducible if and only if the action of $G_0$ on $\R^{d+1}$ is irreducible. Moreover, in that case, the action of $\G$ on $\PP^d$ is proximal if and only if the action of $G_0$ on $\PP^d$ is proximal (Theorem 6.3 of \cite{MR1040268}).
}
\\
\par{
Hence, if one starts with a strongly irreducible and proximal subgroup $\G$ of $\ss$, this procedure gives a connected semi-simple group with finite center $G_0$, an irreducible representation $\rho:G_0 \to \ss$ and two closed subsets of $\PP^d$: $\LG \subset \Lambda_{G_0}$.
}

\subsection{Positive proximality and Zariski closure}

\par{
In this article, we are interested in discrete subgroups of $\ss$ which preserve a properly convex open subset of $\PP^d$. In this context, the notion of positive proximality is interesting.
}
\subsubsection{Positively proximal element and positively proximal group}

\par{
A proximal element $\g$ of $\ss$ is \emph{positively proximal} when its spectral radius $\rho$ is an eigenvalue. A proximal subgroup $G$ of $\PP^d$ is \emph{positively proximal} when all its proximal elements are positively proximal.
}
\\
\par{
A theorem of Benoist makes a bridge between being positive proximal and preserving a properly convex open set. Suppose $\G$ is strongly irreducible. \textit{Then the group $\G$ preserves a properly convex open set if and only if the group $\G$ is positively proximal} (Proposition 1.1 of \cite{MR1767272}). In particular, the group $\G$ is proximal, and the construction explained in the previous paragraph applied.
}

\subsubsection{A key lemma of Benoist}

\begin{lemma}[(Benoist \cite{MR1767272})]\label{adh_zar_benoist}
Let $\G$ be a strongly irreducible subgroup of $\ss$ preserving a properly convex open subset. The connected component $G$ of the Zariski closure of $\G$ is a semi-simple Lie group and the action of $G$ on $\PP^d$ is proximal. Moreover, we can be more precise in two extremal cases:
\begin{enumerate}
\item if the limit set $\Lambda_G$ of $G$ is the boundary of a properly convex open subset of $\PP^d$, then $\Lambda_G$ is an ellisphere and $G$ is conjugate to $\so{d}$.
\item if $\Lambda_G=\PP^d$ then $G=\ss$.
\end{enumerate}
\end{lemma}

The following lemma is an easy consequence of Lemma \ref{adh_zar_benoist}. We state it to clarify our strategy to find the Zariski closure of $\G_P$.

\begin{lemma}\label{detail_adh_zar}
Let $\G$ be a strongly irreducible subgroup of $\ss$ preserving a properly convex open set. Let $G$ be the connected component of the Zariski closure of $\G$. Suppose there exists a point $x \in \Lambda_G$ and a Zariski closed subgroup $H$ of $G$ such that the orbit $H \cdot x$ is a sub-manifold of $\PP^d$ of dimension $d-1$. Then $G$ is conjugate to $\so{d}$ or $G=\ss$.
\end{lemma}

\subsubsection{A useful remark}

The following remark gives a description of the maximal properly convex open set preserved by a strongly irreducible positively proximal discrete subgroup $\G$ of $\ss$.

An element $\g \in \ss$ is bi-proximal if $\g$ and $\g^{-1}$ are proximal. We introduce the following notation. If $\g$ is a bi-proximal element of $\ss$, then $\g^+$ is the eigenvalue corresponding to the spectral radius, $H_{\g}$ is the projective subspace spanned by all the eigenvectors except the one corresponding to the smallest (in modulus) eigenvalue and $H_{\g}^+$ is the affine chart $\PP^d \smallsetminus H_{\g}$. Hence, $\g^+$ is the unique attractive fixed point of $\g \curvearrowright \PP^d$ and $H_{\g}$ is the unique attractive fixed point of $\g \curvearrowright \PP^{d*}$, where $\PP^{d*}$ is the dual of $\PP^d$.

\begin{rem}\label{rem_small_and_big}
The smallest properly convex open set $\O_{min}$ preserved by $\G$ is the convex hull of the limit set $\LG$. The largest $\O_{max}$ is the dual of the convex hull $\O_{min,*}$ of the limit set $\Lambda_{\Gamma,*}$ of the dual action of $\G$ on $\PP^{d*}$.

Let $\textrm{AFP}(\G)$ (resp. $\textrm{AFP}(\G^*)$) be the set of attractive fixed points of proximal elements of $\G$ (resp. $\G^*$). We know that $\textrm{AFP}(\G)$ is dense in $\LG$, so we get $\O_{min}$ is the convex hull of $\textrm{AFP}(\G)$. Now, since $\O_{max}$ is the dual of $\O_{min,*}=C(\textrm{AFP}(\G^*))$, we get $\O_{max} = \bigcap_{\g \in \G^{prox}} H_{\g}^+$, where $\G^{prox}$ is the set of bi-proximal elements of $\G$.
\end{rem}

\subsection{The proof of Theorem \ref{adh_zar_2-perf}}\label{proof_zar_clo}

\subsubsection{The action of $\U_{d-1}$ on $\PP^d$}\label{action_para}

We define a subgroup of $\ss$:
$$
\U_{d-1}=
\left\{
\left.
\begin{pmatrix}
1 & u_1 & \cdots     & u_{d-1} & \frac{1}{2}(u_1^2+\cdots + u_{d-1}^2) \\
& 1 &     &  0 & u_1 \\
  & \ddots & & & \vdots \\
0 & &  & 1 & u_{d-1} \\
0 & \cdots & & 0 & 1
\end{pmatrix}
\right|
(u_1,...u_{d-1}) \in \R^{d-1}
\right\}.
$$

The group $\U_{d-1}$ preserves an ellipsoid $\E$, fixes a point $p \in \partial \E$ and fixes every horosphere of $\E$ centered at $p$. In other words, $\U_{d-1}$ is included in the stabilizer of a horosphere in the hyperbolic space $(\E,d_{\E})$. More precisely, $\U_{d-1}$ is the subgroup composed of the non-screw parabolic elements fixing $p$ of the hyperbolic space $(\E,d_{\E})$. In particular, $\U_{d-1}$ is isomorphic to $\R^{d-1}$. Moreover, if $x$ is not in the tangent space to $\partial \E$ at $p$ then the space $\U_{d-1} \cdot x \cup \{ p \}$ is an ellisphere.

\begin{figure}[h!]
\centering
\begin{tikzpicture}[line cap=round,line join=round,>=triangle 45,x=0.4cm,y=0.4cm]
\clip(-7.2,-6.72) rectangle (4.18,9.12);
\draw [domain=-4.64:4.18] plot(\x,{(-0-0*\x)/3});
\draw [line width=2.4pt] (0,2) circle (0.8cm);
\draw [rotate around={90:(0,1.74)}] (0,1.74) ellipse (0.7cm and 0.69cm);
\draw [rotate around={90:(0,3.2)}] (0,3.2) ellipse (1.28cm and 1.19cm);
\draw [rotate around={90:(0,4.5)}] (0,4.5) ellipse (1.8cm and 1.5cm);
\draw [rotate around={90:(0,1.5)}] (0,1.5) ellipse (0.6cm and 0.57cm);
\draw (-0.08,-0.02) node[anchor=north west] {$p$};
\draw [rotate around={90:(0,-1.72)}] (0,-1.72) ellipse (0.69cm and 0.68cm);
\draw [rotate around={90:(0,-3)}] (0,-3) ellipse (1.2cm and 1.13cm);
\draw (-6,9.68) node[anchor=north west] {$\mathcal{U}_{d-1}\cdot x$};
\draw (1.3,4.6) node[anchor=north west] {$\mathcal{E}$};
\begin{scriptsize}
\fill [color=black] (0,0) circle (1.5pt);
\end{scriptsize}
\end{tikzpicture}
\caption{The orbits of the action of $\U_{d-1}$ on $\PP^d$}
\end{figure}
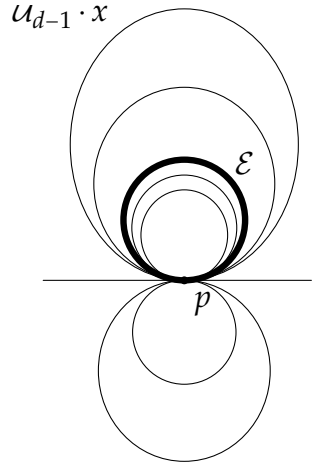

The following lemma is then a direct corollary of Proposition \ref{propo_para}:

\begin{lemma}\label{zar_para}
Let $P$ be an irreducible loxodromic Coxeter polytope of $\S^d$ and let $p$ be a parabolic vertex of $P$. Then the connected component $G_p$ of the Zariski closure of $\G_p$ is conjugate to $\U_{d-1}$.
\end{lemma}

\subsubsection{The action of $\so{d-1}$ on $\PP^d$}\label{action_lox_large}

\par{
The action of $\so{d-1}$ on $\PP^d$ has 7 types of orbits. To see this, one should think of $\PP^d$ as the projective space $\PP(\R^d \oplus \R)$ i.e. the projective completion of $\R^d$. The action $\so{d-1}$ on $\PP^d$ preserves the hyperplane at infinity $H_{\infty}$ of the affine chart $\R^d$ of $\PP^d = \PP(\R^d \oplus \R)$. The action of $\so{d-1}$ on $H_{\infty}=\PP^{d-1}$ has three orbits, the limit set for the proximal action on $H_{\infty}$ which is an ellisphere of dimension $d-2$ and the two connected components of $H_{\infty} \smallsetminus \E$, one of them being a ball.
}
\\
\par{
For the action on the affine chart $\PP^d \smallsetminus H_{\infty}$, the origin is fixed, there is a cone $\C_{light}$ which gives two orbits. Finally, the orbit of an element inside the cone is one sheet of a hyperboloid with two sheets and the orbit of an element outside the cone is a hyperboloid with one sheet. The  space $\PP^d \smallsetminus (H_{\infty} \cup \C_{light})$ has three connected components, two of them are balls, these are the inside of the cone, the remaining one is the outside.
}

\subsubsection{Action of $\textrm{Diag}_{d-1}$ on $\PP^d$}\label{action_lox_aff}

We define the following group:
$$
\textrm{Diag}_{d-1}=
\left\{
\left.
\begin{pmatrix}
\lambda_1 &      && 0 \\
  & \ddots & \\
 & &  \lambda_d & \\
0 & & &  1
\end{pmatrix}
\right|
\lambda_1,...,\lambda_{d} \in \R_+^* \textrm{ such that } \lambda_1 \cdots \lambda_{d} =1
\right\}.
$$
\par{
The action of $\textrm{Diag}_{d-1}$ on $\PP^d$ has exactly $d+1$ fixed points which are in generic position. This action preserves $d+1$ hyperplanes $(H_i)_{i=1,...,d+1}$ each of them generated by $d$ fixed points. The orbit of any point $x \in \PP^d$ which is not in one of the $(H_i)_{i=1,...,d+1}$ is a convex hypersurface of $\PP^d$ i.e. an open subset of the boundary of a properly convex open  subset.
}

\subsubsection{Conclusion}

\begin{lemma}\label{lemm_vertex}
Let $P$ be an irreducible loxodromic Coxeter  polytope of $\S^d$. If $P$ has a perfect non-elliptic vertex then $G_P$ is conjugate to $\so{d}$ or equal to $\ss$.
\end{lemma}

\begin{proof}
\par{
To simplify notation, we denote $G_P$ by $G$. Since $P$ is irreducible and loxodromic, we know from Theorem \ref{adh_zar_benoist} and Proposition \ref{5thomy} that $G$ is a semi-simple proximal Lie subgroup of $\ss$. Hence the limit set $\Lambda_G$ of $G$ is the unique closed orbit of the action of $G$ on $\PP^d$ and a  smooth Zariski closed sub-manifold of $\PP^d$.
}
\par{
If $P$ admits a parabolic vertex $p$, then the Zariski closure of $\G_p$ is conjugate to $\U_{d-1}$ (Lemma \ref{zar_para}). Apart from the points on a unique hyperplane $H_p$ containing $p$, for every $x \notin H_p$, the space $\U_{d-1}\cdot x \cup \{ p \}$ is an ellisphere $\E_x$. Since $G$ is irreducible, we can find a point $x\in \Lambda_G$ but not in $H_p$. Thus, the limit set $\Lambda_G$ must contain an ellisphere, and so $\Lambda_G$ is an ellisphere or the all $\PP^d$. Lemma \ref{detail_adh_zar} concludes. 
}
\par{
If $P$ admits a loxodromic vertex $p$ such that $W_p$ is not affine, then the connected component of the Zariski closure $G_p$ of $\G_p$ is conjugate to $\s{d}$ or $\so{d-1}$, thanks to Theorem \ref{zar_dens_perf_hard}. If $P$ admits a loxodromic vertex $p$ such that $W_p$ is affine, then the connected component of the Zariski closure $G_p$ of $\G_p$  is conjugate to $\textrm{Diag}_{d-1}$ thanks to Theorem \ref{zar_dens_perf_easy}. We again apply the idea of Lemma \ref{detail_adh_zar}. In all these three cases, since the action of $G$ is irreducible, we can find a point $x\in \Lambda_G$ such that the orbit of $x$ under $G_p$ is of dimension $d-1$. Hence, Lemma \ref{detail_adh_zar} concludes.
}
\end{proof}


\begin{proof}[Proof of Theorem \ref{adh_zar_2-perf}]
We assume $P$ is not perfect, so $P$ admits a perfect non-elliptic vertex and Lemma \ref{lemm_vertex} concludes.
\end{proof}

\begin{proof}[Proof of Corollary \ref{coro_adh_zar}]
Thanks to Theorems \ref{zar_dens_perf_hard} and \ref{adh_zar_2-perf} which can be applied to $P$ or $P^\dagger$ and the fact that $\G_P \subset \G_{P^\dagger}$, we just have to prove that if $G_P=\so{d}$ then $\G_{P^\dagger} \subset \so{d}$. In that case, $\G_P$ preserves a unique ellipsoid $\E$, any loxodromic vertex $p$ is outside $\overline{\E}$. Let $\Pi_p$ be the hyperplane spanned by the polar $[v_s]$ for $s$ facets of $P$ containing $p$. The hyperplane $\Pi_p$ is the hyperplane $p^\bot$ for the quadratic form defined by $\E$. Hence, the group $\G_{P^\dagger} \subset \so{d}$.
\end{proof}

\subsection{Degenerate 2-perfect case}

We just give the statement for the degenerate 2-perfect case without proof since the proof are similar and easier. The subgroups $\textrm{Trans}_{d-1}$, $\so{d-1}$ and $\s{d}$ of $\s{d}$ can be embedded in $\ss$ in the upper-left corner. We make the abuse of notation of identifying these subgroups of $\s{d}$ with their images in $\ss$.

\begin{propo}
Let $P$ be a 2-perfect Coxeter polytope of $\S^d$ which is not perfect. If $P$ is decomposable then 
$G_P$ is conjugate to $\textrm{Trans}_{d-1}$, $\so{d-1}$ or $\s{d}$.
\end{propo}

\section{About the convex set}

In this section, we prove Theorems \ref{maxi_intro}, \ref{mini_intro}, \ref{strict_intro} and \ref{exis_strict_intro}.

\subsection{The convex set $\O_P$ is the largest convex open subset of $\PP^d$ preserved by $\G_P$}

We start with Theorem \ref{maxi_intro}.

\begin{theorem}\label{Theo_maxi}
Let $P$ be a loxodromic 2-perfect Coxeter polytope. Then $\O_P$ is the largest convex open subset preserved by $\G_P$.
\end{theorem}

\begin{proof}
Remark \ref{rem_small_and_big} shows that $\O_{max} = \bigcap_{\g \in \G^{prox}} H_{\g}^+$. Proposition \ref{propo_con3} shows that $\O_P \smallsetminus \O_{min}$ modulo $\G$ is a finite union of sets each containing a $(\G_P,\G_p)$-precisely invariant nicely embedded cone $\C_p$, for $p$ running over the set of loxodromic vertices of $P$ and $\D_p(\C_p) = \O_p$.

The closure of the set  $F_p^{prox}$ of attractive bi-proximal fixed points of $\G_p$ is the limit set $\Lambda_p$ (by Benoist \cite{MR1767272}), and Vey shows that since the action of $\G_p$ on $\O_p$ is cocompact, we also have $C(F_p^{prox}) = \O_p$ (\cite{MR0283720}). We stress that $p \in  H_{\g}$ for every $\g \in \G_p^{prox}$.

Hence, the convex set $\O_{max} = \bigcap_{\g \in \G^{prox}} H_{\g}^+$ contains in its boundary any loxodromic vertex $p$ of $P$, and we have $\D_p(\O_{max}) = \O_p$. Let $\A$ be an affine chart containing $\overline{\O_{max}}$. The convex set $\O'_p = \bigcap_{\g \in \G_p^{prox}} H_{\g}^+ \cap \A$ is a cone of summit $p$ such that $\D_p(\O'_p) = \O_p$, so $\O_{max} \subset \O_P$.
\end{proof}

\subsection{When is $\O_P$ the smallest convex open subset of $\PP^d$ preserved by $\G_P$ ?}

We are ready to prove Theorem \ref{mini_intro}. 

\begin{theorem}\label{Theo_mini}
Let $P$ be a loxodromic 2-perfect Coxeter polytope. The convex set $\O_P$ is the smallest convex open subset of $\PP^d$ preserved by $\G_P$ if and only if the action of $\G_P$ on $\O_P$ is of finite covolume. In that case, the convex set $\O_P$ is the unique properly convex open set preserved by $\G_P$.
\end{theorem}

\begin{proof}
\par{
Thanks to Theorem \ref{Theo_vol_fini}, we only have to show that the convex set $\O_P$ is the smallest convex open subset of $\PP^d$ preserved by $\G_P$ if and only if every vertex of $P$ is elliptic or parabolic. Suppose one vertex $p$ of $P$ is loxodromic, Proposition \ref{propo_con3} builds a convex set $\O'$ preserved by $\G_P$ which is strictly included in $\O_P$.
}
%

\par{
Suppose every vertex of $P$ is elliptic or parabolic. The parabolic vertices of $P$ are in $\LGP$ by Proposition \ref{propo_para} and the elliptic vertices are in $C(\LGP)$ by Proposition \ref{propo_sphe}. Thereby, the vertices of $P$ are in $\overline{C}(\LGP)$, so $P \cap \O_P  \subset C(\LGP)$. This implies $\O_P \subset C(\LGP)$ by definition of $\O_P$ and so $\O_P=C(\LGP)$. Hence, $\O_P$ is the smallest properly convex open set preserved by $\G_P$.
}
\end{proof}

\subsection{Strict convexity of $\O_P$}$\,$

Here we show Theorem \ref{strict_intro}. The word parabolic can cover different things in geometry. We need to recall some definitions to be precise.

\subsubsection{Parabolic automorphism}$\,$
\par{
An automorphism $\g$ of a properly convex open set $\O$ is \emph{parabolic} when the quantity $\inf_{x \in \O} d_{\O}(x,\g\cdot x)=0$ and the infimum is not achieved.
One can show that such an element has spectral radius 1 and fixes every point of a unique face of $\O$ (see \cite{Cooper:2011fk}).
}
\\
\par{
An isometry $\g$ of a Gromov-hyperbolic space $X$ is \emph{parabolic} when the quantity $\inf_{x \in \O} d_{\O}(x,\g\cdot x)=0$ and the infimum is not achieved. Such an isometry has a unique fixed point in the boundary $\partial X$ of $X$. Every point fixed by a parabolic element of a group $\G$ acting on $\partial X$ is a \emph{parabolic fixed point}.
}
\subsubsection{Projective structure and holonomy}$\,$

A \emph{convex projective manifold} $M$ is a quotient $\Quo$ of a properly convex open set $\O$ by a torsion-free discrete subgroup $\G$ of $\Aut(\O)$. The \emph{holonomy} of an element $\g \in \pi_1(M)$ is the corresponding element in $\G$. We say that an element $\g \in \pi_1(M)$ has \emph{parabolic holonomy} when the corresponding element in $\G$ is parabolic. Every point of $\dO$ fixed by a parabolic element is called a \emph{parabolic fixed point}.

\subsubsection{The notion of relative hyperbolicity}$\,$

\begin{de}
Let $\G$ be a discrete group and $(\P_i)_{i \in I}$ a finite family of subgroups of $\G$. The group $\G$ is \emph{relatively hyperbolic relatively to the family $(\P_i)_{i \in I}$} if and only if there exists a proper Gromov-hyperbolic space $X$ and a geometrically finite action\footnote{See paragraph \ref{best_why} for a definition.
} of $\G$ on $X$ such that the stabilizer of any parabolic fixed point is conjugate to one of the $(\P_i)_{i \in I}$.
\end{de}

\subsubsection{The statement}$\,$

\begin{theorem}
[(compact case by Benoist  \cite{MR2094116}, Cooper, Long and Tillmann \cite{Cooper:2011fk})]\label{convex_gro_hyp}
Let $\G$ be a torsion free discrete subgroup of $\ss$ acting on a properly convex open set $\O$. Suppose the action is of finite covolume, the manifold $\Quo$ is the interior of a compact manifold $N$ with boundary and the holonomy of every component of $\partial N$ is parabolic. Then the following are equivalent:
\begin{enumerate}
\item The metric space $(\O,d_{\O})$ is Gromov-hyperbolic.
\item The properly convex open set $\O$ is strictly convex.
\item The boundary $\dO$ of $\O$ is $\C^1$.
\item The group $\G$ is relatively hyperbolic relatively to the stabilizer of its parabolic fixed points.
\end{enumerate}
\end{theorem}

\begin{rem}\label{quick_rem}
Without any action of a group, one can show that a properly convex open set $\O$ such that $(\O,d_{\O})$ is Gromov-hyperbolic has to be strictly convex (Benoist \cite{MR2094116}) and with $\C^1$-boundary (Karlsson and Noskov \cite{MR1923418}).
\end{rem}

\begin{rem}
One can find avatars of this theorem in the literature, one by Choi in \cite{Choi:2010fk} and the implication $1) \Rightarrow 4)$ by M. Crampon and the author in \cite{Crampon:2012fk} in the context of geometrically finite actions. For the next theorem we will need the version quoted previously.
\end{rem}

Let $P$ be a Coxeter polytope. If $p$ is a parabolic vertex of $P$ we say the subgroup $\G_p$ is a \emph{geometric parabolic subgroup} of $\G_P$.

\begin{theorem}\label{omega_strict}
Let $P$ be a loxodromic Coxeter polytope. The following are equivalent:
\begin{enumerate}
\item The properly convex open set $\O_P$ is strictly convex. 
\item The Coxeter polytope $P$ is 2-perfect and the boundary $\dO_P$ of $\O_P$ is $\C^1$.
\item The Coxeter polytope $P$ is quasi-perfect and the group $\G_P$ is relatively hyperbolic relatively to its geometrical parabolic subgroups.
\end{enumerate}
In that case, the metric space $(\O_P,d_{\O_P})$ is Gromov-hyperbolic and the action is of finite covolume.
\end{theorem}

\begin{proof}
\par{
Suppose we have 3) and let show 1) and 2). Theorem \ref{Theo_vol_fini} shows that the action of $\G_P$ on $\O_P$ is of finite covolume. Since $\G_P$ is of finite type by Selberg's lemma we can find a finite index subgroup $\G'$ of $\G_P$ which is torsion free. The quotient manifold $\Quotient{\O_P}{\G'}$ is of finite volume, it is the interior of a compact manifold $N$ and the holonomy of each component of $\partial N$ is parabolic since $P$ is quasi-perfect. Hence, Theorem \ref{convex_gro_hyp} shows that $(\O_P,d_{\O_P})$ is Gromov-hyperbolic, therefore strictly convex with $\C^1$-boundary by Remark \ref{quick_rem}.
}
\\
\par{
We first show that $P$ has to be $2$-perfect if we assume 1). If $P$ is not 2-perfect then it exists an edge $e$ of $P$ such that the group $W_e$ is infinite (Proposition \ref{def_2-perfect}). This implies $e \subset \dO_P$ by point 5) of Theorem \ref{theo_vinberg}. In particular, $\O_P$ is not strictly convex.
}
\\
\par{
Suppose we have 1) or 2) and $P$ is 2-perfect. First remark that no vertex can be loxodromic from part 3) of Proposition \ref{propo_con3}. Thereby, every vertex of $P$ is either elliptic or parabolic, hence $P$ is quasi-perfect, so Theorem \ref{Theo_vol_fini} shows that $\mu_{\O_P}(P) < \infty$, and we have the first part of the assertion. For the same reason than in the first paragraph of this proof we can use Theorem \ref{convex_gro_hyp} which shows that $\G_P$ is relatively hyperbolic relatively to the stabiliser of its parabolic fixed points (i.e. its geometrical parabolic subgroups).
}
\end{proof}

The following statement is a straightforward corollary of Theorem \ref{omega_strict} which does not use Theorem \ref{convex_gro_hyp}.

\begin{cor}
Let $W$ be a Coxeter group. The Tits convex set $\O_{\Delta_W}$ is strictly convex if and only if $W$ is quasi-Lannér. In particular, in that case, $\O_{\Delta_W}$ is an ellipsoid.
\end{cor}

\begin{proof}
If $W$ is quasi-Lannér, then $\O_{\Delta_W}$ is an ellipsoid and the action of $W$ on $\O_{\Delta_W}$ is of finite covolume. Now, suppose $\O_{\Delta_W}$ is strictly convex. Then Theorem \ref{omega_strict} shows that $\Delta_W$ is quasi-perfect\footnote{We don't need to know that $W$ is relatively hyperbolic since every quasi-Lannér Coxeter group is relatively hyperbolic.}. This means by Remark \ref{Tits_2-perf} that $W$ is quasi-Lannér.
\end{proof}

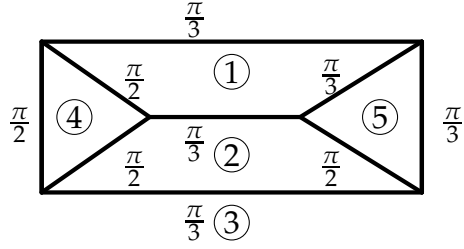
\begin{figure}
\centering
\begin{tikzpicture}[line cap=round,line join=round,>=triangle 45,x=1.0cm,y=1.0cm]
\clip(-0.53,-1.2) rectangle (9.18,3.44);
\draw [line width=1.6pt] (2,2)-- (7,2);
\draw [line width=1.6pt] (7,2)-- (7,0);
\draw [line width=1.6pt] (7,0)-- (2,0);
\draw [line width=1.6pt] (2,0)-- (2,2);
\draw [line width=1.6pt] (2,2)-- (3.42,1);
\draw [line width=1.6pt] (3.42,1)-- (2,0);
\draw [line width=1.6pt] (2,0)-- (2,2);
\draw [line width=1.6pt] (7,2)-- (5.4,1);
\draw [line width=1.6pt] (5.4,1)-- (7,0);
\draw [line width=1.6pt] (7,0)-- (7,2);
\draw [line width=1.6pt] (3.42,1)-- (5.4,1);
\draw (4.5,1.6) node[circle,draw,inner sep=1pt,outer sep=1pt] {$1$};
\draw (4.5,0.5) node[circle,draw,inner sep=1pt,outer sep=1pt] {$2$};
\draw (4.5,-0.4) node[circle,draw,inner sep=1pt,outer sep=1pt] {$3$};
\draw (6.45,1) node[circle,draw,inner sep=1pt,outer sep=1pt] {$5$};
\draw (2.43,1) node[circle,draw,inner sep=1pt,outer sep=1pt] {$4$};
\draw (3.7,2.67) node[anchor=north west] {$\frac{\pi}{3}$};
\draw (3.7,1.02) node[anchor=north west] {$\frac{\pi}{3}$};
\draw (3.7,-0.02) node[anchor=north west] {$\frac{\pi}{3}$};
\draw (5.5,1.91) node[anchor=north west] {$\frac{\pi}{3}$};
\draw (7.1,1.28) node[anchor=north west] {$\frac{\pi}{3}$};
\draw (5.5,0.7) node[anchor=north west] {$\frac{\pi}{2}$};
\draw (1.4,1.28) node[anchor=north west] {$\frac{\pi}{2}$};
\draw (2.91,1.85) node[anchor=north west] {$\frac{\pi}{2}$};
\draw (2.91,0.7) node[anchor=north west] {$\frac{\pi}{2}$};
\end{tikzpicture}
\caption{An indecomposable quasi-divisible prism which gives a non-strictly convex set}
\label{prism}
\end{figure}

\begin{proof}[Proof of Theorem \ref{existenceDim3_intro}]
Consider the prism $\GG$ given by Figure \ref{prism}. The main result of \cite{MR2660566} shows that the space of finite covolume Coxeter prisms $P$ such that the dihedral angle of $P$ are the one given by the label of the edges of $\GG$ is homeomorphic to $\R^*$, so in particular not empty. The group $\G_P$ is not relatively hyperbolic relatively to the unique parabolic vertex (intersection of the faces 1-3-5) because the subgroup generated by $\sigma_1,\sigma_2,\sigma_3$ is virtually $\Z^2$. Theorem \ref{tri} shows that $\G_P$ is strongly irreducible, hence $\O_P$ is indecomposable. Theorem \ref{Theo_vol_fini} shows that $\mu_{\O_P}(P)< \infty$. Theorem \ref{omega_strict} concludes that $\O_P$ is not strictly convex. 
\end{proof}

\subsection{Existence of a strictly convex open set preserved}

We now show Theorem \ref{exis_strict_intro}.

\subsubsection{The statement}

Two standard Coxeter subsystems $T$ and $U$ of $(S,M)$ are \emph{orthogonal} when for every $t \in T$, $u \in U$, $m_{tu}=2$. We denote by $T^{\perp}$ the maximal subsystem orthogonal to $T$.

\begin{theorem}
[(Moussong \cite{MR2636665} hyperbolic case, Caprace \cite{MR2665193,Caprace:fk})]\label{caprace}
For every Coxeter system $(S,M)$, and every collection $\P$ of standard parabolic subgroups of $W_S$, the group $W_S$ is relatively hyperbolic relatively to the $W_T$ for $T \in \P$ if and only if the following three conditions are satisfied:
\begin{enumerate}
\item Each affine sub-system of rank at least $3$ of $(S,M)$ is included in one $T \in \P$. For each pair $S_1,S_2$ of irreducible infinite subsystem which are orthogonal, there exists a $T \in \P$ such that $S_1 \cup S_2 \subset T$.
\item For all $T \neq T' \in \P$, $T \cap T'$ is a spherical Coxeter system.
\item For each $T \in \P$, for each irreducible infinite subsystem $U$ of $T$, we have $U^{\perp} \subset T$
\end{enumerate}
\end{theorem}

We are going to make the abuse of taking together parabolic vertices and the associated Coxeter parabolic subgroup.

\begin{propo}\label{point23}
Let $P$ be a loxodromic 2-perfect Coxeter polytope and $\P$ the set of parabolic vertices of $P$. Then the pair ($W_P,\P)$ satisfies the $2^\textrm{nd}$ and the $3^\textrm{rd}$ points of Theorem \ref{caprace}.
\end{propo}

\begin{proof}
We begin by the second point. Given two vertices $p,q$ of $P$, the segment $[p,q]$ is included in a unique face $f$ of $P$ of minimal dimension and $W_p \cap W_q = W_f$, which is spherical since $\dim(f) \geqslant 1$ and $P$ is 2-perfect. 

For the third point, for each $p \in \P$, the Coxeter group $W_p$ is affine, hence a direct product of irreducible affine Coxeter group $W_{a_1},...,W_{a_r}$. Let $a$ be the union of some $a_i$ and $b$ the union of the others, so that $W_a \times W_b = W_p$.

Suppose there exists a facet $f$ of $P$ in $a^{\perp} \smallsetminus b$. If $s$ is a facet of $P$, then $F_s$ denotes its support and $\A_s$ the affine chart $\S^d \smallsetminus F_s$ not containing $p$ . Let $l$ be the intersection $l=\bigcap_{s \in a} F_s$ (if $r=1$ then $l=\{ p,-p \}$).

The polar $v_f$ of $f$ belongs to $l$. Moreover, $\alpha_f(v_f) =2$ and $\alpha_f(p) <0$ (since $f \notin a \cup b$), hence $v_f \in \A_f \cap l$ (if $r=1$ then we get $v_f=-p$). So, there cannot exist an affine chart containing $P$ and its polars, contradicting Lemma \ref{exis_affi_chart_polar}.
\end{proof}

When $P$ is a 2-perfect Coxeter polytope and $p$ a loxodromic vertex, we will call $\G_p$ a \emph{geometrical loxodromic subgroup} of $\G_P$.

\begin{cor}\label{exist_strict}
Let $P$ be a loxodromic 2-perfect Coxeter polytope whose loxodromic vertices are simple. The following are equivalent:
\begin{enumerate}
\item The convex set $\O_{P^{\dagger}}$ is strictly convex.
\item The boundary of $\O_{P^{\dagger}}$ is $\C^1$.
\item There exists a strictly convex open set $\O'$ preserved by $\G_P$,
\item There exists a properly convex open set $\O'$ with $\C^1$-boundary preserved by $\G_P$,
\item The group $\G_P$ is relatively hyperbolic relatively to its geometric parabolic subgroups.
\item The group $\G_{P^{\dagger}}$ is relatively hyperbolic relatively to its geometric parabolic subgroups.
\end{enumerate}
In this case, the metric space $(\O_{P^{\dagger}},d_{\O_{P^{\dagger}}})$ is Gromov-hyperbolic, hence $\O_{P^{\dagger}}$ is strictly convex with $\C^1$-boundary.
\end{cor}

\begin{rem}
If the group $\G_P$ is relatively hyperbolic relatively to its geometric parabolic subgroups then its loxodromic subgroups are Gromov-hyperbolic since for every ridge $r$ the group $\G_r$ is finite.
\end{rem}

\subsubsection{A lemma about just-infinite subsystems}

\begin{de}
Let $W$ be a Coxeter group given by the Coxeter system $(S,M)$. A subsystem $U$ of $S$ is \emph{just infinite} when the Coxeter group $W_U$ is infinite and for every element $u \in U$, the Coxeter group $W_{U \smallsetminus \{ u \} }$ is finite.
\end{de}

An infinite Coxeter group $W$ always contains a just infinite subsystem. A Coxeter group $W$ is just infinite if and only if $W$ is irreducible affine or Lann\'er.

\begin{de}
Let $P$ be a Coxeter  polytope. Let $U$ be a set of facets of $P$. We say $\U$ \emph{bounds a right angle facet} when there exists a facet $f$ of $P$ such that every ridge of $f$ is also a ridge of a facet of $U$, and all the ridges of $f$ are right angle.
\end{de}

If $U$ bounds a right angle facet $f$ then the projective subspace $\Pi_U$  spanned by the polar of the facets of $U$ is included in the support of $f$.

\begin{de}
Let $P$ be a Coxeter  polytope. Let $U$ be a set of facets of $P$ of cardinal $r$. The projective subspace $\Pi_U$ \emph{meets nicely} $P$ when:
\begin{enumerate}
\item $\Pi_U$ is of dimension $r-1$.
\item  $\Pi_U \cap P \neq \varnothing$.
\item $U$ bounds a right angle facet, or
\item[$\therefore '$] the only facets of $P$ met by $\Pi_U$ are the facets of $U$, and the ridges of $P$ met by $\Pi_U$ are met in their relative interior.
\end{enumerate}
\end{de}

\begin{rem}
Let $P$ be a Coxeter  polytope. Let $U$ be a subsystem of facets of $P$ such that the projective space $\Pi_U$ meets nicely $P$. Then $P \cap \Pi_U$ is a polytope and its facets are in correspondence with $U$, hence $P$ induces a Coxeter structure on $P \cap \Pi_U$, and $P \cap \Pi_U$ is tiling the convex set $\O_P \cap \Pi_U$. Roughly speaking, we find a sub-Coxeter-polytope of $P$.
\end{rem}

\begin{rem}
Let $P$ be an irreducible loxodromic Coxeter polytope. Let $p$ be a vertex of $P$, and $S_p$ be the set of facets containing $p$. We saw at Proposition \ref{propo_con} that the projective space $\Pi_{S_p}$ meets nicely $P$ if $p$ is loxodromic, perfect and simple.
\end{rem}

\begin{rem}
Let $P$ be an irreducible loxodromic Coxeter polytope. Let $U$ be the union of two facets which do not intersect. Then the projective space $l=\Pi_U$ is a line that intersects $P$ nicely thanks to the inequalities $(C)$. Hence, $P \cap l$ is a Coxeter segment which is tiling the segment $\O_P \cap l$. 
\end{rem}

\begin{lemma}\label{lemm_qui_tue}
Let $P$ be an irreducible loxodromic Coxeter polytope. Let $U$ be a just-infinite set of facets of $P$ such that $U \not \subset S_p$ for every parabolic or loxodromic vertex $p$ of $P$. Then the projective space $\Pi_U$ meets nicely $P$, the Coxeter polytope $\Delta = \Pi_U \cap P$ is a simplex, and verifies $A_{\Delta}=A_U$ and $W_{\Delta} = W_U$.

In particular, the group $\G_U$ acts cocompactly on $\Pi_U \cap \O_P$. In particular, $\G_U$ contains a bi-proximal element.
\end{lemma}

\begin{proof}
\par{
First, we show that $\Pi_{U}$ is of dimension the rank $r$ of $U$ minus 1. If $W_U$ is a Lannér Coxeter group then $A_U$ is the Cartan matrix of a perfect loxodromic simplex so of strictly negative determinant hence of full rank qed. If $W_U$ is irreducible affine then either $A_U$ is the Cartan matrix of a perfect loxodromic simplex (and $W_U = \tilde{A}_{r-1}$) and we conclude by the same argument. Otherwise, $A_U$ is the Cartan matrix of a parabolic simplex and Lemma \ref{lemm_qui_tue_de_Vinberg} shows that there exists a vertex $p$ of $P$ such that $U=S_p$. We assume that this is not the case.
}
\par{
We denote by $S$ the set of facets of $P$ and by $T$ the complement $S \smallsetminus U$ of $U$. If $t \in T$, let $F_t$ be the hyperplane spanned by $t$. We denote by $\A_t$ the connected component of $\S^d \smallsetminus F_t$ that contains the interior of $P$. Finally, let $C_T= \bigcap_{t \in T} \overline{\A_t}$. The convex set $C_T$ is not necessarily properly convex. The inequalities $(C)$ show that for every $u \in U$, the polar $v_u \in C_T$
}
\par{
Let $U'$ be any proper subset of $U$. Since $U$ is just-infinite, $U'$ is spherical and Lemma \ref{lemm_qui_tue_de_Vinberg} shows that the intersection $f_{U'} = \bigcap_{u \in \U'} u$ is a face of $P$. The intersection $f_U = \bigcap_{u \in \U} u$ is not a face of $P$ because otherwise we would have $U \subset S_p$ for some vertex $p$ of $P$. So there exists a set $V$ of $d-r+2$ facets of $P$ not in $U$ such that the polytope $Q$ obtained from $P$ by keeping only the facets in $U \cup V$ is a polytope of dimension $d$ with $d+2$ facets. The combinatorics of such a polytope is well-known, they are product of two simplices, or cone over a polytope of dimension $d-1$ with $(d-1)+2$ facets.
}
\par{
The polytope $Q$ is not a cone, since the intersection of any two facets of $U$ is a ridge of $Q$, thanks to Lemma \ref{lemm_qui_tue_de_Vinberg} that can be applied because $U$ is just-infinite. So $Q$ is the product of two simplices. Finally, any proper subset of facets of $U$ intersect to give a face of $Q$, so $Q$ is the product of a $(r-1)$-simplex by a $(d-r+1)$-simplex.
}
\par{
Let $C_V= \bigcap_{t \in V} \overline{\A_t}$. We have $P \cap \Pi_U  \subset C_V$, thanks to the inequalities $(C)$. Now since $Q$ is the product of a $(r-1)$-simplex by a $(d-r+1)$-simplex, we get that $P \cap \Pi_U \neq \varnothing$.
}
\par{
If $U$ bound a right angle facet. Then $\Pi_U$ meets nicely $P$ by definition. Suppose $U$ does not bound a right angle facet. Then  $P \cap \Pi_U$ is included in the interior of $C_T$ and so a facet $f$ of $P$ such that $f \cap \Pi_{U} \neq \varnothing$ is a facet of $U$ and the ridges of $P$ met by $\Pi_U$ are met on their interior. The polytope $P \cap \Pi_U$ is a simplex since Lemma \ref{lemm_qui_tue_de_Vinberg} shows that any proper set of facets of $P \cap \Pi_U$ meets.
}
\end{proof}


\begin{lemma}[(Vinberg, Theorem 7)]\label{lemm_qui_tue_de_Vinberg}
Let $P$ be a Coxeter polytope. Let $(S,M)$ be the Coxeter system associated to $P$ and $W$ the corresponding Coxeter group. Let $S'$ be a subsystem.
\begin{itemize}
\item If $W_{S'}$ is finite, then there exists a face $f$ of $P$ such that $S'=S_f=\{ s\in S \,\mid \, s \supset f \}$.
\item If $A_{S'}$ is the Cartan matrix of a parabolic simplex then there exists a vertex $p$ of $P$ such that $S'=S_p$.
\end{itemize}
\end{lemma}

\subsubsection{The proof of Theorem \ref{exist_strict}}

\begin{proof}[Proof of Theorem \ref{exist_strict}]
\par{
We begin by $6) \Leftrightarrow 1) \Leftrightarrow 2)$. Since $P^{\dagger}$ is quasi-perfect, the conclusion follows from Theorem \ref{omega_strict}. The implication $1) \Rightarrow 3)$ and $2) \Rightarrow 4)$ are obvious since the convex set $\O_{P^{\dagger}}$ is preserved by $\G_P$ and $\G_P \subset \G_{P^{\dagger}}$. Theorem \ref{caprace} shows $5) \Leftrightarrow 6)$.
}
\\
\par{
$\textrm{Not} \, 5) \Rightarrow \textrm{Not} \, 3) \, \textrm{and} \, \textrm{Not} \, 4)$. Let $\O'$ be a properly convex open set preserved by $\G_P$. By Theorem \ref{caprace} and Proposition \ref{point23}, we only have to distinguish the cases, A) there exists a loxodromic vertex $p$ such that $W_p$ is not Gromov-hyperbolic, B) there exists an affine sub-system $U$ of rank at least 3 which is not included in a geometric parabolic or loxodromic subgroup of $\G_P$ and C) there exist two infinite sub-systems $U_1$ and $U_2$ which are orthogonal and $U_1 \cup  U_2$ is not included in a geometric parabolic or loxodromic subgroup of $\G_P$.
}
\\
\par{
Suppose we are in case A). We have to show that $\O'$ is not strictly convex nor with $\C^1$-boundary. Consider the projective space $\Pi_p=\Pi_{S_p}$ where $S_p$ is the set of facets containing $p$. Since $p$ is simple loxodromic, the projective space spanned by the limit set $\Lambda_p$ of $\G_p$ is $\Pi_p$, and we know from Proposition \ref{propo_con} that $\Pi_p$ is of dimension $d-1$. So the convex set $\Pi_p \cap \O'$ is of dimension $d-1$ and the action of $\G_P$ on it is cocompact since $P$ is 2-perfect, hence by Theorem \ref{convex_gro_hyp} (cocompact case) the convex set $\O' \cap \Pi_p$ is not strictly convex nor with $\C^1$-boundary since $\G_p=W_p$ is not Gromov-hyperbolic. Hence, the same is true for $\O'$.
}
\\
\par{
Suppose we are in case B) or C). we claim that in this case, the group $\G_P$ contains a subgroup isomorphic to $\Z^2$ generated by two bi-proximal elements hence Lemma \ref{tri_bound} below shows that $\O_P$ cannot be strictly convex nor with $\C^1$-boundary.
}
\\
\par{
Suppose we are in case B). We can assume $U$ is just-infinite. Since $U \not \subset S_p$ for any loxodromic or parabolic vertex $p$ of $P$, by Lemma \ref{lemm_qui_tue}, the projective space $\Pi_U$ meets nicely $P$, hence $\G_U$ acts cocompactly on $\O_P \cap \Pi_U$. Since, $W_U$ is an irreducible affine Coxeter group, we know that $W_U$ has to be of type $\tilde{A}_n$ with $n \geqslant 2$ and by Proposition \ref{5thomy} $\O_P \cap \Pi_U$ is a simplex. Hence, $\G_U$ contains two bi-proximal elements which generate a $\Z^2$ and Lemma \ref{tri_bound} concludes.
}
\\
\par{
Suppose we are in case C). We can assume $U_1$ and $U_2$ are just-infinite sub-systems. We need to distinguish two cases before concluding. a) $U_1 \subset S_p$ for some vertex $p$ of $P$. In that case, $U_1 \not\subset (S_p)_q$ for any parabolic or loxodromic vertex $q$ of $P_p$ since $P_p$ is perfect. Hence, Lemma \ref{lemm_qui_tue} applied to $P_p$ shows $\G_{U_1}$ contains a bi-proximal element for its action on $\S_p^{d-1}$. But the eigenvalue at $p$ for any element of $\G_p$ is one, hence $\G_{U_1}$ as a subgroup of $\ss$, has a bi-proximal element. b) $U_1 \not \subset S_p$, for any vertex $p$ of $P$. Then Lemma \ref{lemm_qui_tue} applied to $P$ shows $\G_{U_1}$ contains a bi-proximal element. 
}
\\
\par{ 
So in any situation, the groups $\G_{U_1}$ and  $\G_{U_2}$ contain a bi-proximal element. Since these two groups commute, we get that $\G_P$ contains two elements which are bi-proximal and generate a $\Z^2$. Lemma \ref{tri_bound} concludes.
}
\end{proof}

\begin{lemma}\label{tri_bound}
Let $\O$ be a properly convex open set. Suppose $\Aut(\O)$ contains two bi-proximal elements $\g,\delta$ which generate a $\Z^2$. Then $\O$ is not strictly convex nor with $\C^1$-boundary.
\end{lemma}
\vspace{-7pt}

\begin{proof}
\par{
Let $p^+_{\g},p^-_{\g},p^+_{\delta},p^-_{\delta}$ be the attractive and repulsive fixed points of $\g$ and $\delta$. Let $\G$ be the group generated by $\g$ and $\delta$. We claim that the set $F=\{ p^+_{\g},p^-_{\g},p^+_{\delta},p^-_{\delta} \}$ is of cardinality 3. Indeed, if $F$ is of cardinality $2$ then the group $\G$ acts properly on the segment joining the two points of $F$ include in $\O$, hence $\G$ is cyclic. If $F$ is of cardinality $4$, then a ping-pong argument shows $\G$ contains a free subgroup of rank $2$.
}
\par{
We call $p^0$ the point $p^+_{\delta}$ or $p^-_{\delta}$ different from $p^+_{\g},p^-_{\g}$. Hence, the plane $\Pi$ generated by $p^0,p^+_{\delta},p^-_{\delta}$ is preserved by $\g$ and we are in a dimension 2 situation. It is then easy to see that the segments $[p^0,p^+_{\delta}]$ and $[p^0,p^-_{\delta}]$ are included in $\dO \cap \Pi$. Thereby, $\O$ is not strictly convex nor with $\C^1$-boundary. See \cite{Marquis:2009kq} for more details.
}
\end{proof}

\begin{rem}
If we do not assume that the loxodromic vertices are simple then the statements 1) , 2) and 6) of Theorem \ref{exist_strict} do not make sense any more. But, we still have 3) or 4) $\Rightarrow$ 5). But, I don't know how to build a strictly convex invariant open set (or with $\C^1$ boundary) assuming 5).
\end{rem}


\bibliographystyle{alpha}

\end{document}